\documentclass[A4paper,oneside,11pt]{article}
\usepackage{amssymb,amsthm,amsmath,marvosym,amsfonts,fontenc}
\usepackage[normalem]{ulem}
\usepackage[usenames,dvipsnames]{color}
\usepackage{enumerate,marginnote}
\usepackage{mdwlist}
\usepackage{subfigure}
\usepackage{pifont}
\usepackage{bm}
\usepackage{multind}\makeindex{general}\makeindex{mathsymbols}
\usepackage{caption}
\usepackage{fullpage} 
\usepackage{epsfig,psfrag}
\usepackage{fancyhdr} 
\usepackage{xr}
\numberwithin{equation}{section}
\numberwithin{table}{section} 
\numberwithin{figure}{section}
\newtheorem{theorem}{Theorem}[section]
\newtheorem{subtheorem}{Theorem}[theorem] 
\newtheorem{subsubtheorem}{Theorem}[subtheorem]
\newtheorem{lem}[theorem]{Lemma}
\newtheorem{lemma}[theorem]{Lemma}
\newtheorem{conjecture}[theorem]{Conjecture}

\newtheorem{fact}[theorem]{Fact}
\newtheorem{remark}[theorem]{Remark}
\newtheorem{definition}[theorem]{Definition}
\newtheorem{setting}[theorem]{Setting}

\theoremstyle{remark}
\newtheorem{claim}[subtheorem]{Claim}

\newtheorem{subclaim}[subsubtheorem]{Subclaim}

\newcommand{\By}[2]{\overset{\mbox{\tiny{#1}}}{#2}}
\newcommand{\ByRef}[2]{   \By{\eqref{#1}}{#2} }

\newcommand{\leBy}[1]{    \By{#1}{\le} }
\newcommand{\geBy}[1]{    \By{#1}{\ge} }

\newcommand{\gByRef}[1]{  \ByRef{#1}{>} }
\newcommand{\leByRef}[1]{ \ByRef{#1}{\le} }
\newcommand{\geByRef}[1]{ \ByRef{#1}{\ge} }

\newcommand{\Referee}[1]{}
\newcommand{\BUG}[1]{#1}

\let\sm\setminus
\let\subset\subseteq 
\let\supset\supseteq 
\let\epsilon\varepsilon
\let\sharp\#
\def\dcup{\dot\cup} 
\renewcommand{\leq}{\leqslant}
\renewcommand{\le}{\leqslant}
\renewcommand{\geq}{\geqslant}
\renewcommand{\ge}{\geqslant}

\let\oldmarginpar\marginpar
\renewcommand\marginpar[1]{\-\oldmarginpar[\raggedleft\footnotesize #1]%
{\raggedright\footnotesize #1}}

\newcommand{\HIDDENPROOF}[1]{
}

\newcommand{\HIDDENTEXT}[1]{
}

\title{The Approximate Loebl--Koml\'os--S\'os Conjecture III:\\
The finer structure of LKS graphs} 
\author{Jan
Hladk\'y
\thanks{\emph{Corresponding author.} Institute of Mathematics, Academy of Science of the Czech Republic. \v Zitn\'a 25, 110 00, Praha, Czech Republic. The Institute of Mathematics of the Academy of Sciences of the Czech Republic is supported by RVO:67985840. Email:
\texttt{honzahladky@gmail.com}.
The research leading to these results has received funding from the People Programme (Marie Curie Actions) of the European Union's Seventh Framework Programme (FP7/2007-2013) under REA grant agreement umber 628974. Much of the work was done while supported by an EPSRC postdoctoral fellowship while affiliated with DIMAP and Mathematics Institute, University of
Warwick.}
\quad 
J\'anos Koml\'os\thanks{Department of Mathematics, Rutgers University, 110 Frelinghuysen Rd., Piscataway, NJ~08854-8019, USA} 
\quad 
Diana Piguet\thanks{Institute of Computer Science, Czech Academy of Sciences, Pod Vod\'arenskou v\v e\v z\'i 2, 182~07 Prague, Czech Republic. With institutional support RVO:67985807. Supported by the Marie Curie fellowship FIST, DFG grant TA 309/2-1,  Czech Ministry of Education project 1M0545, EPSRC award EP/D063191/1, and EPSRC Additional Sponsorship EP/J501414/1.
	The research leading to these results has received funding from the European Union Seventh
	Framework Programme (FP7/2007-2013) under grant agreement no. PIEF-GA-2009-253925.
    The work leading to this invention was supported by the European Regional Development Fund (ERDF), project ``NTIS -- New Technologies for Information Society'', European Centre of Excellence, CZ.1.05/1.1.00/02.0090. Partially supported by the Czech Science Foundation, grant number GJ16-07822Y.}
    \\ 
    Mikl\'os Simonovits\thanks{R\'enyi
    Institute, Budapest, Hungary. Supported by OTKA~78439, OTKA~101536, ERC-AdG.~321104} 
\quad 
Maya Stein\thanks{Department of Mathematical Engineering,
University of Chile, Santiago, Chile.  Supported by Fondecyt Iniciacion grant 11090141, Fondecyt Regular grant 1140766 and CMM Basal.}
\quad 
Endre Szemer\'edi\thanks{R\'enyi
	Institute, Budapest, Hungary. Supported by OTKA~104483 and ERC-AdG.~321104}}

\def\semiregular{regularized }
\def\semiregulars{regularized}

\def\kknnaagg{hub }
\def\kknnaaggss{hubs }

\def\kknnaaggssNOSPACE{hubs}

\newcommand{\PARAMETERPASSING}[2]{{\mathrm{#1}\ref{#2}}}

\def\NN{\mathbb{N}}

\newcommand{\M}{\mathcal M}\newcommand{\V}{\mathcal V}
\newcommand{\eps}{\epsilon}

\def\probability{\mathbf{P}}
\def\expectation{\mathbf{E}}

\def\mindeg{\mathrm{mindeg}}
\def\maxdeg{\mathrm{maxdeg}}
\def\density{\mathrm{d}}
\def\neighbour{\mathrm{N}}

\def\shadow{\mathrm{shadow}}
\newcommand{\treeclass}[1]{\mathbf{trees}({#1})}

\newcommand{\LKSgraphs}[3]{\mathbf{LKS}({#1},{#2},{#3})}
\newcommand{\LKSmingraphs}[3]{\mathbf{LKSmin}({#1},{#2},{#3})}
\newcommand{\LKSsmallgraphs}[3]{\mathbf{LKSsmall}({#1},{#2},{#3})}

\newcommand{\smallvertices}[3]{\mathbb{S}_{{#1},{#2}}({#3})}
\newcommand{\largevertices}[3]{\mathbb{L}_{{#1},{#2}}({#3})}

\newcommand{\JUSTIFY}[1]{\mbox{\tiny{(#1)}}\quad}

\def\Gcapt{G_\nabla}
\def\GD{G_{\mathcal{D}}}
\def\Gblack{G_{\mathrm{reg}}}
\def\Gexp{G_{\mathrm{exp}}}
\def\BGblack{\mathbf{G}_{\mathrm{reg}}}
\def\smallatoms{\mathbb{E}}
\def\clusters{\mathbf{V}}
\def\class{\nabla}
\def\HugeVertices{\mathbb{H}}
\def\DenseSpots{\mathcal{D}}

\def\YA{\mathbb{YA}}
\def\YB{\mathbb{YB}}
\def\WantiC{V_{\leadsto\HugeVertices}}

\def\NUP{\mathrm{N}^{\uparrow}}
\def\NDOWN{\mathrm{N}^{\downarrow}}
\def\Vgood{V_\mathrm{good}}

\def\exceptVertSplit{\bar V}
\def\exceptSemSplit{\bar \V}
\def\exceptClustSplit{\bar \clusters}

\def\gC{\mathcal{C}}
\def\gP{\mathbb{J}}
\def\gPatoms{\mathbb{J}_{\smallatoms}}

\def\shrubA{\mathcal S_{A}}
\def\shrubB{\mathcal S_{B}}

\def\XA{\mathbb{XA}}
\def\XB{\mathbb{XB}}
\def\XC{\mathbb{XC}}

\def\Mgood{\M_{\mathrm{good}}}
\def\NAtom{{\mathcal N_{\smallatoms}}}

\newcommand{\colouringp}[1]{\mathbb{A}_{#1}}
\newcommand{\colouringpI}[1]{^{\restriction{#1}}}
\def\colouringpartition{(\colouringp{0},\colouringp{1},\colouringp{2})}
\newcommand{\proporce}[1]{\mathfrak{p}_{#1}}
\def\shadowsplit{\mathbb{F}}

\def\largeintoatoms{V_{\leadsto\smallatoms}}
\def\clustersintoatoms{\clusters_{\leadsto\smallatoms}}

\def\clustersize{\mathfrak{c}}

\def\LargeTen{\mathcal L^*}

\def\epsilonD{\pi}
\def\alphaD{\widehat{\alpha}}

\def\AXA{\mathbb{A}}
\def\aXa{\mathfrak{q}}
\def\VXV{\mathbb{B}}

\renewcommand{\today}{}
\date{}

\begin{document}
\pagenumbering{roman}
\maketitle
\begin{abstract}
This is the third of a series  of four papers in which
we prove the following relaxation of the
Loebl--Koml\'os--S\'os Conjecture: For every~$\alpha>0$
there exists a number~$k_0$ such that for every~$k>k_0$
 every $n$-vertex graph~$G$ with at least~$(\frac12+\alpha)n$ vertices
of degree at least~$(1+\alpha)k$ contains each tree $T$ of order~$k$ as a
subgraph. 

In the first paper of the series, we gave a
decomposition of the graph~$G$  into several parts of different characteristics. In the second paper, we found a combinatorial structure inside the decomposition. In this paper, we will give a refinement of this structure. In the forthcoming fourth paper, the refined structure will be used for embedding the tree $T$.
\end{abstract}

\bigskip\noindent
{\bf Mathematics Subject Classification: } 05C35 (primary), 05C05 (secondary).\\
{\bf Keywords: }extremal graph theory; Loebl--Koml\'os--S\'os Conjecture; tree embedding; regularity lemma; sparse graph; graph decomposition.

\newpage

\rhead{\today}

\newpage

\tableofcontents
\newpage
\pagenumbering{arabic}
\setcounter{page}{1}

\section{Introduction}\label{sec:intro2}

This is the third of a series of four papers~\cite{cite:LKS-cut0, cite:LKS-cut1, cite:LKS-cut2, cite:LKS-cut3} 
in which we provide an approximate solution of the Loebl--Koml\'os--S\'os Conjecture. The conjecture reads as follows.
 
\begin{conjecture}[Loebl--Koml\'os--S\'os Conjecture 1995~\cite{EFLS95}]\label{conj:LKS}
Suppose that $G$ is an $n$-vertex graph with at least $n/2$ vertices of degree more than $k-2$. Then $G$ contains each tree of order $k$.
\end{conjecture}

We discuss the history and state of the art in detail  in the first paper~\cite{cite:LKS-cut0} of our series. The main result, which will be proved in~\cite{cite:LKS-cut3}, is the 
approximate solution of the Loebl--Koml\'os--S\'os Conjecture, namely the following.

\begin{theorem}[Main result~\cite{cite:LKS-cut3}]\label{thm:main}
For every $\alpha>0$ there exists  $k_0$ such that for any
$k>k_0$ we have the following. Each $n$-vertex graph $G$ with at least
$(\frac12+\alpha)n$ vertices of degree at least $(1+\alpha)k$ contains each tree $T$ of
order $k$.
\end{theorem}
In the first paper~\cite{cite:LKS-cut0} we exposed the decomposition techniques, finding a \emph{sparse decomposition} of the host graph $G$. The sparse decomposition should be thought of as a counterpart to the Szemer\'edi regularity lemma (but compared to the Szemer\'edi regularity lemma the sparse decomposition seems to be less versatile). In the second paper~\cite{cite:LKS-cut1}, we combined the sparse decomposition with a matching structure, obtaining in \cite[Lemma~\ref{p1.prop:LKSstruct}]{cite:LKS-cut1} what we call \emph{the rough structure}. The rough structure obtained in~\cite[Lemma~\ref{p1.prop:LKSstruct}]{cite:LKS-cut1} depends on the graph $G$ only, i.e., is independent of the tree $T$. The rough structure encodes the general information how~$T$ should be embedded on a macroscopic scale. However, from the perspective of embedding small parts of~$T$ locally, the properties of the rough structure are insufficient. In the present paper we take the preparation of the host graph one step further, refining the rough structure. This way we obtain one of ten possible \emph{configurations}. Formally, each of the configuration --- denoted by $\mathbf{(\diamond1)}$--$\mathbf{(\diamond10)}$\Referee{(1)} --- is a collection of favourable properties the said graph must satisfy. Each of these configurations is based on the building blocks of the sparse decomposition, and describes in a very fine way a substructure in~$G$. Some of the configurations involve some basic parameters of the tree $T$. That is, while the presence of some individual configurations (namely, configuration $\mathbf{(\diamond1)}$--$\mathbf{(\diamond5)}$ and $\mathbf{(\diamond10)}$ introduced in Section~\ref{sec:configurations}) suffices for embedding of each $k$-vertex tree, configurations  $\mathbf{(\diamond6)}$--$\mathbf{(\diamond9)}$ are accompanied by parameters (denoted by $h$, $h_1$ and $h_2$ in Definitions~\ref{def:CONF6}--\ref{def:CONF9}) that depend on certain parameters of the tree $T$.

In the last paper~\cite{cite:LKS-cut3} we will prove that each of these ten configurations allows to embed~$T$. 
This will complete the proof of Theorem~\ref{thm:main}. An overview of how the embedding goes for each individual configuration is given in~\cite[Section~\ref{p3.ssec:embeddingOverview}]{cite:LKS-cut3}. We recommend the reader to consult this part of~\cite{cite:LKS-cut3} in parallel when reading through the definitions of the configurations in Section~\ref{sec:typesofconfigurations}.

\medskip

\begin{figure}[ht]
	\centering 
	\includegraphics{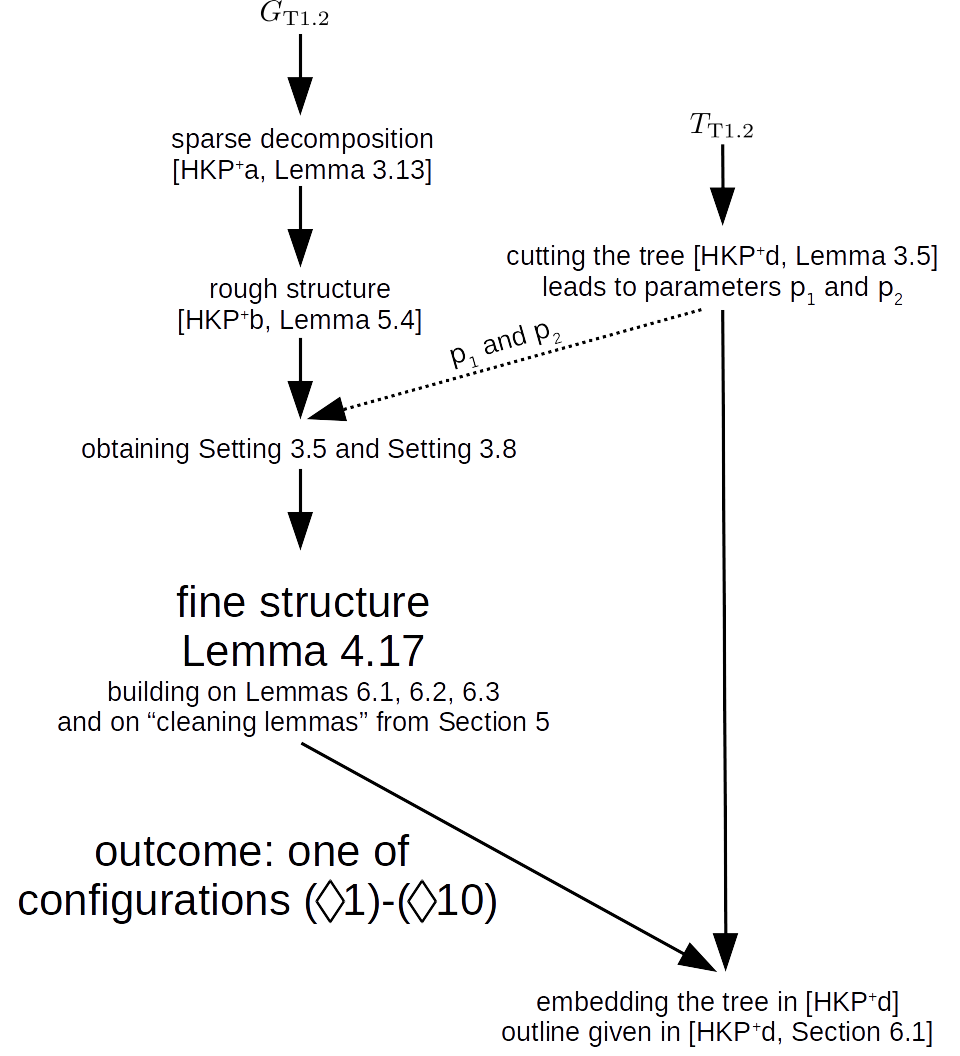}
	\caption{Diagram of the proof of Theorem~\ref{thm:main} with focus on the part dealt with in this paper.}
	\label{fig:LKSCUT2struct}
\end{figure} 
The paper is organized as follows. In Section~\ref{sec:notation} we introduce some basic notation. In Section~\ref{sec:configurations} we introduce some further auxiliary notions, and two ``settings'' that will be common to the rest of the paper.
In Section~\ref{sec:typesofconfigurations}, we present the 
main result of this paper, Lemma~\ref{outerlemma}. The lemma says that in any graph that satisfies the conditions of Theorem~\ref{thm:main}, we can find at least one of the ten configurations described above. To prove it, we first introduce some preliminary ``cleaning lemmas'' in Section~\ref{ssec:cleaning}. The proof of Lemma~\ref{outerlemma} then occupies Section~\ref{ssec:obtainingConf}. This is illustrated in Figure~\ref{fig:LKSCUT2struct}.

\section{Notation, basic facts, and bits from other papers in the series}\label{sec:notation}

\subsection{General notation}
The set $\{1,2,\ldots, n\}$ of the first $n$ positive integers is
denoted by \index{mathsymbols}{*@$[n]$}$[n]$. 
We frequently employ indexing by many indices. We write
superscript indices in parentheses (such as $a^{(3)}$), as
opposed to notation of powers (such as $a^3$).
We use sometimes subscript to refer to
parameters appearing in a fact/lemma/theorem. For example
$\alpha_\PARAMETERPASSING{T}{thm:main}$ refers to the parameter $\alpha$ from Theorem~\ref{thm:main}.
We omit rounding symbols when this does not affect the
correctness of the arguments.

Table~\ref{tab:notation} shows the system of notation we use in this paper and in~\cite{cite:LKS-cut0,cite:LKS-cut1,cite:LKS-cut3}.
\begin{table}[h]
\centering
\caption{Specific notation used in the series.}
\label{tab:notation}
\begin{tabular}{r|l}
\hline
lower case Greek letters &  small positive constants ($\ll 1$)\\
                         & $\phi$ reserved for embedding; $\phi:V(T)\rightarrow V(G)$\\
\hline
upper case Greek letters & large positive constants ($\gg 1$)\\
\hline
one-letter bold& sets of clusters \\
\hline
bold (e.g., $\treeclass{k},\LKSgraphs{n}{k}{\eta}$)& classes of graphs\\
\hline
blackboard bold (e.g., $\HugeVertices,\smallatoms,\smallvertices{\eta}{k}{G},\XA$)& distinguished vertex sets except for\\
& $\NN$ which denotes the set $\{1,2,\ldots\}$\\
\hline 
script  (e.g., $\mathcal A,\mathcal D,\mathcal N$)& families (of vertex sets, ``dense spots'', and regular pairs)\\
\hline
$\class$(=nabla)&sparse decomposition (see Definition~\ref{sparseclassdef})\\
\hline
\end{tabular}
\end{table}

We write \index{mathsymbols}{*VG@$V(G)$}$V(G)$ and \index{mathsymbols}{*EG@$E(G)$}$E(G)$ for the vertex set and edge set of a graph $G$, respectively. Further, \index{mathsymbols}{*VG@$v(G)$}$v(G)=|V(G)|$ is the order of $G$, and \index{mathsymbols}{*EG@$e(G)$}$e(G)=|E(G)|$ is its number of edges. If $X,Y\subset V(G)$ are two, not necessarily disjoint, sets of vertices we write \index{mathsymbols}{*EX@$e(X)$}$e(X)$ for the number of edges induced by $X$, and \index{mathsymbols}{*EXY@$e(X,Y)$}$e(X,Y)$ for the number of ordered pairs $(x,y)\in X\times Y$ such that $xy\in E(G)$. In particular, note that $2e(X)=e(X,X)$.

\index{mathsymbols}{*DEG@$\deg$}\index{mathsymbols}{*DEGmin@$\mindeg$}\index{mathsymbols}{*DEGmax@$\maxdeg$}
For a graph $G$, a vertex $v\in V(G)$ and a set $U\subset V(G)$, we write
$\deg(v)$ and $\deg(v,U)$ for the degree of $v$, and for the number of
neighbours of $v$ in $U$, respectively. We write $\mindeg(G)$ for the minimum
degree of $G$, $\mindeg(U):=\min\{\deg(u)\::\: u\in U\}$, and
$\mindeg(V_1,V_2)=\min\{\deg(u,V_2)\::\:u\in V_1\}$ for two sets $V_1,V_2\subset
V(G)$. Similar notation is used for the maximum degree, denoted by $\maxdeg(G)$.
The neighbourhood of a vertex $v$ is denoted by
\index{mathsymbols}{*N@$\neighbour(v)$}$\neighbour(v)$, and we write $\neighbour(U)=\bigcup_{u\in
U}\neighbour(u)$. These symbols have a subscript to emphasize the host graph.

The symbol ``$-$'' is used
for two graph operations: if $U\subset V(G)$ is a vertex
set then $G-U$ is the subgraph of $G$ induced by 
$V(G)\setminus U$. If $H\subset G$ is a subgraph of $G$ then the graph
$G-H$ is defined on the vertex set $V(G)$ and corresponds
to deletion of edges of $H$ from $G$.



A family $\mathcal A$ of pairwise disjoint subsets of $V(G)$ is an \index{general}{ensemble}\index{mathsymbols}{*ENSEMBLE@$\ell$-ensemble}\emph{$\ell$-ensemble in $G$} if  $|A|\ge \ell$ for each $A\in\mathcal A$. 


\subsection{Regular pairs}

We now define regular pairs in the sense of
Szemer\'edi's regularity lemma.
Given a graph $H$ and a pair $(U,W)$ of disjoint
sets $U,W\subset V(H)$ the
\index{general}{density}\index{mathsymbols}{*D@$\density(U,W)$}\emph{density of the pair $(U,W)$} is defined as
$$\density(U,W):=\frac{e(U,W)}{|U||W|}\;.$$
Similarly, for a bipartite graph $G$ with colour
classes $U$, $W$ we talk about its \index{general}{bipartite
density}\emph{bipartite density}\index{mathsymbols}{*D@$\density(G)$} $\density(G)=\frac{e(G)}{|U||W|}$.
For a given $\varepsilon>0$, a pair $(U,W)$ of disjoint
sets $U,W\subset V(H)$ 
is called an \index{general}{regular pair}\emph{$\epsilon$-regular
pair} if $|\density(U,W)-\density(U',W')|<\epsilon$ for every
$U'\subset U$, $W'\subset W$ with $|U'|\ge \epsilon |U|$, $|W'|\ge
\epsilon |W|$. If the pair $(U,W)$ is not $\epsilon$-regular,
then it is called \index{general}{irregular}\emph{$\epsilon$-irregular}. A stronger notion than
regularity is that of super-regularity which we recall now. A pair $(A,B)$ is
\index{general}{super-regular pair}\emph{$(\epsilon,\gamma)$-super-regular} if it is
$\epsilon$-regular, and both $\mindeg(A,B)\ge\gamma |B|$, and
$\mindeg(B,A)\ge \gamma |A|$. Note that then $(A,B)$ has bipartite density at least $\gamma$.


The following facts are well known.

\begin{fact}\label{fact:BigSubpairsInRegularPairs}
Suppose that $(U,W)$ is an $\varepsilon$-regular pair of density
$d$. Let $U'\subset W, W'\subset W$ be sets of vertices with $|U'|\ge
\alpha|U|$, $|W'|\ge \alpha|W|$, where $\alpha>\epsilon$.
Then the pair $(U',W')$ is a $2\varepsilon/\alpha$-regular pair of density at least
$d-\varepsilon$.
\end{fact}
\begin{fact}\label{fact:manyTypicalVertices}
Suppose that $(U,W)$ is an $\varepsilon$-regular pair of density
$d$. Then all but at most $\epsilon|U|$ vertices $v\in U$ satisfy
$\deg(v,W)\ge (d-\epsilon)|W|$.
\end{fact}
%
%
%
%

The next lemma asserts that if we have many  $\epsilon$-regular pairs $(R,Q_i)$, then most vertices in $R$ have approximately the total degree into the set
$\bigcup_i Q_i$ that we would expect.
\begin{lemma}\label{lem:degreeIntoManyPairs}
Let $Q_1,\ldots,Q_\ell$ and $R$ be disjoint vertex sets. Suppose
further that  for each
$i\in[\ell]$, the pair $(R,Q_i)$ is $\epsilon$-regular. Then we have
\begin{enumerate}[(a)]
\item \label{mindestensdengrad}
$ \deg (v,\bigcup_i Q_i)\geq \frac{e(R,\bigcup_i Q_i)}{|R|}-\epsilon\left|\bigcup_i
Q_i\right|$ for all but at most $\epsilon|R|$ vertices $v\in R$, and
\item \label{hoechstensdengrad} 
$ \deg (v,\bigcup_i Q_i)\le
\frac{e\left(R,\bigcup_i Q_i\right)}{|R|}+\epsilon\left|\bigcup_i
Q_i\right|$ for all but at most $\epsilon|R|$ vertices $v\in R$.
\end{enumerate}
\end{lemma}
\begin{proof}
We prove \eqref{mindestensdengrad}, the proof of~\eqref{hoechstensdengrad} is similar.
Suppose for contradiction that~\eqref{mindestensdengrad} does not hold.
Without loss of generality, assume that there is a set $X\subset
R$, $|X|>\epsilon |R|$ such that $\frac{e(R,\bigcup
Q_i)}{|R|}-\epsilon|\bigcup Q_i|> \deg(v,\bigcup Q_i)$ for
each $v\in X$. By averaging, there is an index $i\in [\ell]$
such that $\frac{|X|}{|R|}e(R,Q_i)-\epsilon|X||Q_i|>
e(X,Q_i)$, or equivalently, $\density(R,Q_i)-\epsilon>
\density(X,Q_i).$ This contradicts the $\epsilon$-regularity of the pair $(R,Q_i)$.
\end{proof}

\subsection{LKS graphs}
We now give some notation specific to our setting. We write \index{mathsymbols}{*trees@$\treeclass{k}$}
$\treeclass{k}$ for the set of all trees (up to isomorphism) of order $k$. We write 
\index{mathsymbols}{*LKSgraphs@$\LKSgraphs{n}{k}{\eta}$}$\LKSgraphs{n}{k}{\alpha}$
for the class of all $n$-vertex graphs with at least
$(\frac12+\alpha)n$ vertices of degrees at least
$(1+\alpha)k$. With this notation Conjecture~\ref{conj:LKS} states that every graph in $\LKSgraphs{n}{k-1}{0}$ contains every tree from $\treeclass{k}$.

Given a graph $G$, denote by
\index{mathsymbols}{*S@$\smallvertices{\eta}{k}{G}$}$\smallvertices{\eta}{k}{G}$ the set of those
vertices of $G$ that have degree less than $(1+\eta)k$ and by
\index{mathsymbols}{*L@$\largevertices{\eta}{k}{G}$}$\largevertices{\eta}{k}{G}$ the set of those
vertices of $G$ that have degree at least $(1+\eta)k$.

In~\cite{cite:LKS-cut0} we introduced the class $\LKSmingraphs{n}{k}{\eta}$ of the graphs that are edge-minimal with respect to membership to $\LKSgraphs{n}{k}{\eta}$. It would be sufficient to prove Theorem~\ref{thm:main} for graphs in $\LKSmingraphs{n}{k}{\eta}$. This class, however, is too rigid with respect changes that are necessary when applying the sparse decomposition. Therefore, in~\cite[Section~\ref{p0.ssec:LKSgraphs}]{cite:LKS-cut0} we derived a relaxation of the class $\LKSmingraphs{n}{k}{\eta}$ which we introduce next. 
\begin{definition}\label{def:LKSsmall}
Let \index{mathsymbols}{*LKSsmallgraphs@$\LKSsmallgraphs{n}{k}{\eta}$}$\LKSsmallgraphs{n}{k}{\eta}$ be the class of graphs $G\in\LKSgraphs{n}{k}{\eta}$ having the following three properties:
\begin{enumerate}
   \item All the neighbours of every vertex $v\in V(G)$ with $\deg(v)>\lceil(1+2\eta)k\rceil$ have degrees at most $\lceil(1+2\eta)k\rceil$.\label{def:LKSsmallA}
   \item All the neighbours of every vertex of $\smallvertices{\eta}{k}{G}$
    have degree exactly $\lceil(1+\eta)k\rceil$. \label{def:LKSsmallB}
   \item We have $e(G)\le kn$.\label{def:LKSsmallC}
\end{enumerate}
\end{definition}

\subsection{Sparse decomposition}\label{sec:FromPaper1}
Here we recall some definitions from~\cite{cite:LKS-cut0}: dense spots, avoiding sets, and the key notions of bounded and sparse decomposition. This section is a rather dry list for later reference only, and the reader should consult \cite[Section~\ref{p0.sec:class}]{cite:LKS-cut0} for a more detailed description of these notions. Here, we just recall that the purpose for introducing dense spots, avoiding sets, nowhere-dense graph is that together with high degree vertices they form a sparse decomposition of a given graph. The main result of the first paper in the series, \cite[Lemma~\ref{p0.lem:LKSsparseClass}]{cite:LKS-cut0}, asserts that each graph from $\LKSgraphs{n}{k}{\eta}$ has a sparse decomposition in which almost all edges are of one of the above types. (In fact, the sparse decomposition is not specific to LKS graphs, and indeed in \cite[Lemma~\ref{p0.lem:genericBD}]{cite:LKS-cut0} we provide a corresponding general statement.)
\begin{definition}[\bf \index{general}{dense spot}$(m,\gamma)$-dense spot,
\index{general}{nowhere-dense}$(m,\gamma)$-nowhere-dense]\label{def:densespot} 
Suppose that $m\in\NN$ and $\gamma>0$.
An \emph{$(m,\gamma)$-dense spot} in a graph $G$ is a non-empty bipartite sub\-graph  $D=(U,W;F)$ of  $G$ with
$\density(D)>\gamma$ and $\mindeg (D)>m$. We call a graph $G$
\emph{$(m,\gamma)$-nowhere-dense} if it does not contain any $(m,\gamma)$-dense spot.

When the parameters $m$ and $\gamma$ are not relevant, we call $D$ simply a \emph{dense spot}.
\end{definition}
Note that dense spots do not have a 
specified orientation. That is, we view $(U,W;F)$ and $(W,U;F)$ as
the same object.

\begin{definition}[\bf $(m,\gamma)$-dense
cover]\index{general}{dense cover}
Suppose that $m\in\NN$ and $\gamma>0$.
An \emph{$(m,\gamma)$-dense cover} of a given
graph $G$ is a family $\DenseSpots$ of edge-disjoint
$(m,\gamma)$-dense
spots such that $E(G)=\bigcup_{D\in\DenseSpots}E(D)$.
\end{definition}

The following two facts are proved in~\cite[Facts~\ref{p0.fact:sizedensespot},~\ref{p0.fact:boundedlymanyspots}]{cite:LKS-cut0}.
\begin{fact}\label{fact:sizedensespot}
Let $(U,W;F)$ be a $(\gamma k,\gamma)$-dense spot in a
graph $G$ of maximum degree at most $\Omega k$. Then
$\max\{|U|,|W|\}\le \frac{\Omega}{\gamma}k.$
\end{fact}

\begin{fact}\label{fact:boundedlymanyspots}
Let $H$ be a graph of maximum degree at most $\Omega k$, let $v\in V(H)$, and let $\DenseSpots$ be a family of edge-disjoint $(\gamma k,\gamma)$-dense spots. Then less than $\frac{\Omega}{\gamma}$ dense spots from $\DenseSpots$ contain $v$.
\end{fact}

We now define the avoiding set. Informally, a set $\smallatoms$ of vertices is avoiding if for each set $U$ of size at most $\Lambda k$ (where $\Lambda\gg 1$ is a large constant) and for each vertex $v\in\smallatoms$ there is a dense spot containing $v$ and almost disjoint from $U$. Favourable properties of avoiding sets for embedding trees are shown in~\cite[Section~\ref{p0.sssec:whyavoiding}]{cite:LKS-cut0}.
\begin{definition}[\bf
\index{general}{avoiding}$(\Lambda,\epsilon,\gamma,k)$-avoiding set]\label{def:avoiding} 
Suppose that $\epsilon,\gamma>0$, $\Lambda>0$, and $k\in\NN$. Suppose that
$G$ is a graph and $\DenseSpots$ is a family of dense spots in $G$. A set
\index{mathsymbols}{*E@$\smallatoms$}$\smallatoms\subset \bigcup_{D\in\DenseSpots} V(D)$ is \emph{$(\Lambda,\epsilon,\gamma,k)$-avoiding} with
respect to $\DenseSpots$ if for every $U\subset V(G)$ with $|U|\le \Lambda k$ the following holds for all but at most $\epsilon k$ vertices $v\in\smallatoms$. There is a dense spot $D\in\DenseSpots$ with $|U\cap V(D)|\le \gamma^2 k$ that contains $v$.
\end{definition}

Finally, we can introduce the most important tool in the proof of Theorem~\ref{thm:main}, the \emph{sparse decomposition}. It generalises the notion of equitable partition from Szemer\'edi's regularity lemma. The first step towards this end is the notion of bounded decomposition. An illustration is given in Figure~\ref{fig:sparsedecomposition}.
\begin{figure}[t]
     \centering
     \subfigure[Bounded decomposition]{
          \includegraphics[width=.4\textwidth]{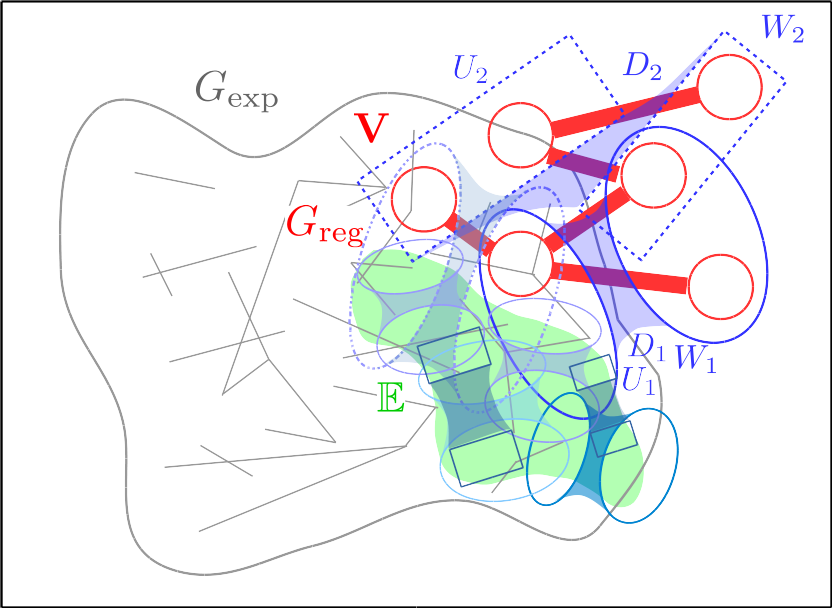}}
     \hspace{.07in}
     \subfigure[Sparse decomposition]{
          \includegraphics[width=.4\textwidth]{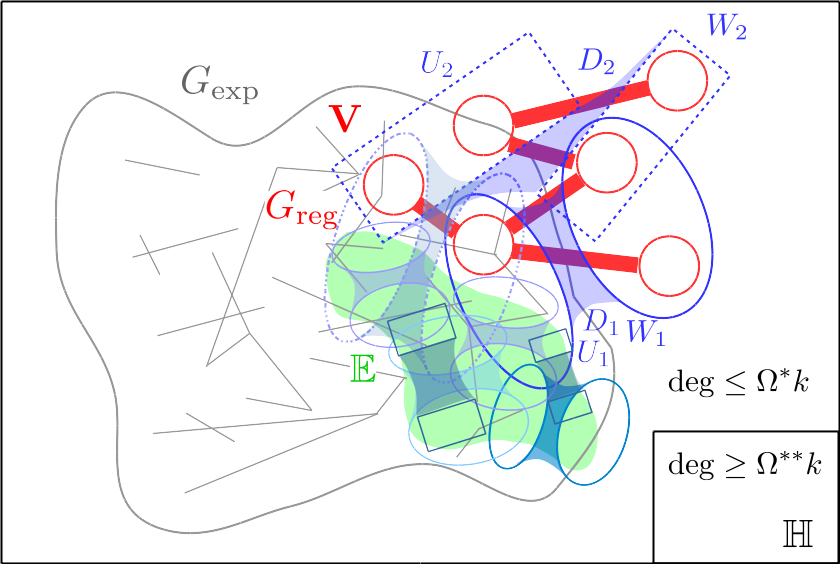}}
     \hspace{.07in}
     \caption{A simplified illustration of a bounded/sparse decomposition of a graph. The nowhere-dense graph $\Gexp$ shown in grey, the cluster graph $\Gblack$ and clusters $\clusters$ shown in red, the avoiding set $\smallatoms$ in green, and the dense spots $\DenseSpots$ in blue (different shades and shapes). The difference between the bounded and the sparse decomposition is that no distinction regarding degrees of vertices is made in the former.}
\label{fig:sparsedecomposition}
\end{figure}

\begin{definition}[\index{general}{bounded decomposition}{\bf
$(k,\Lambda,\gamma,\epsilon,\nu,\rho)$-bounded decomposition}]\label{bclassdef}
Let $\mathcal V=\{V_1, V_2,\ldots, V_s\}$ be a partition of the vertex set of a graph $G$. We say that $( \clusters,\DenseSpots, \Gblack, \Gexp,
\smallatoms )$ is a {\em $(k,\Lambda,\gamma,\epsilon,\nu,\rho)$-bounded
decomposition} of $G$ with respect to $\mathcal V$ if the following properties
are satisfied:
\begin{enumerate}
\item\label{defBC:nowheredense}
$\Gexp$ is  a $(\gamma k,\gamma)$-nowhere-dense subgraph of $G$ with $\mindeg(\Gexp)>\rho k$.
\item\label{defBC:clusters} $\clusters$ consists of disjoint subsets of 
$ V(G)$.
\item\label{defBC:RL} $\Gblack$ is a subgraph of $G-\Gexp$ on the vertex set $\bigcup \clusters$. For each edge
 $xy\in E(\Gblack)$ there are distinct $C_x\ni x$ and $C_y\ni y$ from $\clusters$,
and  $G[C_x,C_y]=\Gblack[C_x,C_y]$. Furthermore, 
$G[C_x,C_y]$ forms an $\epsilon$-regular pair of  density at least $\gamma^2$.
\item We have $\nu k\le |C|=|C'|\le \epsilon k$ for all
$C,C'\in\clusters$.\label{Csize}
\item\label{defBC:densepairs}  $\DenseSpots$ is a family of edge-disjoint $(\gamma
k,\gamma)$-dense spots  in $G-\Gexp$.  For
each $D=(U,W;F)\in\DenseSpots$ all the edges of $G[U,W]$ are covered
by $\DenseSpots$ (but not necessarily by $D$).
\item\label{defBC:dveapul} If  $\Gblack$
contains at least one edge between $C_1\in\clusters$ and $C_2\in\clusters$ then there exists a dense
spot $D=(U,W;F)\in\DenseSpots$ such that $C_1\subset U$ and $C_2\subset
W$.
\item\label{defBC:prepartition}
For
each $C\in\clusters$ there is $V\in\mathcal V$ so that either $C\subseteq V\cap V(\Gexp)$ or $C\subseteq V\setminus V(\Gexp)$.
For
each $C\in\clusters$ and each $D=(U,W; F)\in\DenseSpots$ we have that either $C$ is disjoint from $D$ or contained in $D$.
\item\label{defBC:avoiding}
$\smallatoms$ is a $(\Lambda,\epsilon,\gamma,k)$-avoiding subset  of
$V(G)\setminus \bigcup \clusters$ with respect to dense spots $\DenseSpots$.
\end{enumerate}

\smallskip
We say that the bounded decomposition $(\clusters,\DenseSpots, \Gblack, \Gexp,
\smallatoms )$ {\em respects the avoiding threshold~$b$}\index{general}{avoiding threshold} if for each $C\in \clusters$ we either have $\maxdeg_G(C,\smallatoms)\le b$, or $\mindeg_G(C,\smallatoms)> b$.
\end{definition}

The members of $\clusters$ are called \index{general}{cluster}{\it clusters}. Define the
{\it cluster graph} \index{mathsymbols}{*Gblack@$\BGblack$}  $\BGblack$ as the graph
on the vertex set $\clusters$ that has an edge $C_1C_2$
for each pair $(C_1,C_2)$ which has density at least $\gamma^2$ in the graph
$\Gblack$. 

We can now introduce the notion of sparse decomposition in which we enhance a bounded decomposition  by distinguishing between vertices of huge and moderate degree.
\begin{definition}[\bf \index{general}{sparse
decomposition}$(k,\Omega^{**},\Omega^*,\Lambda,\gamma,\epsilon,\nu,\rho)$-sparse decomposition]\label{sparseclassdef}
Suppose that $k\in\NN$ and $\epsilon,\gamma,\nu,\rho>0$ and $\Lambda,\Omega^*,\Omega^{**}>0$. 
Let $\mathcal V=\{V_1, V_2,\ldots, V_s\}$ be a partition of the vertex set of a graph $G$. We say that 
$\class=(\HugeVertices, \clusters,\DenseSpots, \Gblack, \Gexp, \smallatoms )$
is a
{\em $(k,\Omega^{**},\Omega^*,\Lambda,\gamma,\epsilon,\nu,\rho)$-sparse decomposition} of $G$
with respect to $V_1, V_2,\ldots, V_s$ if the following holds.
\begin{enumerate}
\item\label{def:classgap}
\index{mathsymbols}{*H@$\HugeVertices$} $\HugeVertices\subset V(G)$,
$\mindeg_G(\HugeVertices)\ge\Omega^{**}k$,
$\maxdeg_K(V(G)\setminus \HugeVertices)\le\Omega^{*}k$, where $K$ is spanned by the edges of $\bigcup\DenseSpots$, $\Gexp$, and
edges incident with $\HugeVertices$,
\item \label{def:spaclahastobeboucla} $( \clusters,\DenseSpots, \Gblack,\Gexp,\smallatoms)$ is a 
$(k,\Lambda,\gamma,\epsilon,\nu,\rho)$-bounded decomposition of
$G-\HugeVertices$ with respect to $V_1\setminus \HugeVertices, V_2\setminus \HugeVertices,\ldots, V_s\setminus \HugeVertices$.
\end{enumerate}
\end{definition}

If the parameters do not matter, we call $\class$ simply a {\em sparse
decomposition}, and similarly we speak about a {\em bounded decomposition}.

 \begin{definition}[\bf captured edges, graphs $\Gcapt$ and $\GD$]\label{capturededgesdef}
\index{general}{captured edges}
\Referee{(3)}
In the situation of Definition~\ref{sparseclassdef}, we define the graph
\index{mathsymbols}{*GD@$\GD$}$\GD$ as the graph induced by the dense spots, i.e., $V(\GD)=\bigcup_{D\in\DenseSpots}V(D)$, $E(\GD)=\bigcup_{D\in\DenseSpots}E(D)$.

We refer to the edges in
$ E(\Gblack)\cup E(\Gexp)\cup
E_G(\HugeVertices,V(G))\cup E_{\GD}(\smallatoms,\smallatoms\cup \bigcup \clusters)$
as \index{general}{captured edges}{\em captured} by the sparse decomposition. 
 We write
\index{mathsymbols}{*Gclass@$\Gcapt$}$\Gcapt$ for the subgraph of $G$ on the same vertex set which consists of the captured edges.

Likewise, the \emph{captured edges} of a bounded decomposition
$(\clusters,\DenseSpots, \Gblack,\Gexp,\smallatoms )$ of a graph $G$ are those
in $E(\Gblack)\cup E(\Gexp)\cup E_{\GD}(\smallatoms,\smallatoms\cup\bigcup\clusters)$.
\end{definition}

\subsection{Regularized matchings}
We recall the notion of a \semiregular matching, introduced in~\cite{cite:LKS-cut1}.\footnote{In older versions of~\cite{cite:LKS-cut1,cite:LKS-cut3} (available on the arXiv) and in the published version of~\cite{LKS:overview} we used the name of ``semiregular matchings''.}
\begin{definition}[\bf $(\epsilon,d,\ell)$-\semiregular matching]\label{def:semiregular} 
	Suppose that $\ell\in\NN$ and  $d,\epsilon>0$.
	A collection
	$\mathcal N$ of pairs $(A,B)$ with $A,B\subset V(H)$ is called an
	\index{general}{regularized@\semiregular matching}\emph{$(\epsilon,d,\ell)$-\semiregular matching} of a
	graph $H$ if \begin{enumerate}[(i)] \item $|A|=|B|\ge \ell$ for each
		$(A,B)\in\mathcal N$, \item $(A,B)$ induces in $H$ an $\epsilon$-regular pair of
		density at least $d$, for each $(A,B)\in\mathcal N$, and \item the sets $\{A\}_{(A,B)\in\mathcal N}$ and $\{B\}_{(A,B)\in\mathcal N}$ are  pairwise disjoint.
	\end{enumerate}
	Sometimes, when the parameters do not matter we simply write  \emph{\semiregular matching}.
\end{definition}
Suppose that $\mathcal N$ is a \semiregular matching, and $(A,B)\in\mathcal N$. Then we call $A$ a \index{general}{partner}\emph{partner} of $B$, and $B$ a partner of $A$ (in $\mathcal N$).

We shall make use of some auxiliary results from~\cite{cite:LKS-cut1}. To this end, we need a definition.
\begin{definition}[{\cite[Definition~\ref{p1.tupelclass}]{cite:LKS-cut1}}]\label{tupelclass}
	We define $\mathcal G(n,k,\Omega,\rho,\nu, \tau)$\index{mathsymbols}{*G@$\mathcal G(n,k,\Omega,\rho,\nu, \tau)$} to be the class of all tuples $(G,\DenseSpots,H,\mathcal A)$ with the following properties:
	\begin{enumerate}[(i)]
		\item  $G$ is a graph of order $n$ with $\maxdeg(G)\le \Omega k$,\label{maxroach}
		\item $H$ is a bipartite subgraph of
		$G$ with colour classes $A_H$ and $B_H$ and with $e(H)\ge \tau kn$,\label{duke}
		\item  $\DenseSpots$ is a $(\rho k, \rho)$-dense cover of $G$,
		\item $\mathcal A$ is a $(\nu k)$-ensemble in $G$,
		and $A_H\subseteq \bigcup \mathcal A$,\label{bird}
		\item  $A\cap U\in\{\emptyset,A\}$ for each $A\in\mathcal A$ and for each $D=(U,W;F)\in\DenseSpots$.\label{mingus}
	\end{enumerate}
\end{definition}

\begin{lemma}[{\cite[Lemma~\ref{p1.lem:edgesEmanatingFromDensePairsIII}]{cite:LKS-cut1}}]\label{lem:edgesEmanatingFromDensePairsIII}
	For every $\bar\Omega\in\NN$ and $\bar\rho,\bar\epsilon,\bar\tau \in(0,1)$ there exists
	an $\bar\alpha>0$ such that for every $\bar\nu\in (0,1)$ there is a number $\bar k_0\in\NN$
	such that the following holds for every $k>\bar k_0$. 
	
	For each $(\bar G,\bar\DenseSpots,\bar H,\bar{\mathcal A})\in\mathcal G(n,k,\bar\Omega,\bar\rho,\bar\nu,\bar\tau)$ there exists an
	$(\bar\epsilon ,\frac{\bar\tau\bar\rho}{8\Omega},\bar\alpha\bar\nu k)$-\semiregular matching $\bar{\mathcal M}$ of $\bar G$ such that
	\begin{enumerate}[(1)]
		\item for each $(X,Y)\in\bar{\mathcal M}$ there are $A\in\bar{\mathcal A}â$, and
		$D=(U,W;F)\in \bar\DenseSpots$ such that  $X\subset U\cap A\cap A_H$ and $Y\subset
		W\cap B_H$,\label{bedingung1}
		and
		\item $|V(\bar{\mathcal M})|\ge\frac{\bar\tau}{2\bar\Omega} n$.\label{bedingung3}
	\end{enumerate}
\end{lemma}

\subsection{Cutting trees}\label{sec:TREEcut}
We outline the way we process any $k$-vertex tree $T$ in our proof of Theorem~\ref{thm:main}. This is done in detail in \cite[Section~\ref{p3.ssec:cut}]{cite:LKS-cut3}. The purpose of the informal description below is only to serve as a reference when we motivate the configurations in Section~\ref{ssec:TypesConf}.

Given $T$, we introduce a constant number (i.e., independent of $k$) of \index{general}{cut-vertex}\emph{cut-vertices} $W\subset V(T)$. We can do so in such a way that the following properties are satisfied:\footnote{Here, we list only properties that are relevant for the description later. See~\cite[Definition~\ref{p3.ellfine} and Lemma~\ref{p3.lem:TreePartition}]{cite:LKS-cut3} for details.}
\begin{itemize}
\item The set $W$ is partitioned into sets \index{mathsymbols}{*WA@$W_A$}\index{mathsymbols}{*WB@$W_B$}$W_A\dcup W_B$ such that the distance between each vertex of $W_A$ and each vertex of $W_B$ is odd.
\item The trees of $T-W$, which are called \index{general}{shrub}\emph{shrubs}, are all small, i.e., of order $O(\frac{k}{|W|})$. Each shrub either neighbours one vertex of $W$ (in which case it is called an \index{general}{end shrub}\emph{end shrub}) or two vertices of $W$  (in which case it is called an \index{general}{internal shrub}\emph{internal shrub}).
\item The two neighbours in $W$ of each internal shrub are from $W_A$.
\item The components of $T[W]$ are referred to as \index{general}{\kknnaagg}\emph{\kknnaaggssNOSPACE}.
\item The shrubs that neighbour a vertex (or two vertices) of $W_A$ are denoted \index{mathsymbols}{*SA@$\shrubA$}$\shrubA$. The shrubs that neighbour a vertex of $W_B$ are denoted \index{mathsymbols}{*SB@$\shrubB$}$\shrubB$.
\end{itemize}
We call the quadruple $(W_A,W_B,\shrubA,\shrubB)$ a \index{general}{fine partition}\emph{fine partition} of $T$.

\section{Shadows, random splitting, and common settings}\label{sec:configurations}
In this section we will prove some preliminaries needed for the main results of this paper, presented in Section~\ref{sec:typesofconfigurations}.
The present section is organized as follows. In Section~\ref{ssec:shadows} we introduce an auxiliary notion of shadows and prove some simple properties. Section~\ref{ssec:RandomSplittins}
introduces randomized splitting of the vertex set of an input
graph. In Section~\ref{ssec:ExceptionalVertices} we 
 introduce building blocks for
the finer structure we will obtain in Section~\ref{sec:typesofconfigurations}.

\subsection{Shadows}\label{ssec:shadows}
We will find it convenient to work with the notion of a shadow. To motivate this notion, we recall the greedy embedding strategy. Suppose that $T$ is a tree of order $k$ and $G$ is a graph with minimum degree at least $k-1$. We can then root $T$ at an arbitrary vertex. Then, we embed that vertex in $G$. Now, at each step, we have a partial embedding of $T$ in $G$. We pick one vertex of $T$ that is already embedded but whose children are yet unembedded, and we embed those in $T$. The minimum degree condition tells us that we can always accommodate these children.

The greedy embedding strategy clearly fails in the setting of Theorem~\ref{thm:main}. So, we need to enhance the strategy by not embedding the vertices of $T_\PARAMETERPASSING{T}{thm:main}$ in some part $U$ (which is not suitable for continuing the embedding) of $G_\PARAMETERPASSING{T}{thm:main}$. This forces us to look-ahead: when embedding a vertex $v$ of $T_\PARAMETERPASSING{T}{thm:main}$ we want not only to avoid $U$, but also vertices that send many edges to $U$, since we want to avoid $U$ also with the children of $v$. The notion of shadow formalizes this.

Given a graph $H$, a set
$U\subset V(H)$, and a number $\ell$ we define inductively
\index{mathsymbols}{*SHADOW@$\shadow$}
\begin{align*}
& \shadow^{(0)}_H(U,\ell):=U\text{, and }\\
& \shadow^{(i)}_H(U,\ell):=\{ v\in V(H)\::\:
\deg_H(v,\shadow^{(i-1)}_H(U,\ell))>\ell \} \text{ for }i\geq 1.
\end{align*}
We abbreviate $\shadow^{(1)}_H(U,\ell)$ as $\shadow_H(U,\ell)$. Further, the
graph $H$ is omitted from the subscript if it is clear from the context.
Note that the shadow of a set $U$ might intersect $U$.

Below, we state two facts which bound the size of a shadow of a given set.
Fact~\ref{fact:shadowbound} gives a bound for general graphs of bounded maximum
degree and Fact~\ref{fact:shadowboundEXPANDER} gives a stronger bound for
nowhere-dense graphs.
\begin{fact}\label{fact:shadowbound}
Suppose $H$ is a graph with $\maxdeg(H)\le \Omega k$. Then for each
$\alpha>0, i\in\{0,1,\ldots\}$, and each set $U\subset V(H)$, we have
$$|\shadow^{(i)}(U,\alpha k)|\le \left(\frac{\Omega}{\alpha}\right)^i|U|\;.$$
\end{fact}
\begin{proof}
Proceeding by induction on $i$ it suffices to show that
$|\shadow^{(1)}(U,\alpha k)|\le \Omega|U|/\alpha$. To this end, observe that
$U$ sends out at most $\Omega k |U|$ edges while each vertex of
$\shadow(U,\alpha k)$ receives at least $\alpha k$ edges from $U$.
\end{proof}
\begin{fact}\label{fact:shadowboundEXPANDER}
Let $\alpha,\gamma,Q>0$ be three numbers such that $1\le Q\le
\frac{\alpha}{16\gamma}$. Suppose that $H$ is a $(\gamma k,\gamma)$-nowhere-dense graph, and let $U\subset V(H)$ with $|U|\le Qk$. Then we have $$|\shadow(U,\alpha
k)|\le \frac{16Q^2\gamma}{\alpha}k.$$
\end{fact}

\begin{proof}
Suppose the contrary and let $W\subset \shadow(U,\alpha k)$ be of size
$|W|=\frac{16Q^2\gamma}{\alpha}k\le Qk$.\Referee{(2)} Then $e_H(U\cup W)\ge \frac12 \sum_{v\in W}\deg_H(v,U)\ge 8\gamma Q^2 k^2$. Thus $H[U\cup W]$ has average degree at least $$\frac{2 e_H(U\cup W)}{|U|+|W|}\ge 8\gamma Qk\;,$$ and therefore, by a well-known fact, contains a subgraph $H'$ of minimum degree at least $4\gamma Qk$. Taking a maximal cut $(A,B)$ in $H'$, it is easy to see that $H'[A,B]$ has minimum degree at least $2\gamma Qk\ge \gamma k$. Further, $H'[A,B]$ has density at least $\frac{|A|\cdot 2\gamma Qk}{|A||B|}\geq\gamma$,  contradicting that $H$ is $(\gamma k,
\gamma)$-nowhere-dense.
\end{proof}

\subsection{Random splitting}\label{ssec:RandomSplittins}
Suppose a graph $G$ (together with its bounded decomposition)
 is given. In this section we split its vertex
set into several classes the sizes of which have given ratios. It is important that most vertices will
have their degrees split obeying approximately these ratios. The
corresponding statement is given in Lemma~\ref{lem:randomSplit}. It will be used to split the vertices of the host graph $G=G_\PARAMETERPASSING{T}{thm:main}$
according to which part of the tree
$T=T_\PARAMETERPASSING{T}{thm:main}\in\treeclass{k}$ they will host. More
precisely, suppose that $(W_A,W_B,\shrubA,\shrubB)$ is a fine partition
of $T$. Let $t_\mathrm{int}$ and $t_\mathrm{end}$
be the total sizes of the internal and end shrubs, respectively. We then
want to partition $V(G)$ into three sets
$\colouringp{0},\colouringp{1},\colouringp{2}$ in the ratio (approximately)
$$(|W_A|+|W_B|)\::\:t_\mathrm{int}\::\:t_\mathrm{end}$$ so that degrees of the
vertices of $V(G)$ are split proportionally. This will allow us to embed the vertices of
$W_A\cup W_B$ into $\colouringp{0}$, the internal shrubs into $\colouringp{1}$, and
end shrubs into $\colouringp{2}$. Actually, since our embedding procedure is
more complex, we not only require the degrees to be split proportionally, but
also to partition proportionally the objects from the bounded decomposition.
In~\cite{cite:LKS-cut3} it will get clearer why such a random
splitting needs to be used.

Lemma~\ref{lem:randomSplit} below is formulated in an abstract setting, without
any reference to the tree $T$, and with a general number of classes in the partition.

\begin{lemma}\label{lem:randomSplit} For each $p\in\NN$ and $a>0$ there exists $k_0>0$ such that
for each $k>k_0$ we have the following.

Suppose $G$ is a graph of order $n\ge k_0$ and $\maxdeg(G)\le\Omega^*k$ with its 
$(k,\Lambda,\gamma,\epsilon,k^{-0.05},\rho)$-bounded
decomposition $( \clusters,\DenseSpots, \Gblack, \Gexp, \smallatoms )$. As
usual, we write $\Gcapt$ for the subgraph captured by $( \clusters,\DenseSpots,
\Gblack, \Gexp, \smallatoms )$, and $\GD$ for the spanning subgraph of $G$
consisting of the edges in $\DenseSpots$. Let $\M$ be an
$(\epsilon,d,k^{0.95})$-\semiregular matching in $G$, and $\VXV_1,\ldots, \VXV_p$
be subsets of $V(G)$. Suppose that $\Omega^*\ge 1$ and $\Omega^*/\gamma<k^{0.1}$.

Suppose that $\aXa_1,\ldots,\aXa_p\in\{0\}\cup[a,1]$ are reals with $\sum \aXa_i\le 1$.
Then there exists a partition $\AXA_1\cup \ldots\cup \AXA_p=V(G)$, and sets $\bar
V\subset V(G)$, $\bar\V\subset \V(\M)$, $\bar\clusters\subset\clusters$ with the following properties.
\begin{enumerate}[(1)]
  \item\label{It:H1} $|\bar V|\le \exp(-k^{0.1})n$,
  $|\bigcup\bar\V|\le \exp(-k^{0.1})n$,
  $|\bigcup\bar\clusters|<\exp(-k^{0.1})n$.
  \item\label{It:H2} For each $i\in [p]$ and each $C\in
  \clusters\setminus\bar\clusters$ we have $|C\cap \AXA_i|\ge
  \aXa_i|\AXA_i|-k^{0.9}$.
  \item\label{It:H3} For each $i\in [p]$ and each $C\in \V(\M)\setminus\bar\V$
  we have $|C\cap \AXA_i|\ge \aXa_i|\AXA_i|-k^{0.9}$.
  \item\label{It:H4} For each $i\in [p]$, $D=(U,W; F)\in\DenseSpots$
   and 
   $\mindeg_D (U\setminus \bar V,W\cap \AXA_i)\ge \aXa_i\gamma k-k^{0.9}$.
  \item\label{It:HproportionalSizes} For each $i,j\in[p]$ we have $|\AXA_i\cap
  \VXV_j|\ge \aXa_i|\VXV_j|-n^{0.9}$.
  \item\label{It:H5} For each $i\in [p]$ each $J\subset[p]$ and each $v\in V(G)\setminus \bar V$ we have $$\deg_H(v,\AXA_i\cap \VXV_J)\ge \aXa_i\deg_H(v,\VXV_J)-2^{-p}k^{0.9}\;,$$ for each graph $H\in\{G,\Gcapt,\Gexp,\GD,\Gcapt\cup\GD\}$,
  where 
$\VXV_J:=\big(\bigcap_{j\in J}\VXV_j\big)\sm \big(\bigcup_{j\in [p]\setminus J}
\VXV_j\big)$.
  \item\label{It:H6} For each $i,i',j,j'\in [p]$ ($j\neq j'$), we have
\begin{align*}
e_H(\AXA_i\cap \VXV_{j},\AXA_{i'}\cap \VXV_{j'})&\ge
  \aXa_i\aXa_{i'} e_H(\VXV_j,\VXV_{j'})-k^{0.6}n^{0.6}\;,\\
e_H(\AXA_i\cap \VXV_{j},\AXA_{i'}\cap \VXV_{j})&\ge
  \aXa_i\aXa_{i'} e(H[\VXV_j])-k^{0.6}n^{0.6}\qquad\mbox{if $i\neq i'$, and}\\
  e(H[\AXA_i\cap \VXV_{j}])&\ge
  \aXa_i^2 e(H[\VXV_j])-k^{0.6}n^{0.6}\;.   
\end{align*}  
for each graph
  $H\in\{G,\Gcapt,\Gexp,\GD,\Gcapt\cup\GD\}$.
  \item\label{It:H7} For each $i\in[p]$ if $\aXa_i=0$ then
  $\AXA_i=\emptyset$.
\end{enumerate}
\end{lemma}
\begin{proof}
We can assume that $\sum \aXa_i=1$ since all bounds in \eqref{It:H2}--\eqref{It:H6} are lower bounds. Assume that $k$ is
large enough. We assign each vertex $v\in V(G)$ to one of the sets $\AXA_1$,
\ldots, $\AXA_p$ at random with respective probabilities $\aXa_1,\ldots,\aXa_p$. Let $\bar V_1$ and $\bar V_2$ be the
vertices which do not satisfy~\eqref{It:H4} and~\eqref{It:H5},
respectively. Let $\bar\V$ be the sets of $\V(\M)$ which do
not satisfy~\eqref{It:H3}, and let $\bar\clusters$ be the
clusters of $\clusters$ which do not satisfy~\eqref{It:H2}.
Setting~$\bar V:=\bar V_1\cup\bar V_2$, we need to show
that~\eqref{It:H1}, \eqref{It:HproportionalSizes} and~\eqref{It:H6} are
fulfilled simultaneously with positive probability. Using the union bound,
it suffices to show that each of the properties~\eqref{It:H1},
\eqref{It:HproportionalSizes} and \eqref{It:H6} is violated with probability
at most $0.2$. The probability of each of these three
properties can be controlled in a
straightforward way by the Chernoff bound. We only give such a bound (with error probability at most $0.1$) on the size of the set $\bar V_1$ (appearing
in~\eqref{It:H1}), which is the most difficult one to control.

For $i\in [p]$, let $\bar V_{1,i}$ be the set of vertices $v$ for which there
exists $D=(U,W; F)\in\DenseSpots$, $U\ni v$, such that
$\deg_D(v,W\cap \AXA_i)<\aXa_i\gamma k-k^{0.9}$. We aim to show
that for each $i\in [p]$ the probability that $|\bar
V_{1,i}|>\exp(-k^{0.2})n$ is at most $\frac{1}{10p}$. Indeed, summing such an error bound together with similar bounds for other properties will allow us to conclude with the statement.
This will in turn follow from the Markov Inequality provided that we show that
\begin{equation}\label{itsraining}
\expectation[|\bar V_{1,i}|]\le
\frac{1}{10p}\cdot\exp(-k^{0.2})n\;.
\end{equation}
Indeed, let us consider an arbitrary vertex $v\in V(G)$. By Fact~\ref{fact:boundedlymanyspots},
$v$ is contained in at most $\Omega^*/\gamma$ dense spots of $\DenseSpots$. For a 
fixed dense spot $D=(U,W;F)\in\DenseSpots$ with $v\in
U$ let us bound the probability of the event $\mathcal
E_{v,i,D}$ that $\deg_D(v,W\cap \AXA_i)<\aXa_i\gamma k-k^{0.9}$.
To this end, fix a set $N\subset W\cap \neighbour_D(v)$ of size exactly
$\gamma k$ before the random assignment is performed. Now, elements of
$V(G)$ are distributed randomly into the sets $\AXA_1,\ldots, \AXA_p$. In particular, the number
$|\AXA_i\cap N|$ has binomial distribution with parameters
$\gamma k$ and $\aXa_i$. Using the Chernoff bound, we get
$$\probability[\mathcal E_{v,i,D}]\le
\probability\left[|\AXA_i\cap N|<\aXa_i\gamma k-k^{0.9}\right]\le
\exp(-k^{0.3}) \;.$$
Thus, it follows by summing
the tail over at most $\Omega^*/\gamma\le k^{0.1}$ dense
spots containing $v$, that
\begin{equation}\label{itsunny}
\probability[v\in \bar V_{1,i}]\le k^{0.1}\cdot
\exp(-k^{0.3})\;.
\end{equation}
Now,~\eqref{itsraining} follows by linearity of expectation.
\end{proof}

Lemma~\ref{lem:randomSplit} is utilized for the purpose of our proof of
Theorem~\ref{thm:main} using the notion of proportional partition introduced in
Definition~\ref{def:proportionalsplitting} below.

\subsection{Common settings}\label{ssec:ExceptionalVertices}
Throughout Section~\ref{sec:configurations} 
 we shall
be working with the setting that comes from \cite[Lemma~\ref{p1.prop:LKSstruct}]{cite:LKS-cut1}. In order
to keep statements of the subsequent lemmas reasonably short we introduce a common setting.

Suppose that $G$ is a graph with a $(k,\Omega^{**},\Omega^*,\Lambda,\gamma,\epsilon,\nu,\rho)$-sparse decomposition 
$$\class=(\HugeVertices, \clusters,\DenseSpots, \Gblack, \Gexp,\smallatoms
)\;$$ with respect to $(\largevertices{\eta}{k}{G},\smallvertices{\eta}{k}{G})$.\Referee{(4)}
 Suppose further that  $\mathcal M_A,\mathcal M_B$ are
$(\epsilon',d,\gamma k)$-\semiregular matchings  in $\GD$. We then define the triple \index{mathsymbols}{*XA@$\XA(\eta,\class, \mathcal M_A,\mathcal M_B)$} \index{mathsymbols}{*XB@$\XB(\eta,\class, \mathcal M_A,\mathcal M_B)$}
\index{mathsymbols}{*XC@$\XC(\eta,\class, \mathcal M_A,\mathcal
M_B)$}
$(\XA,\XB,\XC)=(\XA,\XB,\XC)(\eta,\class, \mathcal M_A,\mathcal M_B)$
by setting
\begin{align*}
\XA&:=\largevertices{\eta}{k}{G}\setminus V(\mathcal M_B)\;,\\
\XB&:=\left\{v\in V(\mathcal M_B)\cap \largevertices{\eta}{k}{G}\::\:\widehat{\deg}(v)<(1+\eta)\frac
k2\right\}\;,\\
\XC&:=\largevertices{\eta}{k}{G}\setminus(\XA\cup\XB)\;,
\end{align*}
where 
$\widehat{\deg}(v)$ on the second
line is defined by
\begin{equation}\label{eq:defhatdeg}
\widehat{\deg}(v):=\deg_G\big(v,\smallvertices{\eta}{k}{G}\setminus
(V(\Gexp)\cup\smallatoms\cup V(\mathcal M_A\cup\mathcal M_B)\big)\;.
\end{equation}
\begin{remark}\label{rem:aboutXAXB}
The sets $\XA,\XB,\XC$ were defined in~\cite[Definition~\ref{p1.def:XAXBXC}]{cite:LKS-cut1}. Of course, in applications, the matchings $\mathcal M_A$ and $\mathcal M_B$ will be guaranteed to have some favourable properties. These properties are formulated in~\cite[Lemma~\ref{p1.prop:LKSstruct}]{cite:LKS-cut1} and are listed in \eqref{commonsetting1}--\eqref{commonsettingNicDoNAtom} of Setting~\ref{commonsetting} below. It was argued in \cite[Section~\ref{p1.ssec:motivation}]{cite:LKS-cut1} why then the set $\XA$ has excellent properties for accommodating cut-vertices of $T_\PARAMETERPASSING{T}{thm:main}$, and the set $\XB$ has ``half-that-excellent properties'' for accommodating cut-vertices. In particular, the formula defining $\XB$ suggests that we cannot make use of the set $\smallvertices{\eta}{k}{G}\setminus
(V(\Gexp)\cup\smallatoms\cup V(\mathcal M_A\cup\mathcal M_B)$ for the purpose of embedding shrubs neighbouring the  cut-vertices embedded into $\XB$.
\end{remark}

With this notation, we can introduce the common setting, Setting~\ref{commonsetting}. Setting~\ref{commonsetting} serves as an interface between what has been done in~\cite{cite:LKS-cut0,cite:LKS-cut1} and what will be needed in~\cite{cite:LKS-cut3}. Thus, where possible, we interlace the (highly technical) definitions of Setting~\ref{commonsetting} with some motivation and references.
\begin{setting}\label{commonsetting}
We assume that the constants $\Lambda,\Omega^*,\Omega^{**},k_0$ and $\alphaD,\gamma,\epsilon,\epsilon',\eta,\pi,\rho, \tau, d$ satisfy
\begin{align}
\label{eq:KONST}
\begin{split}
 \frac12\ge\eta\gg\frac1{\Omega^*}\gg \frac1{\Omega^{**}}\gg\rho\gg\gamma\gg
d \ge\frac1{\Lambda}\ge\epsilon\ge
\pi\ge  \alphaD
\ge \epsilon'\ge
\nu\gg \tau \gg \frac{1}{k_0}>0\;,
\end{split}
\end{align}
and that $k\ge k_0$.
Here, by writing $c>a_1\gg
a_2\gg \ldots \gg a_\ell>0$ we mean that there exist suitable non-decreasing functions
$f_i:(0,c)^i\rightarrow (0,c)$ ($i=1,\ldots,\ell-1$) such that for each $i\in
[\ell-1]$ we have $a_{i+1}<f_{i}(a_1,\ldots,a_i)$. A suitable choice of these functions in~\eqref{eq:KONST} is determined by the properties we require in~\cite{cite:LKS-cut3}.

\medskip

Suppose that $G\in\LKSsmallgraphs{n}{k}{\eta}$ is given with its
$(k,\Omega^{**},\Omega^*,\Lambda,\gamma,\epsilon',\nu,\rho)$-sparse
decomposition $$\class=(\HugeVertices, \clusters,\DenseSpots, \Gblack,
\Gexp,\smallatoms )\;, $$
with respect to the partition
$\{\smallvertices{\eta}{k}{G},\largevertices{\eta}{k}{G}\}$, and with respect to the avoiding threshold $\frac{\rho k}{100\Omega^*}$. We write 
\begin{equation}
 \index{mathsymbols}{*VA@$\largeintoatoms$}\index{mathsymbols}{*VA@$\clustersintoatoms$}
\largeintoatoms:=\shadow_{\Gcapt-\HugeVertices}(\smallatoms,\frac{\rho k}{100\Omega^*})\quad\mbox{and}\quad\clustersintoatoms:=\{C\in\clusters\::\:C\subset \largeintoatoms\}\;.
\label{eq:deflargeintoatoms}
\end{equation}
The graph \index{mathsymbols}{*Gblack@$\BGblack$}$\BGblack$ is the corresponding cluster graph. Let $\clustersize$
\index{mathsymbols}{*C@$\clustersize$}
be the size of an arbitrary cluster\footnote{The number $\clustersize$ is not defined when $\clusters=\emptyset$. However in that case $\clustersize$ is never actually used.} in~$\clusters$.
 Let \index{mathsymbols}{*G@$\Gcapt$}$\Gcapt$ be the spanning subgraph of $G$ formed by the edges captured by $\class$. There are two $(\epsilon,d,\pi \clustersize)$-\semiregular matchings $\mathcal M_A$
and $\mathcal M_B$ in $\GD$, with the following
properties (we abbreviate
$\XA:=\XA(\eta,\class, \mathcal M_A,\mathcal M_B)$,
$\XB:=\XB(\eta,\class, \mathcal M_A,\mathcal M_B)$, and
$\XC:=\XC(\eta,\class, \mathcal M_A,\mathcal M_B)$):\footnote{Let us note that Properties~\eqref{commonsetting1}--\eqref{commonsettingNicDoNAtom} come from~\cite[Lemma~\ref{p1.prop:LKSstruct}]{cite:LKS-cut1} and Properties~\eqref{commonsetting:numbercaptured} and~\eqref{commonsetting:DenseCaptured} come from~\cite[Lemma~\ref{p0.lem:LKSsparseClass}]{cite:LKS-cut0}.}\Referee{(6)}
\begin{enumerate}[(1)]
\item\label{commonsetting1}
$V(\mathcal M_A)\cap V(\mathcal
M_B)=\emptyset$,
\item\label{commonsetting1apul}
$V_1(\mathcal \M_B)\subset S^0$, where
\begin{equation}\label{eq:defS0}
S^0:=\smallvertices{\eta}{k}{G}\setminus
(V(\Gexp)\cup\smallatoms)\;,
\end{equation}
\item\label{commonsetting2} for each $(X,Y)\in \M_A\cup\M_B$, there is a dense
spot $(U,W; F)\in \DenseSpots$ with $X\subset U$ and $Y\subset W$, and further,
either $X\subset \smallvertices{\eta}{k}{G}$ or $X\subset
\largevertices{\eta}{k}{G}$, and $Y\subset \smallvertices{\eta}{k}{G}$ or
$Y\subset \largevertices{\eta}{k}{G}$,
\item\label{commonsetting3}
for each $X_1\in\V_1(\M_A\cup\M_B)$ there exists a cluster $C_1\in \clusters$ such that $X_1\subset C_1$, and for each $X_2\in\V_2(\M_A\cup\M_B)$ there exists $C_2\in \clusters\cup\{\largevertices{\eta}{k}{G}\cap \smallatoms\}$ such that $X_2\subset C_2$,
\item\label{commonsettingMgood} each pair of the \semiregular matching
$\Mgood:=\{(X_1,X_2)\in\M_A\::\: X_1\cup X_2\subset \XA\}$ corresponds to an
edge in $\BGblack$,
\item\label{commonsettingXAS0}$e_{\Gcapt}\big(\XA,S^0\setminus V(\mathcal
M_A)\big)\le \gamma kn$,
\item\label{commonsetting4} $e_{\Gblack}(V(G)\setminus V(\M_A\cup \M_B))\le
\gamma^2kn$,
\item\label{commonsettingNicDoNAtom} for the \semiregular matching \index{mathsymbols}{*Natom@$\NAtom$}
$\NAtom:=\{(X,Y)\in\M_A\cup\M_B\::\: (X\cup Y)\cap\smallatoms\not=\emptyset\}$ we have $e_{\Gblack}\big(V(G)\setminus
V(\M_A\cup \M_B),V(\NAtom)\big)\le \gamma^2 kn$,
\item \label{commonsetting:numbercaptured}$|E(G)\setminus E(\Gcapt)|\le 2\rho
kn$, 
\item\label{commonsetting:DenseCaptured}$|E(\GD)\setminus (E(\Gblack)\cup
E_G[\smallatoms,\smallatoms\cup\bigcup\clusters])|\le \frac 54\gamma kn$.
\end{enumerate}

We now define several additional vertex sets. The first of them, the set $V_+$, is just the complement of the set used in~\eqref{eq:defhatdeg}.\Referee{(5)}
\begin{align}
\label{eq:defV+}\index{mathsymbols}{*V_+@$V_+$}
V_+&:=V(G)\setminus 
(S^0\setminus V(\mathcal M_A\cup\mathcal
M_B)) \\  \label{defV+eq}
& = \largevertices{\eta}{k}{G}\cup
V(\Gexp)\cup\smallatoms\cup V(\mathcal M_A\cup\mathcal
M_B)\;.
\end{align}

The set $L_\#$ defined below is the set of ``bad vertices of $\largevertices{\eta}{k}{G}$'', that is, the set of those vertices which have many uncaptured neighbours in the sparse decomposition. If we think of the set $V_+$ as candidate vertices for embedding certain shrubs (cf.~Remark~\ref{rem:aboutXAXB}) then we better discard vertices with a big uncaptured degree from that set. This leads us to the definition of the set $\Vgood$. Since the set $\HugeVertices$ is treated separately, it is also deleted from $\Vgood$.
\begin{align}
\label{eq:defLsharp}\index{mathsymbols}{*L@$L_\#$}
L_\#&:=\largevertices{\eta}{k}{G}\setminus\largevertices{\frac9{10}\eta}{k}{\Gcapt}\;\mbox{,
and}\\ 
\label{eq:defVgood}\index{mathsymbols}{*Vgood@$\Vgood$} 
\Vgood&:=V_+\setminus (\HugeVertices\cup L_\#)\;.
\end{align}

We can now define sets $\YA$ and $\YB$ which should be regarded as cleaned versions of the sets $\XA$ and $\XB$. Here, by a \emph{cleaning} we mean the process of getting rid of certain atypical vertices. Indeed, Lemma~\ref{lem:YAYB} below asserts that the $\YA$ approximately equals $\XA$ and $\YB$ approximately equals $\XB$. Set
\begin{align}
\label{eq:defYA}
\YA&:=
 \shadow_{\Gcapt}\left(V_+\setminus L_\#, (1+\frac\eta{10})k\right) \setminus \shadow_{G-\Gcapt}\left(V(G),\frac\eta{100} k\right)\;,
 \index{mathsymbols}{*YA@$\YA$}\\
\label{eq:defYB} 
\YB&:=
 \shadow_{\Gcapt}\left(V_+\setminus L_\#, (1+\frac\eta{10})\frac k2\right) \setminus
 \shadow_{G-\Gcapt}\left(V(G),\frac\eta{100} k\right)\index{mathsymbols}{*YB@$\YB$}\;.
\end{align}
When the set $\HugeVertices$ is negligible the configuration we obtain does not involve $\HugeVertices$ at all. In other words, $\HugeVertices$ is not used for embedding. Thus, we use the concept of shadows in the way described  at the beginning of Section~\ref{ssec:shadows} to avoid $\HugeVertices$, and define $\WantiC$ as follows.
\begin{align}
\WantiC&:=(\XA\cup\XB)\cap \shadow_G\left(\HugeVertices, \tfrac{\eta}{100}
k \right) \index{mathsymbols}{*V@$\WantiC$}\label{eq:defWantiC}\;.
\end{align}
Next, we define ``bad sets'' $\gPatoms$, $\gP_1$, $\gP$, $\gP_2$ and $\gP_3$, again using shadows.  
\begin{align}
\nonumber
\gPatoms&:=\shadow_{\Gblack}(V(\NAtom),\gamma k)\setminus V(\M_A\cup\M_B)\;,
\index{mathsymbols}{*Je@$\gPatoms$}\\
\nonumber
\gP_1&:=\shadow_{\Gblack}(V(G)\setminus V(\M_A\cup\M_B),\gamma k)\setminus V(\M_A\cup\M_B)\;,
\index{mathsymbols}{*J1@$\gP_1$}
\\ \nonumber
\gP&:=(\XA\setminus \YA)\cup((\XA\cup \XB)\setminus \YB)\cup\WantiC
\cup
L_\sharp \cup \gP_1\\
\nonumber
&~~~~~\cup\shadow_{\GD\cup\Gcapt}(\WantiC\cup
L_\sharp\cup\gPatoms\cup\gP_1,\eta^2 k/10^5)\;,
\index{mathsymbols}{*J@$\gP$}
\\
\nonumber
\gP_2&:=\XA\cap \shadow_{\Gcapt}(S^0\setminus
V(\M_A),\sqrt\gamma k)\;,
\index{mathsymbols}{*J2@$\gP_2$}\\
\nonumber
\gP_3&:=\XA\cap\shadow_{\Gcapt}(\XA, \eta^3k/10^3)\;.
\index{mathsymbols}{*J3@$\gP_3$}
\end{align}
Eliminating $\gPatoms$ from an embedding procedure, for example, will guarantee that we will not be forced to enter the set $\NAtom$. This is convenient in some situations. Which sets are ``bad'' depends on a particular configuration we want to get. That is, some properties given in the definitions of our configurations in Section~\ref{ssec:TypesConf} could be phrased in terms of avoiding some of the sets $\gPatoms$, $\gP_1$, $\gP$, $\gP_2$ and $\gP_3$. 
For some other properties of the configurations, we take only some of the sets $\gPatoms$, $\gP_1$, $\gP$, $\gP_2$ and $\gP_3$ as initial natural forbidden sets, but then we need to apply some non-trivial cleaning (in Lemmas~\ref{lem:ConfWhenCXAXB}, \ref{lem:ConfWhenNOTCXAXB}, and 
\ref{lem:ConfWhenMatching}) to get a desired configuration.

We define a set $\mathcal F$ of clusters of $\M_A\cup \M_B$. As it turns out (see Lemma~\ref{lem:propertyYA12}), $\mathcal F$ is actually an $(\M_A\cup \M_B)$-cover.
\begin{align}
\label{def:Fcover}
 \mathcal F&:=\{C\in \V(\M_A):C\subset \XA\}\cup\V_1(\M_B)\;.
\end{align}
\end{setting}
On the interface between Lemma~\ref{outerlemma} and Lemma~\ref{lem:ConfWhenMatching} we shall need to work with a \semiregular matching which is formed of only those edges $E(\DenseSpots)$ which are either incident with $\smallatoms$, or included in $\Gblack$. The following lemma provides us with an appropriate ``cleaned version of $\DenseSpots$''. The notion of being absorbed adapts in a straightforward way to two families of dense spots: a family of dense spots ${\DenseSpots}_1$ \emph{is absorbed} by another family ${\DenseSpots}_2$ if for every $D_1\in{\DenseSpots}_1$ there exists $D_2\in{\DenseSpots}_2$ such that $D_1$ is contained in $D_2$ as a subgraph.

\begin{lemma}\label{lem:clean-spots}
Assume we are in Setting~\ref{commonsetting}. Then there exists a family
\index{mathsymbols}{*D@$\DenseSpots_\class$}
$\DenseSpots_\class$ of edge-disjoint
$(\gamma^3 k/4,\gamma/2)$-dense spots absorbed by $\DenseSpots$ such that
\begin{enumerate}
 \item $|E(\DenseSpots)\setminus E(\DenseSpots_\class)|\le \rho kn$, and
 \item $E(\DenseSpots_\class)\subset E(\Gblack)\cup E(G[\smallatoms, \smallatoms\cup\bigcup\clusters])$.
\end{enumerate}
\end{lemma}
The proof of Lemma~\ref{lem:clean-spots} is a warm-up for proofs in Section~\ref{ssec:cleaning}.
\begin{proof}[Proof of Lemma~\ref{lem:clean-spots}]
Let $\DenseSpots^-\subset\DenseSpots$ be the set of dense spots $D\in\DenseSpots$ for which \Referee{(A)}
\begin{align*}
\sqrt{\gamma}e(D) \le \big|E(D)\setminus (E(\Gblack)\cup E(G[\smallatoms,
 \smallatoms\cup\bigcup\clusters])\big|\;.
\end{align*}
Thus,
\begin{align}
\nonumber
\sqrt{\gamma}e(\DenseSpots^-) &\le \big|E(\DenseSpots^-)\setminus (E(\Gblack)\cup E(G[\smallatoms,
\smallatoms\cup\bigcup\clusters])\big|\\
\nonumber
&\le \big|E(\DenseSpots)\setminus (E(\Gblack)\cup E(G[\smallatoms,
\smallatoms\cup\bigcup\clusters])\big|\\
\label{eq:vrzeto}
\JUSTIFY{by S\ref{commonsetting}\eqref{commonsetting:DenseCaptured}}
&\le \frac54\gamma kn
\;.
\end{align}

 For each  $D\in\DenseSpots\setminus \DenseSpots^-$ we show below how to extract a $(\gamma^3
 k/4,\gamma/2)$-dense spot $D'\subset D$ with
\begin{equation}\label{eq:Pest}
e(D')\ge
 (1-2\sqrt{\gamma})e(D)
\end{equation}
and $E(D')\subset E(\Gblack)\cup E(G[\smallatoms, \smallatoms\cup\bigcup\clusters])$. Let $\DenseSpots_\class$ be the set of all thus obtained $D'$. That is, we have $E(\DenseSpots_\class)\subset E(\DenseSpots\setminus\DenseSpots^-)$. This ensures Property~2. We also have Property~1, since\Referee{(A)}
\begin{align*}
|E(\DenseSpots)\setminus E(\DenseSpots_\class)|&=|E(\DenseSpots^-)|+|E(\DenseSpots\setminus \DenseSpots^-)\setminus E(\DenseSpots_{\class})|\\
\JUSTIFY{\eqref{eq:vrzeto} for 1st term and \eqref{eq:Pest} for 2nd term}&\le
\frac54\sqrt{\gamma}kn+2\sqrt{\gamma}\cdot
e(\DenseSpots)\\ \JUSTIFY{as
$e(\DenseSpots)\le e(G)\le kn$}&\le \rho kn\;.
\end{align*}

We now show how to extract a $(\gamma^3 k/4,\gamma/2)$-dense spot $D'\subset D$
with $e(D')\ge (1-2\sqrt{\gamma})e(D)$ and $E(D')\subset E(\Gblack)\cup E(G[\smallatoms, \smallatoms\cup\bigcup\clusters])$ from any spot $D\in\DenseSpots\setminus \DenseSpots^-$. Let $D=(A,B; F)$, and $a:=|A|$, $b:=|B|$. As
$D$ is $(\gamma k, \gamma)$-dense, we have $a,b\ge \gamma k$. 
Note also that Definition~\ref{def:densespot} gives that\Referee{(B)}
\begin{equation}\label{eq:mtrh}
e(D)\ge\gamma ab>\frac{\gamma^{1.5}ab}{2}\;.
\end{equation}
First, we 
discard from $D$ all edges not contained in $E(\Gblack)\cup
E(G[\smallatoms,\smallatoms\cup\bigcup\clusters])$ to obtain a dense spot
$D^*\subset D$ with $e(D^*)\ge (1-\sqrt{\gamma})e(D)$. Next, we perform a sequential cleaning procedure in $D^*$. As long as there are such vertices, discard from $A$ any vertex whose current degree is less
than $\gamma^2b/4$, and discard from $B$ any vertex whose current degree is less
than $\gamma^2a/4$. When this procedure terminates,  the
resulting graph $D'=(A',B'; F')$ has $\mindeg_{D'}(A')\ge \gamma^2 b/4\ge
\gamma^3 k/4$ and  $\mindeg_{D'}(B')\ge \gamma^3 k/4$. Note that we
deleted at most $a\cdot \gamma^2 b/4+b\cdot\gamma^2 a/4$ edges out of the at
least $(1-\sqrt{\gamma})e(D)$ edges of $D^*$. This means that 
$$
e(D')\ge (1-\sqrt{\gamma})e(D)-\gamma^2ab/2\geByRef{eq:mtrh} (1-2\sqrt{\gamma})e(D)\;,
$$as desired.\Referee{(B)} Thus we also have  the required density of $D'$, namely 
$\density_{D'}(A',B')\ge (1-2\sqrt{\gamma})\gamma\ge \gamma/2$. 
\end{proof}

\bigskip
In some cases we shall also partition the set $V(G)$ into three sets as in
Lemma~\ref{lem:randomSplit}. This motivates the following
definition.
\begin{definition}[\bf Proportional
splitting]\label{def:proportionalsplitting}\index{general}{proportional splitting}
Let $\proporce{0},\proporce{1},\proporce{2}>0$ be three positive reals with
$\sum_i\proporce{i}\le 1$. Under Setting~\ref{commonsetting}, suppose that
$\colouringpartition$ is a partition of $V(G)\setminus\HugeVertices$ satisfying the assertions of Lemma~\ref{lem:randomSplit} with parameter
$p_\PARAMETERPASSING{L}{lem:randomSplit}:=10$ for graph
$G^*_\PARAMETERPASSING{L}{lem:randomSplit}:=(\Gcapt-\HugeVertices)\cup\GD$ (here the union means union of the edges), bounded decomposition $( \clusters,\DenseSpots,
\Gblack, \Gexp, \smallatoms )$, matching
$\M_\PARAMETERPASSING{L}{lem:randomSplit}:=\M_A\cup\M_B$, sets $\VXV_1:=\Vgood,
\VXV_2:=\XA\setminus(\HugeVertices\cup \gP)$, $\VXV_3:=\XB\setminus \gP$,
$\VXV_4:=V(\Gexp)$, $\VXV_5:=\smallatoms$, $\VXV_6:=\largeintoatoms$, $\VXV_7:=\gPatoms$, $\VXV_8:=\largevertices{\eta}{k}{G}$, $\VXV_9:=L_\sharp$, $\VXV_{10}:=\WantiC$ and reals
$\aXa_1:=\proporce{0},\aXa_2:=\proporce{1}$, $\aXa_3:=\proporce{2}$,
$\aXa_4:=\ldots=\aXa_{10}=0$. Note that by
Lemma~\ref{lem:randomSplit}\eqref{It:H7} we have that $\colouringpartition$ is
a partition of $V(G)\setminus \HugeVertices$. We call $\colouringpartition$
a \emph{proportional $(\proporce{0}:\proporce{1}:\proporce{2})$ splitting}.

We refer to properties of the proportional
$(\proporce{0}:\proporce{1}:\proporce{2})$  splitting $\colouringpartition$
using the numbering of Lemma~\ref{lem:randomSplit}; for example,
``Definition~\ref{def:proportionalsplitting}\eqref{It:HproportionalSizes}''
tells us, among others, that $|(\XA\setminus(\gP\cup\HugeVertices))\cap
\colouringp{0}|\ge\proporce{0}|\XA\setminus(\gP\cup\HugeVertices)|-n^{0.9}$.
\end{definition}
\begin{setting}\label{settingsplitting}
Under Setting~\ref{commonsetting}, suppose that we are given a
proportional 
\index{mathsymbols}{*Pa@$\proporce{i}$}
$(\proporce{0}:\proporce{1}:\proporce{2})$
splitting \index{mathsymbols}{*P@$\colouringp{i}$}
$\colouringpartition$ of $V(G)\setminus\HugeVertices$. We assume
that 
\begin{equation}\label{eq:proporcevelke}
\proporce{0},\proporce{1},\proporce{2}\ge\frac{\eta}{100}\;.
\end{equation}
Let\index{mathsymbols}{*V@$\exceptVertSplit$}\index{mathsymbols}{*V@$\exceptSemSplit$}\index{mathsymbols}{*V@$\exceptClustSplit$}
$\exceptVertSplit,\exceptSemSplit,\exceptClustSplit$ be the exceptional sets as in Definition~\ref{def:proportionalsplitting}\eqref{It:H1}.

We write \index{mathsymbols}{*F@$\shadowsplit$}
\begin{equation}\label{def:shadowsplit}
\shadowsplit:=\shadow_{\GD}\left(\bigcup 
\exceptSemSplit\cup\bigcup \exceptSemSplit^*\cup \bigcup \exceptClustSplit,\frac{\eta^2k}{10^{10}}\right)\;,
\end{equation} where \index{mathsymbols}{*V@$\exceptSemSplit^*$}$\exceptSemSplit^*$ are family of partners of $\exceptSemSplit$ in $\M_A\cup\M_B$.

We have 
\begin{equation}\label{eq:boundShadowsplit}
|\shadowsplit|\le \epsilon n\;.
\end{equation}

For an arbitrary
set $U\subset V(G)$ and for $i\in\{0,1,2\}$ we write \index{mathsymbols}{*U@$U\colouringpI{i}$}$U\colouringpI{i}$ for
the set $U\cap\colouringp{i}$. 

For each $(X,Y)\in\M_A\cup\M_B$ such that $X,Y\notin\exceptSemSplit$
we write $(X,Y)\colouringpI{i}$ for an arbitrary fixed pair $(X'\subset
X,Y'\subset Y)$ with the property that
$|X'|=|Y'|=\min\{|X\colouringpI{i}|,|Y\colouringpI{i}|\}$. We extend this notion
of restriction to an arbitrary \semiregular matching $\mathcal N\subset
\M_A\cup\M_B$ as follows. We set\index{mathsymbols}{*N@$\mathcal N\colouringpI{i}$}
$$\mathcal
N\colouringpI{i}:=\big\{(X,Y)\colouringpI{i}\::\:(X,Y)\in\mathcal
N \text{ with } X,Y\notin\exceptSemSplit\big\}\;.$$
\end{setting}

The next lemma provides some simple properties of a restriction of a
\semiregular matching.

\begin{lemma}\label{lem:RestrictionSemiregularMatching}
Assume Setting~\ref{commonsetting} and Setting~\ref{settingsplitting}. Then for each $i\in\{0,1,2\}$, and for
each $\mathcal N\subset\M_A\cup\M_B$ we have that $\mathcal N\colouringpI{i}$ is
a $(\frac{400\epsilon}{\eta},\frac
d2,\frac{\eta\pi}{200}\clustersize)$-\semiregular matching satisfying
\begin{equation}\label{eq:restrictedmatchlarge}
|V(\mathcal N\colouringpI{i})|\ge
\proporce{i}|V(\mathcal N)|-2k^{-0.05}n\;.
\end{equation}
Moreover for all $v\not\in \shadowsplit$ and for all $i=0,1,2$ we have 
$\deg_{\GD}(v, V(\mathcal N)\colouringpI{i}\setminus V(\mathcal
N\colouringpI{i}))\le \frac {\eta^2k}{10^5}$. 
\end{lemma}
\begin{proof}
Let us consider an arbitrary pair $(X,Y)\in\mathcal N$. By
Definition~\ref{def:proportionalsplitting}\eqref{It:H3} we have
\begin{equation}\label{eq:lossesOnePair}
|X\colouringpI{i}|\ge\proporce{i}|X|-k^{0.9}\geByRef{eq:proporcevelke}
\frac{\eta}{200}|X|\quad\mbox{and}\quad|Y\colouringpI{i}|\ge\proporce{i}|Y|-k^{0.9}\geByRef{eq:proporcevelke}
\frac{\eta}{200}|Y|\;.
\end{equation}
In particular, Fact~\ref{fact:BigSubpairsInRegularPairs} gives that
$(X,Y)\colouringpI{i}$ is a $400\epsilon/\eta$-regular pair of density at least
$d/2$.

We now turn to~\eqref{eq:restrictedmatchlarge}. The total order of pairs
$(X,Y)\in\mathcal N$ excluded entirely from $\mathcal N\colouringpI{i}$\Referee{(8)} is
at most 
\begin{equation}\label{eq:exclude1w}
2\exp(-k^{0.1})n<k^{-0.05}n
\end{equation}
by
Definition~\ref{def:proportionalsplitting}\eqref{It:H1}. Further, for each $(X,Y)\in\mathcal N$ whose part is included to $\mathcal
N\colouringpI{i}$\Referee{(8)} we have by that
\begin{equation}\label{eq:Imyd}
|V((X,Y)\colouringpI{i})|\geByRef{eq:lossesOnePair}\proporce{i}(|X|+|Y|)-2k^{0.9}\;.
\end{equation}
Recall that $\mathcal M_A$ and $\mathcal M_B$ are $(\epsilon,d,\pi \clustersize)$-\semiregulars. In particular, $\mathcal M_A$
and $\mathcal M_B$ are $(\epsilon,d,k^{0.95})$-\semiregulars. Consequently, 
\begin{equation}\label{eq:HDhs}
|\mathcal N|\le |\mathcal M_A\cup\mathcal M_B|\le\frac{n}{2k^{0.95}}
\end{equation}\Referee{(C)} Collecting the loss caused by entirely excluded pairs in~\eqref{eq:exclude1w} and the loss of at most $2k^{0.9}$ vertices from~\eqref{eq:Imyd} to each of the at most $|\mathcal N|$-many non-excluded pairs, we get that
$$
|V(\mathcal N\colouringpI{i})|\geByRef{eq:exclude1w}
\proporce{i}|V(\mathcal N)|-k^{-0.05}n-2k^{0.9}|\mathcal N|\geByRef{eq:HDhs}
\proporce{i}|V(\mathcal N)|-2k^{-0.05}n\;,$$ and~\eqref{eq:restrictedmatchlarge} follows.

For the moreover part, note that by Fact~\ref{fact:sizedensespot} and Fact~\ref{fact:boundedlymanyspots}
\begin{equation*}
\deg_{\GD}(v, V(\mathcal N)\colouringpI{i}\setminus V(\mathcal
N\colouringpI{i}))\le \frac
{\eta^2k}{10^{10}}+\frac {(\Omega^*)^2}{\pi\nu\gamma^2}\cdot 3k^{0.9}\le \frac {\eta^2k}{10^5}\;.
\end{equation*}
\end{proof}

\smallskip
\HIDDENTEXT{there were lemmas lem:clSpadlychA, lem:clSpadlychB, now under TRIVIALCLEANING}

The following lemma gives a useful bound on the sizes of some sets defined on
page~\pageref{eq:defLsharp}.

\begin{lemma}\label{lem:YAYB}
Suppose we are in Setting~\ref{commonsetting}. Let 
\begin{equation}\label{eq:verifybeta}
\beta>\eta^2\sqrt{\gamma}
\end{equation} be arbitrary. Suppose that all but at most
$\beta kn$\Referee{(D)} edges are captured by $\class$. 
Then,\begin{align}
\label{eq:Lsharp-small}
|L_\#|&\le \frac{20\beta}{\eta}n\\
\label{eq:XAYA}
|\XA\setminus\YA|&\le \frac{600\beta}{\eta^2}n\;\mbox{,
and}\\\label{eq:XBYB} 
|(\XA\cup \XB)\setminus\YB|&\le
\frac{600\beta}{\eta^2}n \;.
\end{align}

Further, let $\tilde\beta>0$ be arbitrary. If $e_G(\HugeVertices,\XA\cup\XB)\le\tilde \beta kn$ then
\begin{align}\label{eq:WantiCbound}
|\WantiC|\le \frac{100\tilde\beta n}{\eta}\;. 
\end{align}
\end{lemma}
\begin{proof}
Let $W_1:=\{v\in V(G)\::\:\deg_G(v)-\deg_{\Gcapt}(v)\ge \eta k/100\}$. We have $|W_1|\le \frac{200\beta}{\eta}n\le  \frac{100\beta}{\eta^2}n$. 

Observe that $L_\#$ sends out at most
$(1+\frac9{10}\eta)k|L_\#|<\frac{40\beta}{\eta}kn$ edges in $\Gcapt$. Let $W_2:=\{v\in V(G)\::\:\deg_{\Gcapt}(v,L_\#)\ge \eta k/10\}$. We have $|W_2|\le \frac{400\beta}{\eta^2}n$.

Let $W_3:=\{v\in \XA\::\: \deg_{\Gcapt}(v,S^0\setminus V(\M_A))\ge \sqrt{\gamma} k\}$. By Setting~\ref{commonsetting}\eqref{commonsettingXAS0} we have\Referee{(D)} $$|W_3|\le \sqrt\gamma n\leByRef{eq:verifybeta} \frac{\beta}{\eta^2}n\;.$$

For~\eqref{eq:XAYA}, observe that $\XA\setminus\YA\subset W_1\cup W_2\cup W_3$. For~\eqref{eq:XBYB}, observe that $\XB\setminus \YB\subset W_1\cup W_2$ and that $\YA\subset \YB$. Thus, $(\XA\cup \XB)\setminus\YB\subset (\XA\setminus \YA)\cup(\XB\setminus \YB)\subset W_1\cup W_2\cup W_3$.

The bound~\eqref{eq:WantiCbound} follows from~\eqref{eq:defWantiC}.
\HIDDENTEXT{A lengthy proof removed and put into HIDDENTEXT.TEX under LENGTHY.}
\end{proof}

We finish this section with an auxiliary result which will only be used later in the proofs
of Lemmas~\ref{lem:ConfWhenNOTCXAXB} and~\ref{lem:ConfWhenMatching}.

\begin{lemma}\label{lem:propertyYA12}
Assume Settings~\ref{commonsetting} and~\ref{settingsplitting}.
We have
\begin{align}\label{eq:trivka}
\XA\colouringpI{0}\setminus(\gP\cup \shadowsplit)\subseteq
\colouringp{0}\setminus \left(\shadowsplit\cup \shadow_{\GD}\left(\WantiC, \frac
{\eta^2 k}{10^5}\right)\right)\;,\\
\label{eq:propertyYA12cB3}
\maxdeg_{\Gcapt}\left(\XA\setminus(\gP_2\cup \gP_3),\bigcup
\mathcal F\right)\le \frac{3\eta^3}{2\cdot 10^3} k\;,
\end{align}
and for $i=1,2$ we have
\begin{align}
\label{eq:propertyYA12cA} 
\mindeg_{\Gcapt}\left(\XA\setminus(\gP\cup\exceptVertSplit),\Vgood\colouringpI{i} \right)\ge
\proporce{i}\left(1+\frac{\eta}{20}\right)k\;,\\
\label{eq:propertyYA12cB2}
\mindeg_{\Gcapt}\left(\XB\setminus(\gP\cup\exceptVertSplit),\Vgood\colouringpI{i} \right)\ge
\proporce{i}\left(1+\frac{\eta}{20}\right)\frac k2\;\mbox.
\end{align}

Moreover, $\mathcal F$ defined in~\eqref{def:Fcover} is an $(\M_A\cup\M_B)$-cover.
\end{lemma}

\begin{proof}
The definition of $\gP$ gives~\eqref{eq:trivka}.

For~\eqref{eq:propertyYA12cA} and~\eqref{eq:propertyYA12cB2}, assume that $i=2$ (the other case is analogous).
Observe that 
\begin{align*}
\mindeg_{\Gcapt}&\left(\YA\setminus
(\WantiC\cup\exceptVertSplit),\Vgood\colouringpI{2}\right)\\
\JUSTIFY{by Def~\ref{def:proportionalsplitting}\eqref{It:H5}}
&\ge\proporce{2}\cdot \mindeg_{\Gcapt}(\YA\setminus
\WantiC,\Vgood)-k^{0.9}\\ 
\JUSTIFY{by
\eqref{eq:defVgood}}
&\ge\proporce{2}\cdot \big(\mindeg_{\Gcapt}(\YA,V_+\setminus
L_\sharp)-\maxdeg_{\Gcapt}(\YA\setminus\WantiC,\HugeVertices)\big)-k^{0.9}\\ 
\JUSTIFY{by~\eqref{eq:defYA}, \eqref{eq:defWantiC}}
&\ge \proporce{2}\cdot \left(\left(1+\frac{\eta}{10}\right)k-\frac{\eta k}{100}\right)-k^{0.9}\\ 
\JUSTIFY{by~\eqref{eq:KONST}, \eqref{eq:proporcevelke}}
&\ge \proporce{2}\cdot \left(1+\frac{\eta}{20}\right)k\;,
\end{align*}
which proves~\eqref{eq:propertyYA12cA}, as $\XA\setminus (\gP\cup
 \exceptVertSplit)\subseteq \YA\setminus (\WantiC\cup \exceptVertSplit)$.
Similarly, we obtain that $$\mindeg_{\Gcapt}\left(\YB\setminus
(\WantiC\cup\exceptVertSplit),\Vgood\colouringpI{2}\right)\ge
\proporce{2}\left(1+\frac{\eta}{20}\right)\frac k2\;,$$ which proves~\eqref{eq:propertyYA12cB2}.

We have $\maxdeg_{\Gcapt}(\XA\setminus \gP_3,\XA)< \frac{\eta^3}{10^3} k$, and 
$\maxdeg_{\Gcapt}(\XA\setminus \gP_2,S^0\setminus V(\M_A))<\sqrt\gamma k$.
Thus~\eqref{eq:propertyYA12cB3} follows from
Setting~\ref{commonsetting}\eqref{commonsetting1apul}
and by~\eqref{eq:KONST}.

For the ``moreover'' part, it
suffices to prove that $\{C\in \V(\M_A):C\subset \XA\}=\mathcal F\sm\V_1(\M_B)$ is an $\M_A$-cover. Let  $(T_1,T_2)\subset \M_A$. As $G\in\LKSsmallgraphs{n}{k}{\eta}$, we have,
by Setting~\ref{commonsetting}\eqref{commonsetting2} that for some
$i\in\{1,2\}$, $T_i$ is contained in $\largevertices{\eta}{k}{G}$. Then
by Setting~\ref{commonsetting}\eqref{commonsetting1}, $T_i\subset \XA$, as desired.
\end{proof}

\section{Ten types of Configurations}\label{sec:typesofconfigurations}

We now come to the heart of the present paper. We will
 introduce ten configurations
--- called $\mathbf{(\diamond1)}$--$\mathbf{(\diamond10)}$ --- which may be found in a
graph $G\in \LKSgraphs{n}{k}{\eta}$.\footnote{Saying that ``we have Configuration X'', ``the graph is in Configuration X'', or  ``Configuration X occurs'' is the same.} We will be able to infer from the main
results of this section
(Lemmas~\ref{lem:ConfWhenCXAXB}--\ref{lem:ConfWhenMatching}) and from other
structural results of this paper and of~\cite{cite:LKS-cut1} that each graph $G\in \LKSgraphs{n}{k}{\eta}$
contains at least one of these configurations.
Lemmas~\ref{lem:ConfWhenCXAXB}--\ref{lem:ConfWhenMatching} are based on the
structure provided by \cite[Lemma~\ref{p1.prop:LKSstruct}]{cite:LKS-cut1}. We refer to~\cite[Section~\ref{p3.ssec:embeddingOverview}]{cite:LKS-cut3} where we describe in more detail how each of the configurations $\mathbf{(\diamond1)}$--$\mathbf{(\diamond10)}$ 
 can be used for the
embedding of any given tree from $\treeclass{k}$, as required for
Theorem~\ref{thm:main}. A full description and proofs of the embedding strategies is given in \cite[Section~\ref{p3.sec:MainEmbedding}]{cite:LKS-cut3}.

The organization of this section is as follows.
In
Section~\ref{ssec:TypesConf} we state some preliminary definitions and introduce the configurations $\mathbf{(\diamond1)}$--$\mathbf{(\diamond10)}$. In Section~\ref{ssec:cleaning} we prove certain ``cleaning lemmas''. The main results are then stated and proved  in Section~\ref{ssec:obtainingConf}. The results of
Section~\ref{ssec:obtainingConf} rely on the auxiliary lemmas of
Section~\ref{ssec:RandomSplittins} and~\ref{ssec:cleaning}.

\subsection{The configurations}\label{ssec:TypesConf}
We can now define the following preconfigurations~$\mathbf{(\clubsuit)}$,
$\mathbf{(\heartsuit1)}$, $\mathbf{(\heartsuit2)}$, $\mathbf{(exp)}$, and
$\mathbf{(reg)}$, and the configurations\footnote{The
word ``configuration'' is used for a final structure in a graph which is suitable for embedding purposes while ``preconfigurations'' are building blocks for configurations.} $\mathbf{(\diamond1)}$--$\mathbf{(\diamond10)}$. 
Lemma~\ref{outerlemma} (proof of which occupies Section~\ref{ssec:obtainingConf}) asserts that each graph
$\LKSgraphs{n}{k}{\eta}$ contains at least one of the configurations $\mathbf{(\diamond1)}$--$\mathbf{(\diamond10)}$.
More precisely, after getting the ``rough structure'' we obtained in~\cite{cite:LKS-cut1} we get one of the configurations $\mathbf{(\diamond1)}$--$\mathbf{(\diamond10)}$ from Lemma~\ref{outerlemma}, which builds on the analysis given in Lemmas~\ref{lem:ConfWhenCXAXB}--\ref{lem:ConfWhenMatching}. 

We now give a brief overview of these configurations. Recall that for our proof of Theorem~\ref{thm:main} we combine these configurations (in the host graph $G_\PARAMETERPASSING{T}{thm:main}$) with a given fine partition of the tree $T_\PARAMETERPASSING{T}{thm:main}$ which was informally explained in Section~\ref{sec:TREEcut}.

Configuration~$\mathbf{(\diamond1)}$ covers the easy and lucky case when $G$ contains a subgraph with high minimum degree. A very simple tree-embedding strategy similar to the greedy strategy turns out to work in this case.

The purpose of
Preconfiguration~$\mathbf{(\clubsuit)}$ is to utilize vertices
of~$\HugeVertices$. On the one hand these vertices seem very powerful because of
their large degree, on the other hand the edges incident with them are very
unstructured. Therefore Preconfiguration~$\mathbf{(\clubsuit)}$  distils some
structure in~$\HugeVertices$. This preconfiguration is then a part of
configurations~$\mathbf{(\diamond2)}$--$\mathbf{(\diamond5)}$ which deal with the
case when $\HugeVertices$ is substantial. Indeed,
Lemma~\ref{lem:ConfWhenCXAXB} asserts that whenever $\HugeVertices$ is incident with many edges, 
then at least one of
configurations~$\mathbf{(\diamond1)}$--$\mathbf{(\diamond5)}$ must occur. 

Let us note that each of the configurations $\mathbf{(\diamond1)}$--$\mathbf{(\diamond5)}$ alone suffices for embedding all $k$-vertex trees. However, when $\HugeVertices$ is negligible, we may need different configurations $\mathbf{(\diamond6)}$--$\mathbf{(\diamond10)}$ (with different parameters) for embedding different individual trees from $\treeclass{k}$.

The cases when the number of edges incident with $\HugeVertices$ is negligible
are covered by configurations $\mathbf{(\diamond6)}$--$\mathbf{(\diamond10)}$. More precisely, in this setting Lemma~\ref{outerlemma} transforms the output structure we obtained in~\cite{cite:LKS-cut1} into an input structure for either Lemma~\ref{lem:ConfWhenNOTCXAXB} or Lemma~\ref{lem:ConfWhenMatching}. These lemmas then assert that, indeed, one of the Configurations $\mathbf{(\diamond6)}$--$\mathbf{(\diamond10)}$ must occur. 
The configurations~$\mathbf{(\diamond6)}$--$\mathbf{(\diamond8)}$ involve
combinations of one of the
two preconfigurations $\mathbf{(\heartsuit1)}$ and $\mathbf{(\heartsuit2)}$ and
one of the two preconfigurations $\mathbf{(exp)}$ and
$\mathbf{(reg)}$. The idea here is that the \kknnaaggss are embedded
using the structure of $\mathbf{(exp)}$ or $\mathbf{(reg)}$ (whichever is
applicable), the internal shrubs are embedded using the structure which is
specific to each of the
configurations~$\mathbf{(\diamond6)}$--$\mathbf{(\diamond8)}$, and the end
shrubs are embedded using the structure of $\mathbf{(\heartsuit1)}$ or
$\mathbf{(\heartsuit2)}$. For this reason, configurations~$\mathbf{(\diamond6)}$--$\mathbf{(\diamond9)}$ are accompanied by parameters (denoted by $h$, $h_1$ and $h_2$ in Definitions~\ref{def:CONF6}--\ref{def:CONF9})
which correspond to the total orders of shrubs of different kinds.
The configuration $\mathbf{(\diamond10)}$ is very similar to the structures obtained in the dense setting
in~\cite{PS07+,HlaPig:LKSdenseExact}, and $\mathbf{(\diamond9)}$ should be
considered as half-way towards it.


\smallskip
Some of the  configurations below are accompanied with parameters in the parentheses; note that we do not make explicit those numerical parameters which are inherited from Setting~\ref{commonsetting}.

\bigskip We start  by defining Configuration~$\mathbf{(\diamond1)}$. This is a very easy configuration in which a modification of the greedy tree-embedding strategy works.
\begin{definition}[\bf Configuration~$\mathbf{(\diamond1)}$]\index{mathsymbols}{**1@$\mathbf{(\diamond1)}$}
We say that a graph $G$ is in \emph{Configuration~$\mathbf{(\diamond1)}$} if
there exists a non-empty bipartite graph $H\subset G$ with $\mindeg_G(V(H))\ge k$ and $\mindeg(H)\ge k/2$.
\end{definition}

\bigskip We now introduce the configurations $\mathbf{(\diamond2)}$--$\mathbf{(\diamond5)}$ which make use of the set $\HugeVertices$. These configurations build on Preconfiguration $\mathbf{(\clubsuit)}$. 

\begin{definition}[\bf Preconfiguration 
$\mathbf{(\clubsuit)}$]\index{mathsymbols}{***@$\mathbf{(\clubsuit)}$}
\label{def:PreClub}\Referee{(E)}
Suppose that we are in
Setting~\ref{commonsetting}. We say that the graph $G$ is in
\emph{Preconfiguration~$\mathbf{(\clubsuit)}(\Omega^{\star})$} if the following
conditions are satisfied.
$G$ contains non-empty sets $L''\subset L'\subset
\largevertices{\frac9{10}\eta}{k}{\Gcapt}\setminus\HugeVertices$, and a non-empty set $\HugeVertices'\subset
\HugeVertices$ such that
\begin{align}\label{eq:clubsuitCOND1}
\maxdeg_{\Gcapt} (L',\HugeVertices\setminus \HugeVertices')&<\frac{\eta
k}{100} \;\mbox{,}\\ 
\label{eq:clubsuitCOND2}
\mindeg_{\Gcapt}(\HugeVertices',\BUG{L'})&\ge \Omega^\star k\;\mbox{, and}\\
\label{eq:clubsuitCOND3}
\maxdeg_{\Gcapt}(L'',\largevertices{\frac{9}{10}\eta}{k}{\Gcapt}\setminus(\HugeVertices\cup
L'))&\le\frac{\eta k}{100}\;.
\end{align}
\end{definition}

\begin{definition}[\bf Configuration
$\mathbf{(\diamond2)}$]\index{mathsymbols}{**2@$\mathbf{(\diamond2)}$}Suppose that we are in Setting~\ref{commonsetting}. We say that the graph $G$ is in
\emph{Configuration $\mathbf{(\diamond2)}(\Omega^\star,
\tilde\Omega,\beta)$} if the following
conditions are satisfied.

 The triple $L'',L',\HugeVertices'$ witnesses preconfiguration
$\mathbf{(\clubsuit)}(\Omega^\star)$ in $G$. There exist a
non-empty set $\HugeVertices''\subset \HugeVertices'$, a set $V_1\subset V(\Gexp)\cap\YB\cap L''$, and a set $V_2\subset V(\Gexp)$ with the following properties.
\begin{align*}
\mindeg_{\Gcapt}(\HugeVertices'',V_1)&\ge\tilde\Omega k\;\\
\mindeg_{\Gcapt}(V_1,\HugeVertices'')&\ge \beta k\;,\\
\mindeg_{\Gexp}(V_1,V_2)&\ge \beta k\;,\\
\mindeg_{\Gexp}(V_2,V_1)&\ge \beta k\;.
\end{align*}
\end{definition}

\begin{definition}[\bf Configuration
$\mathbf{(\diamond3)}$]\index{mathsymbols}{**3@$\mathbf{(\diamond3)}$}
\label{def:CONF3}
Suppose that we are in
Setting~\ref{commonsetting}. We say that the graph $G$ is in
\emph{Configuration $\mathbf{(\diamond3)}(\Omega^\star,
\tilde\Omega,\zeta,\delta)$} if the following
conditions are satisfied.

 The triple $L'',L',\HugeVertices', $ witnesses preconfiguration
$\mathbf{(\clubsuit)}(\Omega^\star)$ in $G$. There exist a
non-empty set $\HugeVertices''\subset \HugeVertices'$, a set $V_1\subset \smallatoms\cap \YB\cap L''$, and a set $V_2\subset V(G)\setminus \HugeVertices$ such that the
following properties are satisfied.
\begin{align}
\nonumber
\mindeg_{\Gcapt}(\HugeVertices'',V_1)&\ge \tilde \Omega k\;,\\
\nonumber
\mindeg_{\Gcapt}(V_1,\HugeVertices'')&\ge \delta k\;,\\
\label{eq:WHtc}
\maxdeg_{\GD}(V_1, V(G)\setminus
(V_2\cup \HugeVertices))&\le \zeta k\;,\\ 
\label{confi3theothercondi}
\mindeg_{\GD}(V_2,V_1)&\ge \delta k\;.
\end{align}
\end{definition}

\begin{definition}[\bf Configuration
$\mathbf{(\diamond4)}$]\index{mathsymbols}{**4@$\mathbf{(\diamond4)}$}
\label{def:CONF4}
Suppose that we are in
Setting~\ref{commonsetting}. We say that the graph~$G$ is in
\emph{Configuration
$\mathbf{(\diamond4)}(\Omega^\star, \tilde\Omega,\zeta,\delta)$} if the
following conditions are satisfied.

 The triple $L'',L',\HugeVertices'$ witnesses preconfiguration
$\mathbf{(\clubsuit)}(\Omega^\star)$ in $G$. There exist a
non-empty set $\HugeVertices''\subset \HugeVertices'$, sets $V_1\subset \YB\cap L''$,
$\smallatoms'\subset \smallatoms$, and $V_2\subset V(G)\setminus \HugeVertices$ with the following properties
\begin{align}
\mindeg_{\Gcapt}(\HugeVertices'',V_1)&\ge \tilde\Omega k\;,\notag \\
\mindeg_{\Gcapt}(V_1,\HugeVertices'')&\ge \delta k\;,\notag \\
\mindeg_{\Gcapt\cup\GD}(V_1,\smallatoms')&\ge \delta k\;,\label{confi4:3} \\
\mindeg_{\Gcapt\cup\GD}(\smallatoms',V_1)&\ge \delta k\;,\label{confi4:4} \\
\mindeg_{\Gcapt\cup\GD}(V_2,\smallatoms')&\ge \delta k\;,\label{confi4othercondi} \\
\maxdeg_{\Gcapt\cup\GD}(\smallatoms',V(G)\setminus
(\HugeVertices\cup V_2))&\le \zeta k\;. \label{confi4lastcondi}
\end{align}
\end{definition}

\begin{definition}[\bf Configuration
$\mathbf{(\diamond5)}$]\index{mathsymbols}{**5@$\mathbf{(\diamond5)}$}
\label{def:CONF5}
Suppose that we are in
Setting~\ref{commonsetting}. We say that the graph $G$ is in
\emph{Configuration
$\mathbf{(\diamond5)}(\Omega^\star,\tilde\Omega,\delta,\zeta,\tilde\pi)$} if the
following conditions are satisfied.

 The triple $L'',L',\HugeVertices'$ witnesses preconfiguration
 $\mathbf{(\clubsuit)}(\Omega^\star)$ in $G$. There exists a non-empty set $\HugeVertices''\subset \HugeVertices'$, and a set $V_1\subset (\YB\cap
L''\cap \bigcup\clusters)\setminus V(\Gexp)$ such that the following conditions are fulfilled.
\begin{align}
\mindeg_{\Gcapt}(\HugeVertices'',V_1)&\ge \tilde\Omega k\;, \\
\mindeg_{\Gcapt}(V_1,\HugeVertices'')&\ge \delta k\;,
\label{eq:diamond5P2}\\
\mindeg_{\Gblack}(V_1)&\ge \zeta k\;.\label{confi5last}
\end{align}
Further, we have 
\begin{equation}\label{eq:diamond5P4}
\mbox{$C\cap V_1=\emptyset$ or $|C\cap V_1|\ge \tilde\pi|C|$}
\end{equation} for every
$C\in\clusters$.
\end{definition}

In remains to introduce
configurations~$\mathbf{(\diamond6)}$--$\mathbf{(\diamond10)}$. In these
configurations the set $\HugeVertices$ is not utilized. All these
configurations make use of Setting~\ref{settingsplitting}, i.e., the set
$V(G)\setminus \HugeVertices$ is partitioned into three sets
$\colouringp{0},\colouringp{1}$ and $\colouringp{2}$. The purpose of $\colouringp{0},\colouringp{1}$ and $\colouringp{2}$
is to make possible to embed the \kknnaaggssNOSPACE, the internal shrubs, and the
end shrubs of $T_\PARAMETERPASSING{T}{thm:main}$, respectively. Thus the
parameters $\proporce{0}, \proporce{1}$ and $\proporce{2}$ are chosen proportionally to
the sizes of these respective parts of $T_\PARAMETERPASSING{T}{thm:main}$.

We first introduce four preconfigurations $\mathbf{(\heartsuit 1)}$,
$\mathbf{(\heartsuit 2)}$, $\mathbf{(exp)}$ and $\mathbf{(reg)}$.

 An \emph{$\M$-cover}\index{general}{cover}\index{mathsymbols}{*COVER@$\M$-cover} of a
\semiregular matching $\M$ is a family $\mathcal F\subset \V(\M)$ with the property that at least one of
the elements $S_1$ and $S_2$ is a member of $\mathcal F$, for each $(S_1,S_2)\in
\M$.

\begin{definition}[\bf Preconfiguration
$\mathbf{(\heartsuit 1)}$]\index{mathsymbols}{**1@$\mathbf{(\heartsuit1)}$}
\label{def:heart1}
 Suppose that we are in
Setting~\ref{commonsetting} and Setting~\ref{settingsplitting}. We say that the
graph $G$ is in \emph{Preconfiguration
$\mathbf{(\heartsuit 1)}(\gamma',h)$} of $V(G)$ if there are two
non-empty sets $V_0,V_1\subset
\colouringp{0}\setminus \left(\shadowsplit\cup\shadow_{\GD}(\WantiC, \frac{\eta^2 k}{10^5})\right)$ with the following
properties.
\begin{align}
\label{COND:P1:3}
\mindeg_{\Gcapt}\left(V_0,\Vgood\colouringpI{2}\right)&\ge h/2 \;\mbox{, and}\\
\label{COND:P1:4}
\mindeg_{\Gcapt}\left(V_1,\Vgood\colouringpI{2}\right)&\ge h \;.
\end{align} Further, there is an
$(\M_A\cup\M_B)$-cover $\mathcal F$ such that
\begin{equation}\label{COND:P1:5}
\maxdeg_{\Gcapt}\left(V_1,\bigcup\mathcal F\right)\le \gamma' k\;.
\end{equation}
\end{definition}

\begin{definition}[\bf Preconfiguration
$\mathbf{(\heartsuit 2)}$]\index{mathsymbols}{**2@$\mathbf{(\heartsuit2)}$} Suppose that we
are in Setting~\ref{commonsetting} and Setting~\ref{settingsplitting}. We say that the
graph $G$ is in \emph{Preconfiguration
$\mathbf{(\heartsuit 2)}(h)$} of $V(G)$ if there are two
non-empty sets $V_0,V_1\subset \colouringp{0}\setminus \left(\shadowsplit\cup \shadow_{\GD}(\WantiC, \frac {\eta^2 k}{10^5})\right)$ with the
following properties.
\begin{align}
\begin{split}\label{COND:P2:4}
\mindeg_{\Gcapt}\left(V_0\cup V_1,\Vgood\colouringpI{2}\right)&\ge h.
\end{split}
\end{align}
\end{definition}

\begin{definition}[\bf Preconfiguration
$\mathbf{(exp)}$]\index{mathsymbols}{**exp@$\mathbf{(exp)}$} 
\label{def:exp8}
Suppose that we
are in Setting~\ref{commonsetting} and Setting~\ref{settingsplitting}. We say that the
graph $G$ is in \emph{Preconfiguration
$\mathbf{\mathbf{(exp)}}(\beta)$} if there are two
non-empty sets $V_0,V_1\subset \colouringp{0}$ with the following properties.
\begin{align}
\label{COND:exp:1}
\mindeg_{\Gexp}(V_0,V_1)&\ge \beta k\;,\\
\label{COND:exp:2}
\mindeg_{\Gexp}(V_1,V_0)&\ge \beta k\;.
\end{align}
\end{definition}

\begin{definition}[\bf Preconfiguration
$\mathbf{(reg)}$]\index{mathsymbols}{**reg@$\mathbf{(reg)}$}
\label{def:reg}
 Suppose that we
are in Setting~\ref{commonsetting} and Setting~\ref{settingsplitting}. We say that the graph $G$ is in \emph{Preconfiguration
$\mathbf{\mathbf{(reg)}}(\tilde \epsilon, d', \mu)$}  if there are two  non-empty sets $V_0,V_1\subset \colouringp{0}$
and a non-empty family of vertex-disjoint $(\tilde\epsilon,d')$-super-regular pairs $\{(Q_0^{(j)},Q_1^{(j)}\}_{j\in\mathcal Y}$ (with respect to the edge set $E(G)$) with $V_0:=\bigcup Q_0^{(j)}$ and $V_1:=\bigcup Q_1^{(j)}$ such that
\begin{align}
\label{COND:reg:0}
\min\left\{|Q_0^{(j)}|,|Q_1^{(j)}|\right\}&\ge\mu k\;.
\end{align}
\end{definition}

\begin{definition}[\bf Configuration
$\mathbf{(\diamond6)}$]\index{mathsymbols}{**6@$\mathbf{(\diamond6)}$}
\label{def:CONF6}
Suppose that we are in
Settings~\ref{commonsetting} and~\ref{settingsplitting}. We say that the graph $G$
is in \emph{Configuration
$\mathbf{(\diamond6)}(\delta, \tilde \epsilon,d',\mu, \gamma', h_2)$} if the
following conditions are satisfied.

The vertex sets $V_0,V_1$ 
witness Preconfiguration
 $\mathbf{(reg)}(\tilde \epsilon,d',\mu)$ or
 Preconfiguration~$\mathbf{(exp)}(\delta)$ and either Preconfiguration~$\mathbf{(\heartsuit1)}(\gamma',h_2)$ or
Preconfiguration~$\mathbf{(\heartsuit2)}(h_2)$. There exist non-empty sets $V_2,V_3\subset \colouringp{1}$ such that
 \begin{align}\label{COND:D6:1}
\mindeg_{G}(V_1,V_2)&\ge \delta k\;,\\ 
\label{COND:D6:2}
\mindeg_{G}(V_2,V_1)&\ge \delta k\;,\\ 
\label{COND:D6:3}
\mindeg_{\Gexp}(V_2,V_3)&\ge \delta k \;,\mbox{and}\\
\label{COND:D6:4}
\mindeg_{\Gexp}(V_3,V_2)&\ge \delta k\;.
\end{align}
\end{definition}

\begin{definition}[\bf Configuration
$\mathbf{(\diamond7)}$]\index{mathsymbols}{**7@$\mathbf{(\diamond7)}$}
\label{def:CONF7}
Suppose that we are in
Settings~\ref{commonsetting} and~\ref{settingsplitting}. We say that the graph $G$
is in \emph{Configuration
$\mathbf{(\diamond7)}(\delta, \rho', \tilde \epsilon, d',\mu, \gamma', h_2)$} if
the following conditions are satisfied.

The sets $V_0,V_1$  witness Preconfiguration
 $\mathbf{(reg)}( \tilde \epsilon, d',\mu)$ and either Preconfiguration~$\mathbf{(\heartsuit
 1)}(\gamma', h_2)$ or Preconfiguration~$\mathbf{(\heartsuit 2)}(h_2)$. There
 exist non-empty sets $V_2\subset \smallatoms\colouringpI{1}\setminus \exceptVertSplit$ and $V_3\subset \colouringp{1}$ such that
 \begin{align}
\label{COND:D7:1}
\mindeg_{G}(V_1,V_2)&\ge \delta k\;,\\
\label{COND:D7:2}
\mindeg_{G}(V_2,V_1)&\ge \delta k\;,\\
\label{COND:D7:3}
\maxdeg_{\GD}(V_2,\colouringp{1}\setminus V_3)&< \rho' k \;\mbox{and}\\
\label{COND:D7:4}
\mindeg_{\GD}(V_3,V_2)&\ge \delta k\;.
\end{align}
\end{definition}

\begin{definition}[\bf Configuration
$\mathbf{(\diamond8)}$]\index{mathsymbols}{**8@$\mathbf{(\diamond8)}$}
\label{def:CONF8}
Suppose that we are in
Settings~\ref{commonsetting} and~\ref{settingsplitting}. We say that the graph $G$
is in \emph{Configuration
$\mathbf{(\diamond8)}(\delta,\rho',\epsilon_1,\epsilon_2, d_1,d_2,\mu_1,\mu_2, h_1,h_2)$}
 if the following conditions are satisfied.

The vertex sets $V_0,V_1$  witness Preconfiguration
 $\mathbf{(reg)}(\epsilon_2, d_2,\mu_2)$ and Preconfiguration~$\mathbf{(\heartsuit 2)}(h_2)$.
 There exist non-empty sets $V_2\subset \colouringp{0}$, $V_3,V_4\subset \colouringp{1}$, $V_3\subset\smallatoms\setminus \exceptVertSplit$, and an $(\epsilon_1, d_1, \mu_1 k)$-\semiregular matching $\mathcal N$ absorbed by $(\M_A\cup\M_B)\setminus \NAtom$, $V(\mathcal N)\subset \colouringp{1}\setminus V_3$ such that
\begin{align}
\label{COND:D8:1}
\mindeg_{G}(V_1,V_2)&\ge \delta k\;,\\
\label{COND:D8:2}
\mindeg_{G}(V_2,V_1)&\ge \delta k\;,\\
\label{COND:D8:3}
\mindeg_{\Gcapt}(V_2,V_3)&\ge \delta k\;,\\
\label{COND:D8:4}
\mindeg_{\Gcapt}(V_3,V_2)&\ge \delta k\;,\\
\label{COND:D8:5}
\maxdeg_{\GD}(V_3,\colouringp{1}\setminus V_4)&< \rho' k \;,\\
\label{COND:D8:6}
\mindeg_{\GD}(V_4,V_3)&\ge \delta k\;\mbox{, and}\\
\label{COND:D8:7}
\deg_{\GD}(v,V_3)+\deg_{\Gblack}(v,V(\mathcal N))&\ge h_1\;\mbox{for each $v\in V_2$.}
\end{align} 
\end{definition}
\begin{definition}[\bf Configuration
$\mathbf{(\diamond9)}$]\index{mathsymbols}{**9@$\mathbf{(\diamond9)}$}
\label{def:CONF9}
Suppose that we are in
Settings~\ref{commonsetting}, and~\ref{settingsplitting}. We say that the graph $G$
is in \emph{Configuration
$\mathbf{(\diamond9)}(\delta, \gamma', h_1, h_2, \epsilon_1, d_1,
\mu_1,\epsilon_2, d_2,\mu_2)$} if the following conditions are satisfied.

The sets $V_0,V_1$ together with the $(\M_A\cup\M_B)$-cover $\mathcal F'$
witness Preconfiguration~$\mathbf{(\heartsuit1)}(\gamma',h_2)$. 
 There exists an $(\epsilon_1, d_1, \mu_1 k)$-\semiregular matching $\mathcal N$ absorbed by $\M_A\cup\M_B$, with $V(\mathcal N)\subset \colouringp{1}$.
Further, there is a family $\{(Q_0^{(j)},Q_1^{(j)})\}_{j\in\mathcal Y}$ as in Preconfiguration~$\mathbf{(reg)}(\epsilon_2,d_2,\mu_2)$. There is a set $V_2\subseteq
 V(\mathcal N)\setminus \bigcup \mathcal F'\subset \bigcup\clusters$ with the
 following properties:
\begin{align}
\label{conf:D9-XtoV}
\mindeg_{\GD}\left(V_1, V_2\right)\ge h_1\;,&\\
\label{conf:D9-VtoX}\mindeg_{\GD}\left(V_2,V_1\right)\ge
\delta k\;.
\end{align}
\end{definition}

Our last configuration, Configuration~$\mathbf{(\diamond10)}$, will lead to an embedding very similar to the one
in the dense case (as treated in~\cite{PS07+}; this will be explained in detail in~\cite{cite:LKS-cut3}). To formalize the configuration we need a preliminary definition. We shall
generalize the standard concept of a regularity graph (in the context of regular
partitions and Szemer\'edi's regularity lemma) to graphs with clusters whose sizes are only bounded from below.

\begin{definition}[\bf{$( \epsilon,d,\ell_1,\ell_2)$-regularized
graph}]\index{general}{regularized graph}\label{def:regularizedGraph} Let $G$ be
a graph, and let $\mathcal V$ be an $\ell_1$-ensemble that partitions $V(G)$.
Suppose that $G[X]$ is empty for each $X\in \mathcal V$ and suppose $G[X,Y]$ is
$\epsilon$-regular and of density either $0$ or at least $d$ for each $X,Y\in
\mathcal V$. Further suppose that for all $X\in \V$ it holds that 
$|\bigcup\neighbour_G(X)|\le \ell_2$.
Then we say that $(G,\mathcal V)$ is an \emph{$(\eps, d,\ell_1,
\ell_2)$-regularized graph}.

A \semiregular matching $\M$ of $G$ is \index{general}{consistent matching}\emph{consistent} with $(G,\mathcal V)$ if $\V(\M)\subset \V$.
\end{definition}

\begin{definition}[\bf Configuration
$\mathbf{(\diamond10)}(\tilde\eps,d',\ell_1, \ell_2,
\eta')$]\index{mathsymbols}{**10@$\mathbf{(\diamond10)}$}
\label{def:CONF10}
Assume Setting~\ref{commonsetting}. The graph $G$ contains an $(
\tilde\epsilon, d', \ell_1, \ell_2)$-regularized graph $(\tilde G,\V)$  and there
is a $( \tilde\epsilon,  d',\ell_1 )$-\semiregular matching $\M$ consistent with
$(\tilde G,\V)$.
There are a family $\LargeTen\subset \V$ and distinct clusters $A, B\in\V$  with
\begin{enumerate}[(a)]
\item\label{diamond10cond1} $E(\tilde G[A,B])\neq \emptyset$, 
\item\label{diamond10cond2} $\deg_{\tilde G}\big(v,V(\M)\cup \bigcup \LargeTen\big)\ge (1+\eta')k$ for all but at most $\tilde\epsilon |A|$
vertices $v\in A$ and for all but at most $\tilde\epsilon|B|$ vertices $v\in B$, and
\item\label{diamond10cond3}
 for each $X\in\LargeTen$ we have $\deg_{\tilde G}(v)\ge (1+\eta')k$ for all but at most $\tilde\epsilon|X|$ vertices $v\in X$.
 \end{enumerate}
\end{definition}

\subsection{The main result}\label{sec:outerlemma}

We are now ready to state the main result of the present paper, Lemma~\ref{outerlemma}.
In the remaining part of the paper we build up the arguments that lead to the proof of  Lemma~\ref{outerlemma}, which is given in Section~\ref{ssec:proofofouterlemma}.
\begin{lemma}\label{outerlemma}
Suppose we are in Settings~\ref{commonsetting} and~\ref{settingsplitting}. Further suppose
that at least one of the followings hold in $G$.
\begin{itemize}
\item[{\bf(K1)}]$2e_G(\XA)+e_G(\XA, \XB)\ge \eta kn/3$,
\item[{\bf(K2)}]  $|V(\Mgood)|\ge \eta n/3$, 
\end{itemize}
where $\Mgood:=\{(A,B)\in \mathcal M_A\::\: A\cup B\subset \XA \}$.
 Then one of the configurations
\begin{itemize}
\item$\mathbf{(\diamond1)}$,
\item$\mathbf{(\diamond2)}\BUG{\left( \frac{\eta^{39}\Omega^{**}}{4\cdot
	10^{90}(\Omega^*)^{11}},\frac{\sqrt[4]{\Omega^{**}}}2,\frac{\eta^{13}\rho^2}{128\cdot
	10^{30}\cdot (\Omega^*)^5}\right)}$,
\item$\mathbf{(\diamond3)}\BUG{\left(\frac{\eta^{39}\Omega^{**}}{4\cdot
	10^{90}(\Omega^*)^{11}},\frac{\sqrt[4]{\Omega^{**}}}2,\frac\gamma2,\frac{\eta^{13}\gamma^2}{128\cdot
	10^{30}\cdot(\Omega^*)^5}\right)}$,
\item$\mathbf{(\diamond4)}\BUG{\left(\frac{\eta^{39}\Omega^{**}}{4\cdot
	10^{90}(\Omega^*)^{11}},\frac{\sqrt[4]{\Omega^{**}}}2,\frac\gamma2,\frac{\eta^{13}\gamma^3}{384\cdot
	10^{30}(\Omega^*)^6}\right)}$,
\item$\mathbf{(\diamond5)}\BUG{\left(\frac{\eta^{39}\Omega^{**}}{4\cdot
	10^{90}(\Omega^*)^{11}},
\frac{\sqrt[4]{\Omega^{**}}}2,\frac{\eta^{13}}{128\cdot
	10^{30}\cdot
	(\Omega^*)^3},\frac{\eta}2,\frac{\eta^{13}}{128\cdot
	10^{30}\cdot (\Omega^*)^4}\right)}$,
  \item
  $\mathbf{(\diamond6)}\big(\frac{\eta^3\rho^4}{10^{14}(\Omega^*)^4},4\epsilonD,\frac{\gamma^3\rho}{32\Omega^*},\frac{\eta^2\nu}{2\cdot10^4
 },\frac{3\eta^3}{2000},\proporce{2}(1+\frac\eta{20})k\big)$,
  \item $\mathbf{(\diamond7)}\big(\frac
{\eta^3\gamma^3\rho}{10^{12}(\Omega^*)^4},\frac
{\eta\gamma}{400},4\epsilonD,\frac{\gamma^3\rho}{32\Omega^*},\frac{\eta^2\nu}{2\cdot10^4
}, \frac{3\eta^3}{2\cdot 10^3}, \proporce{2}(1+\frac\eta{20})k\big)$,
  \item
$\mathbf{(\diamond8)}\big(\frac{\eta^4\gamma^4\rho}{10^{15}
(\Omega^*)^5},\frac{\eta\gamma}{400},\frac{400\epsilon}{\eta},4\epsilonD,\frac
d2,\frac{\gamma^3\rho}{32\Omega^*},\frac{\eta\pi\clustersize}{200k},\frac{\eta^2\nu}{2\cdot10^4
}, \proporce{1}(1+\frac\eta{20})k,\proporce{2}(1+\frac\eta{20})k\big)$,
 \item 
 $\mathbf{(\diamond9)}\big(\frac{\rho
\eta^8}{10^{27}(\Omega^*)^3},\frac
{2\eta^3}{10^3}, \proporce{1}(1+\frac{\eta}{40})k,
\proporce{2}(1+\frac{\eta}{20})k, \frac{400\varepsilon}{\eta},
\frac{d}2,
\frac{\eta\pi\clustersize}{200k},4\epsilonD,\frac{\gamma^3\rho}{32\Omega^*},
\frac{\eta^2\nu}{2\cdot10^4 }\big)$,
  \item $\mathbf{(\diamond10)}\big( \epsilon, \frac{\gamma^2
d}2,\pi\sqrt{\epsilon'}\nu k, \frac
{(\Omega^*)^2k}{\gamma^2},\frac\eta{40} \big)$
\end{itemize}
occurs in $G$.
\end{lemma}
\begin{remark}
The effect of changing the parameters $\proporce{1}$ and $\proporce{2}$ in Setting~\ref{settingsplitting} can be more substantial that a mere change of the parameters in one configuration asserted by Lemma~\ref{outerlemma}. That is, it may happen that for some values of $\proporce{1}$ and $\proporce{2}$ the only configuration that occurs in the graph $G_\PARAMETERPASSING{L}{outerlemma}$ is, say, $\mathbf{(\diamond6)}\big(\cdot,\cdot,\cdot,\cdot,\cdot,\proporce{2}(1+\frac\eta{20})k\big)$, while for other values of $\proporce{1}$ and $\proporce{2}$, the only configuration that occurs is, say, 
$\mathbf{(\diamond8)}\big(\cdot,\cdot,\cdot,\cdot,\cdot,\cdot,\cdot,\cdot, \proporce{1}(1+\frac\eta{20})k,\proporce{2}(1+\frac\eta{20})k\big)$.

Recall that $\proporce{1}$ and $\proporce{2}$ are set proportionally to the sizes of the internal- and end- shrubs of the tree $T_\PARAMETERPASSING{T}{thm:main}$, respectively. Thus the above tells us that different trees $T_\PARAMETERPASSING{T}{thm:main}$ may be embedded into different parts of $G_\PARAMETERPASSING{T}{thm:main}$, and using different embedding techniques.
\end{remark}
Note that it follows from the main results of our previous papers~\cite{cite:LKS-cut0,cite:LKS-cut1} that graphs from Theorem~\ref{thm:main} indeed satisfy the hypothesis of Lemma~\ref{outerlemma}. More specifically, after obtaining a sparse decomposition of $G_\PARAMETERPASSING{T}{thm:main}$ in \cite[Lemma~\ref{p0.lem:LKSsparseClass}]{cite:LKS-cut0}, we can apply \cite[Lemma~\ref{p1.prop:LKSstruct}]{cite:LKS-cut1} which asserts that {\bf(K1)} or {\bf(K2)} are fulfilled.

\section{Cleaning}\label{ssec:cleaning}
This section contains five ``cleaning lemmas''
(Lemma~\ref{lem:envelope}--\ref{lem:clean-Match}). The basic setting of
all these lemmas is the same. There is a system of vertex sets and some density assumptions on edges between certain sets of this
system. The assertion is that a small number of vertices can be discarded
from the sets so that some conditions on the minimum degree are fullfilled.
While the cleaning strategy is simply discarding the vertices which violate these
minimum degree conditions the analysis of the outcome is non-trivial. The simplest application of such an approach was the proof of Lemma~\ref{lem:clean-spots} above.

Lemmas~\ref{lem:envelope}--\ref{lem:clean-Match} are used to get the structures
required by (pre-)configurations introduced in Section~\ref{ssec:TypesConf}.

\bigskip
The first lemma will be used to obtain preconfiguration $\mathbf{(\clubsuit)}$
in certain situations.
\begin{lemma}\label{lem:envelope}\Referee{(E)}
Let $\psi\in (0,1)$, and $\Gamma,\Omega,\Omega'\ge 1$ be arbitrary, with 
\begin{equation}\label{eq:jaksestavidum}
\psi^3\Omega\ge 4\Gamma^2\Omega'\;.
\end{equation} Let $P$ and $Q$ be two
disjoint vertex sets in a graph $G$. Assume that $Y\subset V(G)$ is given. We
assume that \begin{align}\label{eq:envelopeAss1}\mindeg(P,Q)&\ge \Omega
k\;, \mbox{and}\\
\label{eq:52p}
\maxdeg(Q)&\le \Gamma k\;.
\end{align} Then there exist sets
$P'\subset P$, $Q''\subset Q'\subset Q\setminus Y$  such that the
following holds.
\begin{enumerate}[(a)]
\item \label{it5.1:a}$\maxdeg(Q',P\setminus P')< \psi k$,
\item \label{it5.1:b}$\maxdeg(Q'',Q\setminus (Q'\cup Y))<\psi k$, 
\item \label{it5.1:c}$\mindeg(P',Q')\ge \Omega' k$, and
\item\label{en:52e} $e(P',Q'')\ge (1-\psi)e(P,Q)-|Y\cap Q|\Gamma k$.
\end{enumerate}
\end{lemma}
\begin{proof}
Initially, set $P':=P$, $Q':=Q\setminus Y$, and $Q'':=Q'$. We shall
sequentially\footnote{No particular order is imposed on the vertices.} discard from the sets $P'$, $Q'$ and $Q''$ those vertices that
violate any of the properties (a)--(c). Further, if a vertex $v\in Q$ is removed
from $Q'$ then we remove it from the set $Q''$ as well. We thus have $Q''\subset Q'$ in each step. 
After this sequential cleaning procedure
finishes it only remains to establish~\eqref{en:52e}.

First, observe that the way we constructed $P'$ (together with~\eqref{eq:envelopeAss1}) ensures
that 
\begin{equation}\label{DianaAndMatej}
e(P\setminus P',Q'')\le e(P\setminus P',Q')\le \frac{\Omega'}{\Omega}e(P,Q)\;.
\end{equation}

Let $Q^{\ref{it5.1:a}}\subset Q$ be the set of the vertices removed from $Q'$ because of condition~\eqref{it5.1:a}.

Note that a vertex~$u$ of  $P^{\ref{it5.1:c}}=P\setminus P'$ was removed at some point from the set~$P'$ because~\eqref{it5.1:c} failed for~$u$. Let~$C'_u$ denote the set~$Q'$ just before this time.	 Let $f(u):=\deg(u, C_u')$. A vertex $v\in Q^{\ref{it5.1:a}}= Q\setminus (Q'\cup Y)$ was removed at some point from the set $Q'$ because~\eqref{it5.1:a} failed for~$v$. Let~$A'_v$ be the set $P'$ just before this time.
Let
$g(v):=\deg(v,P\setminus A'_v)$. Observe that $\sum_{u\in P^{\ref{it5.1:c}}}f(u)\ge
\sum_{v\in Q^{\ref{it5.1:a}}}g(v)$. Indeed, at the moment when $v\in Q$ is removed
from~$Q'$, the $g(v)$ edges that $v$ sends to the set $P\setminus A'_v$ are 
counted in $\sum_{u\in \neighbour(v)\cap P^{\ref{it5.1:c}}}f(u)$. Note also that we have $f(u)\le \Omega'k$ and $g(v)\ge \psi k$ for each $u\in P^{\ref{it5.1:c}}$ and each $v\in Q^{\ref{it5.1:a}}$, because~$u$ and~$v$ fail~\eqref{it5.1:c} and~\eqref{it5.1:a}, respectively. We therefore have
\begin{equation}\label{eq:okno1}
|P^{\ref{it5.1:c}}|\Omega'k
\ge
\sum_{u\in P^{\ref{it5.1:c}}}f(u)\ge
\sum_{v\in Q^{\ref{it5.1:a}}}g(v)
\ge |Q^{\ref{it5.1:a}}|\psi k\;.
\end{equation}
By~\eqref{eq:envelopeAss1} we have
\begin{equation}\label{eq:dvere1}
|P^{\ref{it5.1:c}}|\le \sum_{u\in P^{\ref{it5.1:c}}}\frac{\deg(u,Q)}{\Omega k} \le \frac{e(P,Q)}{\Omega k}\;.
\end{equation}
Putting~\eqref{eq:okno1} and~\eqref{eq:dvere1} together, we get that
\begin{equation}\label{eq:BoundB}
|Q^{\ref{it5.1:a}}|\le \frac{\Omega'}{\psi \Omega k}e(P,Q)\;.
\end{equation}
Because vertices in  $Q'\setminus Q''$ fail property~\eqref{it5.1:b} we have
\begin{align}
\begin{split}\label{eq:BoundC}
|Q'\setminus Q''|\psi k &
\le \sum_{w\in Q'\setminus Q''}\deg(w,Q\setminus (Q'\cup Y))
\le |Q\setminus (Q'\cup Y)|\Gamma k\\
&=|Q^{\ref{it5.1:a}}|\Gamma k\leByRef{eq:BoundB} \frac{\Gamma\Omega'}{\psi\Omega}e(P,Q)\;.
\end{split}
\end{align}
Finally, we can lower-bound $e(P',Q'')$ as follows.
\begin{align*}e(P',Q'')&\ge
e(P,Q)-e(P\setminus P',Q'')-|Y\cap Q|\Gamma k-|Q^{\ref{it5.1:a}}|\Gamma k-|Q'\setminus
Q''|\Gamma k\\
\JUSTIFY{by~\eqref{DianaAndMatej},~\eqref{eq:BoundB},~\eqref{eq:BoundC}}
&\ge
e(P,Q)\Big(1-\frac{\Omega'}{\Omega}-\frac{\Gamma\Omega'}{\psi\Omega}-\frac{\Gamma^2\Omega'}{\psi^2\Omega}\Big)-|Y\cap
Q|\Gamma k\\
\JUSTIFY{by~\eqref{eq:jaksestavidum}}&\ge(1-\psi)e(P,Q)-|Y\cap Q|\Gamma k \;.\end{align*}
\end{proof}

\bigskip
The purpose of the lemmas below
(Lemmas~\ref{lem:clean-C+yellow}--\ref{lem:clean-Match}) is to distill
vertex sets for configurations $\mathbf{(\diamond2)}$-$\mathbf{(\diamond10)}$.
They will be applied in Lemmas~\ref{lem:ConfWhenCXAXB},~\ref{lem:ConfWhenNOTCXAXB},~\ref{lem:ConfWhenMatching}.
This is the final ``cleaning step'' on our way to the proof of Theorem~\ref{thm:main} --- the outputs of these lemmas can by used for a vertex-by-vertex embedding of any tree $T\in\treeclass{k}$ (although the corresponding embedding procedures in~\cite{cite:LKS-cut3} are quite complex).

The first two of these cleaning lemmas (Lemmas~\ref{lem:clean-C+yellow}
and~\ref{lem:clean-C+black}) are suited when the set $\HugeVertices$ of
vertices of huge degrees (cf.\ Setting~\ref{commonsetting}) needs to be considered.

For the following lemma, recall that we defined $[r]$ as the set of the first $r$ natural numbers,  excluding~$0$.

\begin{lemma}\label{lem:clean-C+yellow}For all
$r,\Omega^*,\Omega^{**}\in \NN$, and
 $\delta,\gamma,\eta\in (0,1)$,   with
$\left(\frac{3\Omega^*}\gamma\right)^r\delta<\eta/10$,
and $\Omega^{**}>1000$ the
following holds. Suppose  there are vertex sets $X_0, X_1,\ldots,X_r$ and $Y$ of
an $n$-vertex graph $G$ such that
\begin{enumerate}
  \item\label{hyp:Ysmall} $|Y|<\eta n/(4\Omega^*)$,
  \item \label{hyp:C+y-edges}$e(X_0,X_1)\geq \eta kn$,
\item \label{hyp:C+y-large} $\mindeg(X_0,X_1)\ge \Omega^{**}k$,
  \item \label{hyp:C+y-deg}
  $\mindeg(X_i,X_{i+1})\ge \gamma k$ for all $i\in [r-1]$, and 
  \item \label{hyp:C+y-bounded} $\maxdeg\left(Y\cup\bigcup_{i\in [r]}X_i\right)\le
  \Omega^* k$.
\end{enumerate}
Then there are sets $X_i'\subseteq X_i$ for $i=0,1,\ldots,r$ such that
\begin{enumerate}[(a)]
  \item \label{conc:X1Ydisj} $X_1'\cap Y=\emptyset$,
  \item \label{conc:C+y-deg}  
  $\mindeg(X_i',X_{i-1}')\ge \delta k$ for all $i\in [r]$,
  \item \label{conc:C+y-avoid}  
  $\maxdeg(X_i',X_{i+1}\setminus X_{i+1}')<\gamma k/2$ for all $i\in [r-1]$,
  \item \label{conc:C+y-large} $\mindeg(X_0',X_1')\ge
  \sqrt{\Omega^{**}}k$, and
  \item \label{conc:C+u-edges} $e(X_0',X_1')\ge \eta kn/2$, in particular
  $X_0'\neq\emptyset$.
\end{enumerate}
\end{lemma}

\begin{proof}
In the formulae below we refer to hypotheses of the lemma
as~``\ref{hyp:Ysmall}.''--``\ref{hyp:C+y-bounded}.''.

Set $X_1':=X_1\setminus Y$. For $i=0,2,3,4,\ldots,r$, set
$X_i':=X_i$. Discard sequentially from $X_i'$ any vertex that
violates any of the
Properties~\eqref{conc:C+y-deg}--\eqref{conc:C+y-large}.
Properties~\eqref{conc:X1Ydisj}--\eqref{conc:C+y-large} are
trivially satisfied  when the procedure terminates. To show that
Property~\eqref{conc:C+u-edges} holds at this point, we bound the
number of edges from $e(X_0,X_1)$ that are incident with $X_0\setminus X_0'$ or
with $X_1\setminus X_1'$ in an amortized way.

For $i\in\{0,\ldots,r\}$ and for $v\in X_i\setminus X_i'$  we
write
\begin{align*}
f_i(v)&:=\deg\big(v,X_{i+1}\setminus X_{i+1}'(v)\big)\;,\\
g_i(v)&:=\deg\big(v,X_{i-1}'(v)\big)\;\mbox{, and}\\
h_i(v)&:=\deg\big(v,X_{i+1}'(v)\big)\;.
\end{align*}
where the sets $X_{i-1}'(v),X_{i}'(v),X_{i+1}'(v)$
above refer to the moment just before $v$ is removed from $X_i'$ (we do
not define $f_i(v)$ and $h_i(v)$ for $i=r$ and $g_i(v)$ for $i=0$).
 
For $i\in[r]$ let
$X_i^{\ref{conc:C+y-deg}}$ denote the vertices in 
 $X_i\setminus X_i'$ that were removed from $X_i'$ because of violating 
 Property~\eqref{conc:C+y-deg}. Then for a given $i\in [r]$ we have that 
 \begin{equation}\label{eq:X_i^a}
 \sum_{v\in X_i^{\ref{conc:C+y-deg}}}
 g_i(v)<\delta kn.
 \end{equation}
For $i=1,\ldots,r-1$
 let  $X_i^{\ref{conc:C+y-avoid}}$  denote
the vertices in $X_i\setminus X_i'$ that violated
Property~\eqref{conc:C+y-avoid}. Set $X_r^{\ref{conc:C+y-avoid}}:=\emptyset$.
\begin{figure}[t]
\centering 
\includegraphics{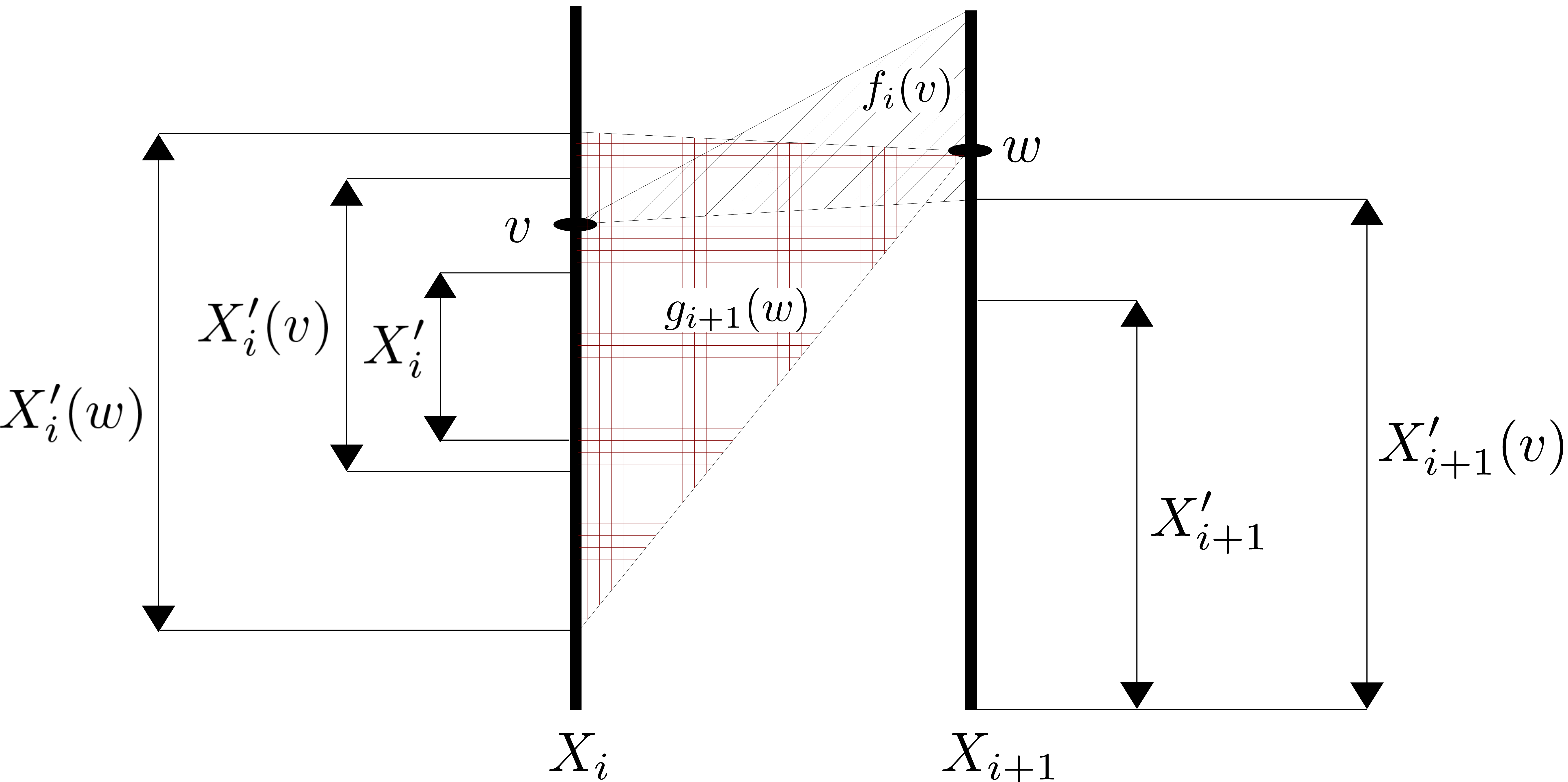}
\caption{Situation in~\eqref{eq:C+y-induction}. A summand from $\sum_{v\in
X_i^{\ref{conc:C+y-avoid}}}f_i(v)$ (corresponding edges hatched), and a summand from $\sum_{w\in
X_{i+1}\setminus
X_{i+1}'}g_{i+1}(w)$. 
Thus the former sum counts the number of edges $vw$ such that $v\in X_i^c$ and $w\in X_{i+1}\setminus X_{i+1}'(w)$. For each such pair $vw$ we have that $X_i'(v)\subseteq X_i'(w)$, as $w$ must have been removed from $X_{i+1}'$ prior to $v$ being removed from $X_i'$. Hence, the edge $vw$ is counted in $g_{i+1}(w)$ as well. 
Similar counting is used in~\eqref{eq:yel-induction} and in~\eqref{eq:Match-induction1}.}\label{fig:FiGi}
\end{figure}
For
a given $i\in [r-1]$ we have\Referee{(F1)}
\begin{equation}\label{eq:C+y-induction}
|X_i^{\ref{conc:C+y-avoid}}|\cdot \gamma k/2\le \sum_{v\in
X_i^{\ref{conc:C+y-avoid}}}f_i(v)\leBy{Fig~\ref{fig:FiGi}} 
\sum_{w\in
 X_{i+1} \setminus
X_{i+1}'}g_{i+1}(w)\;\overset{\ref{hyp:C+y-bounded}., \eqref{eq:X_i^a}}{<}\;\delta
kn+|X_{i+1}^{\ref{conc:C+y-avoid}}|\cdot \Omega^*k\;,
\end{equation}
as $X_i\setminus X_i'=X_i^{\ref{conc:C+y-deg}}\cup X_i^{\ref{conc:C+y-avoid}}$,
for $i=2,\ldots,r$. Using~\eqref{eq:C+y-induction}
for $j=0,\ldots,r-1$, we inductively deduce that
\begin{equation}\label{eq:Indk}
|X_{r-j}^{\ref{conc:C+y-avoid}}|\frac{\gamma
}2\le \sum_{i=0}^{j-1}\left(\frac{2\Omega^*}\gamma\right)^i\delta
n\;.
\end{equation}
(The left-hand side is zero for $j=0$.) The
bound~\eqref{eq:Indk} for $j=r-1$ gives
\begin{equation}\label{eq:B73}
|X_{1}^{\ref{conc:C+y-avoid}}|\le
\frac2\gamma\cdot\sum_{i=0}^{r-2}\left(\frac{2\Omega^*}\gamma\right)^i\delta
n\le\frac{2(2\Omega^*)^{r-1}}{\gamma^r}\delta n\;.
\end{equation}
Therefore,
 \begin{equation}\label{eq:e(YX_1^c,X_0)}
 e(X_0,Y\cup X_1^{\ref{conc:C+y-avoid}})\leq |Y\cup
 X_1^{\ref{conc:C+y-avoid}}|\cdot
 \Omega^*k \overset{\eqref{eq:B73},  
 \ref{hyp:Ysmall}.}{\le}\frac {\eta kn}{4}+\left(\frac
 {2\Omega^*}{\gamma}\right)^r\delta kn\;.
\end{equation}
For any vertex  $u\in X_0\setminus X_0'$ we
have $h_0(u)<\sqrt{\Omega^{**}}k$,
 and at the same time by Hypothesis~\ref{hyp:C+y-large}.\ we have
$\deg(u,X_1)\geq \Omega^{**}k$. So,
\begin{equation}\label{eq:h(X_0-X_0')}
\sum_{u\in
X_0\setminus X_0'}h_0(u)\le \frac{e(X_0,X_1)}{\sqrt{\Omega^{**}}}\;.
\end{equation} We have
\begin{align}\label{eq:contemplate}
e(X_0',X_1')&\ge e(X_0,X_1)-e(X_0,Y\cup
X_1^{\ref{conc:C+y-avoid}})-\sum_{u\in X_0\setminus
X_0'}h_0(u)-\sum_{v\in X_1^{\ref{conc:C+y-deg}}}g_1(v)\;.
\end{align}
(Consult Figure~\ref{fig:contemplate}.)
\begin{figure}[t]
	\centering 
	\includegraphics{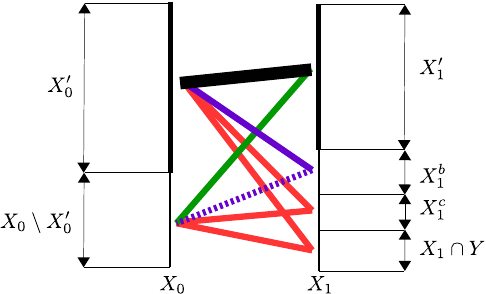}
\caption{The terms in~\eqref{eq:contemplate}. The term $e(X_0,Y\cup
X_1^{\ref{conc:C+y-avoid}})$ are shown in red, some edges of the term $\sum_{u\in X_0\setminus
X_0'}h_0(u)$ are shown in green (note that we undercount here, as the summands $h_0(u)$ reflect preliminary situations in the set $X_1'$). It is clear that each edge between $X_1^{\ref{conc:C+y-deg}}$ and $X_0'$ (solid blue) is counted in $\sum_{v\in X_1^{\ref{conc:C+y-deg}}}g_1(v)$. Consider now an edge $xv$, $x\in X_0\setminus X_0'$, $v\in X_1^{\ref{conc:C+y-deg}}$ (dashed blue). Suppose first that $x$ was removed from $X_0'$ before $v$ was put in $X_1^{\ref{conc:C+y-deg}}$. Then the edge $xv$ was counted in $\sum_{u\in X_0\setminus
X_0'}h_0(u)$. Suppose next that $v$ was put in $X_1^{\ref{conc:C+y-deg}}$ before $x$ was removed from $X_0'$. Then $xv$ was counted in $\sum_{v\in X_1^{\ref{conc:C+y-deg}}}g_1(v)$.}
 \label{fig:contemplate}
\end{figure}
 Therefore,
\begin{align*}
e(X_0',X_1')&\ge e(X_0,X_1)-e(X_0,Y\cup
X_1^{\ref{conc:C+y-avoid}})-\sum_{u\in X_0\setminus
X_0'}h_0(u)-\sum_{v\in X_1^{\ref{conc:C+y-deg}}}g_1(v)\\
\JUSTIFY{by \eqref{eq:X_i^a},
\eqref{eq:e(YX_1^c,X_0)}, \eqref{eq:h(X_0-X_0')}}
&\ge
e(X_0,X_1)-\frac {\eta kn}{4}-\left(\frac {2\Omega^*}\gamma\right)^r \delta kn-
\frac{e(X_0,X_1)}{\sqrt{\Omega^{**}}}-\delta kn\\
\JUSTIFY{by \ref{hyp:C+y-edges}.}&\ge  \eta kn/2\;,
\end{align*} \Referee{(12)}
proving Property~\eqref{conc:C+u-edges}.
\end{proof}

\begin{lemma}\label{lem:clean-C+black}
Let $\delta,\eta,\Omega^*,\Omega^{**},h>0$, let $G$ be  an $n$-vertex graph,
let $X_0, X_1, Y\subset V(G)$, and let $\mathcal C$ be a family of subsets of
$V(G)$  such that
\begin{enumerate}
\item\label{hypfburg}
$20(\delta+\frac2{\sqrt{\Omega^{**}}})<\eta$,
  \item \label{hyp:C+b-edges} $2 kn \geq e(X_0,X_1)\geq \eta kn$,
  \item \label{hyp:C+b-large} $\mindeg(X_0,X_1)\ge \Omega^{**}k$, 
  \item   \label{hyp:C+b-bounded} $\maxdeg(X_1)\le  \Omega^* k$,
  \item $|Y|<\eta n/(4\Omega^*)$, and\label{hypfburg5}
  \item $10h|\mathcal C|\Omega^*<\eta n$.\label{hypfburg6}
\end{enumerate} 
Then  there are sets $X_0'\subseteq X_0$ and $X_1'\subseteq X_1\setminus Y$ such that \begin{enumerate}[(a)]
   \item \label{conc:C+b-large} $\mindeg(X_0',X_1')\ge
   \sqrt{\Omega^{**}}k$,
  \item  \label{conc:C+b-deg} $\mindeg(X_1',X_0')\ge\delta k$, 
  \item \label{conc:C+b-cluster}for all $C\in \mathcal C$, either $X_1'\cap
  C=\emptyset$, or $|X_1'\cap C|\ge h$, and 
  \item \label{conc:C+b-edges}$e(X_0',X_1')\ge \eta kn/2$.
\end{enumerate}
\end{lemma}
\begin{proof}
Set $X_0':=X_0$ and $X_1':=X_1\setminus Y$. Discard sequentially from $X_0'$
any vertex violating
Property~(\ref{conc:C+b-large}). We discard from $X_1'$ any vertex
violating Property~(\ref{conc:C+b-deg}). Last, we discard from $X_1'$ all the vertices lying in any set $C\in
\mathcal C$ violating~(\ref{conc:C+b-cluster}). The deletions from $X_0'$, or $X_1'$ can take turns in an arbitrary order until no more are possible.
When the process ends, we verify Property~(\ref{conc:C+b-edges}) by bounding the
number of edges in $e(X_0,X_1)$ incident with $X_0\setminus X_0'$ or with
$X_1\setminus X_1'$. Given Assumption~\ref{hyp:C+b-edges}, and since
by Assumptions~\ref{hyp:C+b-bounded} and~\ref{hypfburg5} there are
at most $\frac14\eta kn$ edges incident with $Y\cap X_1$
it suffices to prove that
\begin{equation}\label{eq:ItSuff}
e(X_0, X_1)- e(X_0', X_1')- e(Y\cap X_1, X_0)<\frac{\eta
kn}4\;.
\end{equation}

Denote by $X_1^{\ref{conc:C+b-deg}}$ the set of vertices in $X_1\setminus
(Y\cup X_1')$ that violated Property~(\ref{conc:C+b-deg}), and  by
$X_1^{\ref{conc:C+b-cluster}}$ the set of vertices in $X_1\setminus (Y\cup X_1')$ that violated
Property~(\ref{conc:C+b-cluster}). Note that for each $C\in\mathcal C$, we have $|X_1^{\ref{conc:C+b-cluster}}\cap C|<h$, and thus
\begin{equation}\label{eq:Xch}
|X_1^{\ref{conc:C+b-cluster}}|\le h |\mathcal C|\;.
\end{equation}
For a vertex $v\in X_1\setminus (Y\cup X_1')$, let $g(v)$
denote\Referee{(F2)} $\deg(v,X_0'(v))$, where $X_0'(v)$ denotes the set $X_0'$ just before $v$ is removed from
$X_1'$. Analogously we define $f(v)$, for $v\in X_0\setminus X_0'$, as
$\deg(v,X_1'(v))$ where the set $X_1'(v)$ denotes the set $X_1'$ just before~$v$ is removed from~$X_1'$. We have $\sum_{v\in X_1^{\ref{conc:C+b-deg}}}g(v)<\delta kn$, \Referee{(13)}\Referee{(14)}
\begin{align*}
\sum_{v\in X_1^{\ref{conc:C+b-cluster}}}g(v)\leBy{\ref{hyp:C+b-bounded}.}  |X_1^{\ref{conc:C+b-cluster}}| \Omega^* k\leByRef{eq:Xch}
h|\mathcal C|\cdot  \Omega^* k\text{, and}\\
\sum_{v\in X_0\setminus X_0'}f(v)\leBy{\ref{hyp:C+b-large}.}
\frac{e(X_0,X_1)}{\sqrt{\Omega^{**}}}\leBy{\ref{hyp:C+b-edges}.}
\frac2{\sqrt{\Omega^{**}}}kn\;.
\end{align*}
Thus,
\begin{align*}
e(X_0, X_1)&- e(X_0', X_1')- e(Y\cap X_1, X_0)\\&=\sum_{v\in
X_1^{\ref{conc:C+b-deg}}}g(v)+\sum_{v\in X_1^{\ref{conc:C+b-cluster}}}g(v)+\sum_{v\in X_0\setminus
X_0'}f(v)\\
&<\big(\delta+\frac2{\sqrt{\Omega^{**}}}\big)kn+h|\mathcal
C|\Omega^*k\\
\JUSTIFY{by~\ref{hypfburg}.\ and \ref{hypfburg6}.}&<\frac{\eta kn}4\;.
\end{align*}
establishing~\eqref{eq:ItSuff}.
\end{proof}

The next two lemmas (Lemmas~\ref{lem:clean-yellow} and~\ref{lem:clean-Match})
deal with cleaning outside the set of huge degree vertices $\HugeVertices$.

\begin{lemma}\label{lem:clean-yellow} For all $r, \Omega\in \NN$,
$r\ge 2$ and all $\gamma,\delta,\eta>0$ such that
\begin{equation}\label{eq:condCY}
 \left(\frac{8\Omega}\gamma\right)^r\delta\le\frac\eta{10}
\end{equation}
 the following holds. Suppose there are
vertex sets $Y, X_0, X_1,\ldots,X_r\subset V$, where $V$ is a set of $n$
vertices. Suppose that edge sets $E_1,\ldots,E_r$ are given on $V$.
The expressions $\deg_i$, $\maxdeg_i$,
$\mindeg_i$, and $e_i$ below refer to the edge set $E_i$. Suppose that the following properties are fulfilled
\begin{enumerate}
  \item \label{hyp:yel-Ysmall}$|Y|<\delta n$,
  \item \label{hyp:yel-edges}$e_1(X_0,X_1)\geq \eta kn$,
  \item \label{hyp:yel-deg}for all $i\in [r-1]$ we have $\mindeg_{i+1}( X_i\setminus
  Y,X_{i+1})\ge \gamma k$, 
  \item \label{hyp:yel-bounded} for all $i\in\{0,\ldots,r-1\}$, we have
  $\maxdeg_{i+1}(X_{i})\le \Omega k$, and
  $\maxdeg_{i+1}(X_{i+1})\le \Omega k$.
\end{enumerate}
Then there are sets $X_i'\subseteq X_i\setminus Y$ ($i=0,\ldots,r$) satisfying
the following.
\begin{enumerate}[(a)]
  \item \label{conc:yel-deg}For all $i\in [r]$ and we have 
  $\mindeg_{i}(X_{i}',X_{i-1}')\ge \delta k$,
  \item \label{conc:yel-avoid}for all $i\in [r-1]$ we have
  $\maxdeg_{i+1}(X_i',X_{i+1}\setminus X_{i+1}')<\gamma k/2$,
    \item \label{conc:yel-X0X1}$\mindeg_1(X_0',X'_1)\ge \delta k$, and
    \item \label{conc:yel-edges} $e_1(X'_0,X'_1)\ge \eta kn/2$
\end{enumerate}
\end{lemma}
\begin{proof}
We proceed similarly as in the proof of Lemma~\ref{lem:clean-C+yellow}.
Set $X_i':=X_i\setminus Y$ for each $i=0,\ldots,r$.
Discard sequentially from $X_i'$ any vertex that violates
Property~(\ref{conc:yel-deg}) or~(\ref{conc:yel-avoid}), or~(\ref{conc:yel-X0X1}).
When the procedure terminates, we certainly have
that~(\ref{conc:yel-deg})--(\ref{conc:yel-X0X1}) hold. We then show that
Property~(\ref{conc:yel-edges}) holds by bounding the number of edges from
$e_1(X_0,X_1)$ that are incident with $X_0\setminus X_0'$ or with $X_1\setminus
X_1'$. For $i\in\{0,\ldots,r\}$ and
for $v\in X_i\setminus X_i'$ we write\Referee{(F2)}
\begin{align*}
f_{i+1}(v)&:=\deg_{i+1}(v,X_{i+1}\setminus X_{i+1}'(v))\;,\\
g_{i}(v)&:=\deg_i(v,X_{i-1}'(v))\;\mbox{, and}\\
h(v)&:=\deg_1(v,X_{1}'(v))\;,
\end{align*}
where the sets $X_1'(v), X_{i-1}'(v)$ and $X_{i+1}'(v)$
above refer to the sets $X_1'$, $X_{i-1}'$, and $X_{i+1}'$, respectively, at the moment\footnote{if $v\in Y$ then this
moment is the zero-th step} just before~$v$ is removed from $X_i'$  (we do not define $f_{i+1}(v)$ for
$i=r$ and $g_i(v)$ for $i=0$).

Let $X^{\ref{conc:yel-deg}}_i\subset X_i$, $X^{\ref{conc:yel-avoid}}_i\subset
X_i$ for $i\in [r-1]$ be the sets of vertices
removed from $X'_i$ because of Property~(\ref{conc:yel-deg})
and~(\ref{conc:yel-avoid}),\Referee{(15)} respectively. Set
$X^{\ref{conc:yel-deg}}_r:=X_r\setminus X_r'$ and
$X^{\ref{conc:yel-X0X1}}_0:=X_0\setminus X'_0$.
We have for each $i\in[r]$,
\begin{align}\label{eq:HezkejVodotrysk}
\sum_{v\in X_i^{\ref{conc:yel-deg}}}g_i(v)&< \delta kn\;.
\end{align}
Also, note that we have
\begin{equation}\label{eq:X0X1}
\sum_{v\in X_0^{\ref{conc:yel-X0X1}}}h(v)\le \delta kn\;.
\end{equation}

We set $X_r^{\ref{conc:yel-avoid}}:=\emptyset$. For a given $i\in [r-1]$ we have
\begin{align}\nonumber
|X_i^{\ref{conc:yel-avoid}}|\cdot \frac{\gamma k}2&\le \sum_{v\in
X_i^{\ref{conc:yel-avoid}}}f_{i+1}(v)\\
\nonumber
\JUSTIFY{see Figure~\ref{fig:FiGi}}&\le \sum_{v\in
X_{i+1}\setminus
X_{i+1}'}g_{i+1}(v)\\ 
\JUSTIFY{by~\ref{hyp:yel-bounded}.,
\eqref{eq:HezkejVodotrysk}}&\leq \delta
kn+|X_{i+1}^{\ref{conc:yel-avoid}}| \Omega k\;,\label{eq:yel-induction}
\end{align}
as $X_i\setminus X_i'\subset X_i^{\ref{conc:yel-deg}}\cup
X_i^{\ref{conc:yel-avoid}}\cup Y$, for $i=2,\ldots,r$. Using~\eqref{eq:yel-induction}, we deduce inductively that
\begin{equation}\label{eq:Indk1}
\left|X_{r-j}^{\ref{conc:yel-avoid}}\right|\le \left(\frac{8\Omega}\gamma\right)^j\delta
n\;,
\end{equation}
for $j=0,\ldots,r-1$.
(The left-hand side is zero for $j=0$.)
Therefore,
 \begin{align*}
e_1(X_0',X_1')&\geq
 e_1(X_0,X_1)-(|Y|+|X_1^{\ref{conc:yel-avoid}}|)\Omega k-\sum_{v\in
 X_1^{\ref{conc:yel-deg}}}g_1(v)-\sum_{v\in X^{\ref{conc:yel-X0X1}}_0}h(v)\\
 \JUSTIFY{by~\ref{hyp:yel-edges},
 \eqref{eq:Indk1},~\eqref{eq:HezkejVodotrysk},~\eqref{eq:X0X1}}&\ge \eta
 kn-\left(\frac{8\Omega}{\gamma}\right)^r\delta kn-2\delta kn\\
 &\ge \frac{\eta}2kn\;,
\end{align*}
establishing Property~(\ref{conc:yel-edges}).
\end{proof}

\begin{lemma}\label{lem:clean-Match}
For all $r, \Omega\in \NN$, $r\ge 2$
and all $\gamma,\eta,\delta,\epsilon,\mu,d>0$ with
\begin{equation}\label{eq:condCYmatch}
\text{ $20\epsilon<d$ \ and }\  \left(\frac{8\Omega}\gamma\right)^r\delta\le\frac\eta{30}
\end{equation}
the following holds. Suppose there are
vertex sets $Y, X_0, X_1,\ldots,X_r\subset V$, where $V$ is a set of
$n$ vertices. Let $P^{(1)}_i,\ldots,P^{(p)}_i$ partition $X_i$, for $i=0,1$.
Suppose that edge sets $E_1,E_2,E_3,\ldots,E_r$ are given on $V$.
The expressions $\deg_i$, $\maxdeg_i$, and $\mindeg_i$ below refer to the edge set $E_i$. Suppose that 
\begin{enumerate}
  \item \label{hyp:Match-Ysmall}$|Y|<\delta n$,
  \item \label{hyp:Match-edges}$|X_1|\geq \eta n$,
  \item \label{hyp:Match-deg}for all $i\in [r-1]$ we have
  $\mindeg_{i+1}(X_i\setminus Y,X_{i+1})\ge \gamma k$,
  \item \label{hyp:Match-reg} the family
  $\left\{(P^{(j)}_0,P^{(j)}_1)\right\}_{j\in[p]}$ is an $(\epsilon,d,\mu k)$-\semiregular matching with respect to the edge set $E_1$, and
  \item \label{hyp:Match-bounded}for all $i\in\{0,\ldots,r-1\}$,
  $\maxdeg_{i+1}(X_{i+1})\le \Omega k$, and (when $i\not=r$) $\maxdeg_{i+1}(X_i)\le \Omega k$.
\end{enumerate}
Then there exists a non-empty family $\mathcal Y\subset [p]$\Referee{(16)} and a family $\{(Q_0^{(j)},Q_1^{(j)})\}_{j\in\mathcal Y}$
of vertex-disjoint $(4\epsilon,\frac d4)$-super-regular pairs with respect to $E_1$, with
\begin{enumerate}[(a)]
 \item \label{conc:Match-superreg}  $|Q_0^{(j)}|,|Q_1^{(j)}|\ge \frac{\mu k}2$ for each  $j\in\mathcal Y$,
\end{enumerate}
 and sets $X_0':=\bigcup Q_0^{(j)}\subset X_0\setminus Y$, $X_1':=\bigcup Q_1^{(j)}\subset X_1\setminus Y$, $X_i'\subseteq X_i\setminus Y$ ($i=2,\ldots,r$) such that
\begin{enumerate}[(a)]
  \setcounter{enumi}{1}
  \item \label{conc:Match-deg}for all $i\in [r-1]$ we have 
  $\mindeg_{i+1}(X_{i+1}',X_{i}')\ge \delta k$, and
  \item \label{conc:Match-avoid}for all $i\in [r-1]$, we have
  $\maxdeg_{i+1}(X_i',X_{i+1}\setminus X_{i+1}')<\gamma k/2$.
\end{enumerate}
\end{lemma}
\begin{proof}
Initially, set $\mathcal{J}:=\emptyset$ and $X_i':=X_i\setminus Y$ for each $i=0,\ldots,r$. Discard sequentially from
$X_i'$ any vertex that violates any of the Properties~(\ref{conc:Match-deg}) or~(\ref{conc:Match-avoid}).
We would like to keep track of these vertices and therefore we call $X^b_i, X^c_i\subset X_i$ the sets
of vertices removed from $X'_i$ because of
Property~(\ref{conc:Match-deg}), and~(\ref{conc:Match-avoid}),
respectively.
Further, for $i=0,1$ and for $j\in[p]$
remove any vertex $v\in X'_i\cap P^{(j)}_i$ from $X'_i$ if
\begin{equation}\label{eq:pocitac}
\deg_1(v,X'_{1-i}\cap P^{(j)}_{1-i})\le \frac{d|P_{1-i}^{(j)}|}4\;. 
\end{equation}
For $i=0,1$, let $X_i^a$ be the set of those vertices of $X_i$ that were removed because of~\eqref{eq:pocitac}.

If for some $j\in[p]$ we have $|P^{(j)}_0\cap Y|>\frac{|P^{(j)}_0|}4$ or $|P^{(j)}_1\cap (Y\cup X^c_1)|>\frac{|P^{(j)}_1|}4$ we remove simultaneously the sets $P^{(j)}_0$ and $P^{(j)}_1$ entirely from $X_0'$ and $X_1'$, i.e., we set $X'_0:=X'_0\setminus P^{(j)}_0$ and $X'_1:=X'_1\setminus P^{(j)}_1$. We also add the index $j$ to the set $\mathcal{J}$ in this case.

When the procedure terminates define $\mathcal Y:=[p]\setminus \mathcal J$, and for $j\in\mathcal Y$ set $(Q_0^{(j)},Q_1^{(j)}):=(P_0^{(j)}\cap X_0',P_1^{(j)}\cap X_1')$. The sets $X_i'$ obviously satisfy Properties (\ref{conc:Match-deg})--(\ref{conc:Match-avoid}). We now turn to verifying Property~(\ref{conc:Match-superreg}). This relies on the following claim.
\begin{claim}\label{cl:useReg}
If $j\in[p]\sm \mathcal J$ then
$|P^{(j)}_0\cap X^a_0|\le\frac{|P^{(j)}_0|}4$ and $|P^{(j)}_1\cap X^a_1|\le\frac{|P^{(j)}_1|}4$.
\end{claim}
\begin{proof}[Proof of Claim~\ref{cl:useReg}]
Recall that $E_1$ is the relevant underlying edge set when working with the
pairs $(P_0^{(j)},P_1^{(j)})$.  Also, recall that only vertices from $Y\cup
X^a_0$ were removed from $P_0^{(j)}$ and only vertices from $Y\cup X^a_1\cup
X^c_1$ were removed from $P_1^{(j)}$.

Since $j\notin \mathcal J$, the pair $(P_0^{(j)}\setminus Y,P_1^{(j)}\setminus (Y\cup X_1^c))$ is $2\epsilon$-regular of density at least $0.9d$ by Fact~\ref{fact:BigSubpairsInRegularPairs}. Let 
\begin{align*}
K_0&:=\big\{v\in P_0^{(j)}\setminus Y\::\:\deg_1(v,P_1^{(j)}\setminus (Y\cup X_1^c))<0.8d|P_1^{(j)}\setminus (Y\cup X_1^b)|\big\}\;\mbox{, and}\\
K_1&:=\big\{v\in P_1^{(j)}\setminus (Y\cup X_1^c)\::\:\deg_1(v,P_0^{(j)}\setminus Y )<0.8d|P_0^{(j)}\setminus Y|\big\}\;.
\end{align*}
By Fact~\ref{fact:manyTypicalVertices}, we have $|K_0|\le 2\epsilon
|P_0^{(j)}\setminus Y|\le 0.1d|P_0^{(j)}|$ and $|K_1|\le 0.1 d|P_1^{(j)}|$. In particular, we have
\begin{align}
\begin{split}\label{eq:WO1}
\mindeg_1(P^{(j)}_0\setminus (Y\cup K_0),P^{(j)}_1\setminus (Y\cup X^c_1\cup K_1))&\ge 0.8 d|P_1^{(j)}\setminus (Y\cup X_1^c)| -|K_1|\\
&\ge 0.8 d\cdot 0.75 |P_1^{(j)}|-0.1d|P_1^{(j)}|\\
&>0.25d|P_1^{(j)}|\;,\mbox{and}
\end{split}\\
\begin{split}\label{eq:WO2}
\mindeg_1(P^{(j)}_1\setminus (Y\cup X^c_1\cup K_1),P^{(j)}_0\setminus (Y\cup
K_0))&\ge 0.8 d|P_0^{(j)}\setminus Y| -|K_0|\\&\ge 0.8 d\cdot 0.75
|P_0^{(j)}|-0.1d|P_0^{(j)}|\\
&>0.25d|P_0^{(j)}|\;.
\end{split}
\end{align}
Then~\eqref{eq:WO1} and~\eqref{eq:WO2} allow us to prove that $P^{(j)}_i\cap X^a_i\subset K_i$ for $i=0,1$. Indeed, assume inductively that $P^{(j)}_i\cap X^a_i\subset K_i$ for $i=0,1$ throughout the cleaning process until a certain step. Then~\eqref{eq:WO1} and~\eqref{eq:WO2} assert that no vertex outside of $P^{(j)}_0\setminus (Y\cup K_0)$ or of $P^{(j)}_1\setminus (Y\cup X^c_1\cup K_1)$ can be removed because of~\eqref{eq:pocitac}, proving the induction step.
The claim follows.
\end{proof}
Putting together the definition of $\mathcal J$ (through which one controls the size of $P_i^{(j)}\cap (Y\cup X_i^c)$) and Claim~\ref{cl:useReg} (which controls the size of $P_i^{(j)}\cap X_i^a$) we get for each $j\in\mathcal Y$ and $i=0,1$, $$|Q_i^{(j)}|\ge\frac{|P_i^{(j)}|}2\ge \frac{\mu k}2\;.$$
Therefore, these pairs are $4\epsilon$-regular (cf.\ Fact~\ref{fact:BigSubpairsInRegularPairs}). We get the property of $(4\epsilon,\frac d4)$-super-regularity from the definition of $X_i^c$ (cf.~\eqref{eq:pocitac}). Thus, the pairs $(Q_0^{(j)},Q_1^{(j)})$ are as required for Lemma~\ref{lem:clean-Match} and satisfy its Property~(\ref{conc:Match-superreg}).

\medskip

The only thing we have to prove is that the set $X_1'$ is nonempty. By the definition, for each $j\in \mathcal J$, we either have 
$|P^{(j)}_1| \le 4(|(Y\cup X_1^c)\cap P^{(j)}_1|)$ or $|P^{(j)}_0| \le 4|Y\cap P^{(j)}_0|$. We use that $|P^{(j)}_0|=|P^{(j)}_1|$ to see that
\begin{equation}\label{eq:casak}
\left|\bigcup_{\mathcal J} P^{(j)}_1\right| \leq 4(|Y|+|X_1^c|) \;.
\end{equation}

For $i\in\{1,\ldots,r\}$ and for $v\in X_i\setminus X_i'$ write\Referee{(F3)}
\begin{align*}
f_{i+1}(v)&:=\deg_{i+1}(v,X_{i+1}\setminus X_{i+1}'(v))\;\mbox{, and}\\
g_{i}(v)&:=\deg_i(v,X_{i-1}'(v))\;. 
\end{align*}
where the sets $X_1'(v), X_{i-1}(v)'$ and $X_{i+1}'(v)$
above refer to the sets $x-1'$, $X_{i-1}'$, and $X_{i+1}'$, respectively, at the  moment\footnote{if $v\in Y$ then this
moment is the zero-th step} just before~$v$ is removed from $X_i'$ (we do not define
$f_{i+1}(v)$ for $i=r$).

Observe that for each $i\in\{2,\ldots,r\}$, we have
\begin{align}\label{eq:HezkejVodotrysk1}
\sum_{v\in X_i^b}g_i(v)&< \delta kn\;.
\end{align}

We set $X_r^c:=\emptyset$. For a given $i\in [r-1]$ we have
\begin{align}\nonumber
|X_i^{\ref{conc:Match-avoid}}|\cdot \frac{\gamma k}2&\le \sum_{v\in
X_i^{\ref{conc:Match-avoid}}}f_{i+1}(v)\\
\nonumber
\JUSTIFY{see Figure~\ref{fig:FiGi}} & \le \sum_{v\in
X_{i+1}\setminus 
X_{i+1}'}g_{i+1}(v)
\\ \label{eq:Match-induction1}
\JUSTIFY{by~\ref{hyp:Match-Ysmall}.\ ,\ref{hyp:Match-bounded}.\ ,
\eqref{eq:HezkejVodotrysk1}} & <\delta
kn+|X_{i+1}^{\ref{conc:Match-avoid}}| \Omega k,
\end{align}
as $X_i\setminus X_i'\subset X_i^b\cup X_i^{\ref{conc:Match-avoid}}\cup Y$,
for $i=2,\ldots,r$. Using~\eqref{eq:Match-induction1}, we deduce inductively that
$|X_{r-j}^{\ref{conc:Match-avoid}}|\le \left(\frac{8\Omega}\gamma\right)^j\delta
n$ for $j=1,2,\ldots,r-1$, and in particular that
\begin{equation}\label{eq:VX1B}
|X_{1}^{\ref{conc:Match-avoid}}|\le
\left(\frac{8\Omega}\gamma\right)^{r-1}\delta n\;.
\end{equation}

As $X_1^a=\emptyset$, we obtain that
\begin{align*}
|X'_1 | & =\left|X_1\setminus \left(\bigcup_{j\in\mathcal J}P^{(j)}_1\cup
\bigcup_{j\in\mathcal Y}\big(P^{(j)}_1\cap (Y\cup X_1^a\cup
X_1^c)\big)\right)\right|\\ 
\JUSTIFY{by~\eqref{eq:casak}}& \geq |X_1|-4(|Y|+|X_1^c|) -
\left|\bigcup_{j\in\mathcal Y}\left(P^{(j)}_1\cap X_1^a
\right)\right|\\ 
\JUSTIFY{by~\ref{hyp:Match-Ysmall}., \eqref{eq:condCYmatch}, \eqref{eq:VX1B}} &
\ge |X_1|-\frac{\eta n}2 - \left|\bigcup_{j\in\mathcal
Y}(P^{(j)}_1\cap X_1^a )\right|\\ \JUSTIFY{by Cl~\ref{cl:useReg}}  & \ge
|X_1|-\frac{\eta n}2 - \frac{|X_1|}4\\ \JUSTIFY{by \ref{hyp:Match-edges}.}& >0,
\end{align*} 
as desired.
\end{proof}

\section{Obtaining a configuration}\label{ssec:obtainingConf}
In this section we prove that the structure in the graph
$G\in\LKSgraphs{n}{k}{\eta}$ guaranteed by the main results of~\cite{cite:LKS-cut0,cite:LKS-cut1}
always leads to one of the configurations
$\mathbf{(\diamond1)}$--$\mathbf{(\diamond10)}$, as promised in Lemma~\ref{outerlemma}. We distinguish two cases. When the set  $\HugeVertices$ of
 vertices of huge degree (coming from a sparse decomposition
of $G$) sees many edges, then one of the configurations $\mathbf{(\diamond1)}$--$\mathbf{(\diamond5)}$ must
occur (cf.  Lemma~\ref{lem:ConfWhenCXAXB}). Otherwise, when the edges incident with $\HugeVertices$ can be neglected,
we obtain one of the configurations
$\mathbf{(\diamond6)}$--$\mathbf{(\diamond10)}$ 
(cf. Lemmas~\ref{lem:ConfWhenNOTCXAXB} and~\ref{lem:ConfWhenMatching}). 

Lemmas~\ref{lem:ConfWhenCXAXB},~\ref{lem:ConfWhenNOTCXAXB},
and~\ref{lem:ConfWhenMatching} are stated in the first subsection of this section, and their proofs
occupy
Sections~\ref{sssec:ProofConfWhenCXAXB},~\ref{sssec:ProofConfWhenNOTCXAXB},
and~\ref{sssec:ProofConfWhenMatching}, respectively. The proof of Lemma~\ref{outerlemma} is in Section~\ref{ssec:proofofouterlemma}. 

\subsection{Statements of the auxiliary lemmas}\label{sssec:StatementsofResultsObtaining}

The proof of the main result of this paper, Lemma~\ref{outerlemma}, relies
on Lemmas~\ref{lem:ConfWhenCXAXB},~\ref{lem:ConfWhenNOTCXAXB} and~\ref{lem:ConfWhenMatching} below. For an input graph $G_\PARAMETERPASSING{L}{outerlemma}$ one of these lemmas is applied depending on the majority type of ``good'' edges in $G_\PARAMETERPASSING{L}{outerlemma}$. Observe that {\bf(K1)} of \cite[Lemma~\ref{p1.prop:LKSstruct}]{cite:LKS-cut1} guarantees edges between $\HugeVertices$ and $\XA\cup \XB$, or between $\XA$ and $\XA\cup \XB$ either in $E(\Gexp)$ or in $E(\GD)$. Lemma~\ref{lem:ConfWhenCXAXB} is used if we find edges between $\HugeVertices$ and $\XA\cup \XB$. Lemma~\ref{lem:ConfWhenNOTCXAXB} is used  if we find edges of $E(\Gexp)$  between $\XA$ and $\XA\cup \XB$. The remaining case can be reduced to the setting of Lemma~\ref{lem:ConfWhenMatching}.  Lemma~\ref{lem:ConfWhenMatching} is also used to obtain a configuration if we are in case {\bf(K2)} of \cite[Lemma~\ref{p1.prop:LKSstruct}]{cite:LKS-cut1}.

\begin{lem}\label{lem:ConfWhenCXAXB}
Suppose we are in Setting~\ref{commonsetting}. Assume that
\begin{align}
\label{libelle7.38b}
e_{\Gcapt}(\HugeVertices,\XA\cup\XB)&\ge \frac{\eta^{13}
kn}{10^{28}(\Omega^*)^3}.
\end{align}
Then $G$ contains at least one of the configurations 
\begin{itemize}
\item$\mathbf{(\diamond1)}$,
\item$\mathbf{(\diamond2)}\BUG{\left( \frac{\eta^{39}\Omega^{**}}{4\cdot
10^{90}(\Omega^*)^{11}},\frac{\sqrt[4]{\Omega^{**}}}2,\frac{\eta^{13}\rho^2}{128\cdot
10^{30}\cdot (\Omega^*)^5}\right)}$,
\item$\mathbf{(\diamond3)}\BUG{\left(\frac{\eta^{39}\Omega^{**}}{4\cdot
10^{90}(\Omega^*)^{11}},\frac{\sqrt[4]{\Omega^{**}}}2,\frac\gamma2,\frac{\eta^{13}\gamma^2}{128\cdot
10^{30}\cdot(\Omega^*)^5}\right)}$,
\item$\mathbf{(\diamond4)}\BUG{\left(\frac{\eta^{39}\Omega^{**}}{4\cdot
10^{90}(\Omega^*)^{11}},\frac{\sqrt[4]{\Omega^{**}}}2,\frac\gamma2,\frac{\eta^{13}\gamma^3}{384\cdot
10^{30}(\Omega^*)^6}\right)}$, or 
\item$\mathbf{(\diamond5)}\BUG{\left(\frac{\eta^{39}\Omega^{**}}{4\cdot
10^{90}(\Omega^*)^{11}},
\frac{\sqrt[4]{\Omega^{**}}}2,\frac{\eta^{13}}{128\cdot
10^{30}\cdot
(\Omega^*)^3},\frac{\eta}2,\frac{\eta^{13}}{128\cdot
10^{30}\cdot (\Omega^*)^4}\right)}$.
\end{itemize}
\end{lem}

\smallskip

\begin{lem}\label{lem:ConfWhenNOTCXAXB}
Suppose that we are in Setting~\ref{commonsetting} and
Setting~\ref{settingsplitting}. If there exist two disjoint sets $\YA_1,\YA_2\subset V(G)$ such that
\begin{align}
\label{eq:manyXAXAXBobt}
e_{\Gexp}(\YA_1,\YA_2)&\ge 2\rho kn\;,
\end{align}
and either
\begin{align}
\label{2ndcondiObt2}
\YA_1\cup \YA_2\subset \XA\colouringpI{0}\setminus (\gP\cup \exceptVertSplit\cup
\shadowsplit)&\mbox{, or}\\
\label{3rdcondiObt2}
\YA_1\subset \XA\colouringpI{0}\setminus (\gP\cup \exceptVertSplit\cup
\shadowsplit\cup \gP_2\cup \gP_3)&\mbox{, and }\YA_2\subset
\XB\colouringpI{0}\setminus (\gP\cup \exceptVertSplit\cup \shadowsplit)\;,
\end{align}
then $G$ has
configuration~$\mathbf{(\diamond6)}(\frac{\eta^3\rho^4}{10^{14}(\Omega^*)^3},0,1,1,
\frac {3\eta^3}{2\cdot 10^3},\proporce{2}\left(1+\frac{\eta}{20}\right)k)$.
\end{lem} 

\smallskip

\begin{lem}\label{lem:ConfWhenMatching}
Suppose that we are in Setting~\ref{commonsetting} and
Setting~\ref{settingsplitting}. Let $\DenseSpots_\class$ be as in Lemma~\ref{lem:clean-spots}. Suppose that there exists an
$(\bar\epsilon,\bar d,\beta k)$-\semiregular matching $\M$,  with $V(\M)\subset \colouringp{0}$, $|V(\M)|\ge \frac {\rho n}{\Omega^*}$, and fulfilling one of the following two properties.
\begin{itemize}
 \item[{\bf(M1)}] $\M$ is absorbed by $\Mgood$, $\bar\epsilon:=\frac{10^5\epsilon'}{\eta^2}$, $\bar d:=\frac{\gamma^2}4$, and $\beta:=\frac
{\eta^2\clustersize}{8\cdot 10^3 k}$.
 \item[{\bf(M2)}] $E(\M)\subset E(\DenseSpots_\class)$, $\M$ is absorbed by $\DenseSpots_\class$, $\bar\epsilon:=\epsilonD$, $\bar d:=\frac{\gamma^3\rho}{32\Omega^*}$, and $\beta:=\frac{\alphaD\rho}{\Omega^*}$.
\end{itemize}
Suppose further that one of the following occurs.
\begin{itemize}
  \item [$\mathbf{(cA)}$] $V(\M)\subset \XA\colouringpI{0}\setminus
  (\gP\cup\exceptVertSplit\cup\shadowsplit)$, and we have for the set
$$R:=\shadow_{\Gcapt}\left((\largeintoatoms\cap\largevertices{\eta}{k}{G})\setminus V(\M_A\cup\M_B),\frac{2\eta^2 k}{10^5}\right)$$
one of the following
\begin{itemize}
 \item[{\bf(t1)}] $V_1(\M)\subset \shadow_{\Gcapt}(V(\Gexp),\rho k)$,
 \item[{\bf(t2)}]$V_1(\M)\subset \largeintoatoms$,
 \item[{\bf(t3)}]$V_1(\M)\subset R\setminus (\shadow_{\Gcapt}(V(\Gexp),\rho k)\cup \largeintoatoms)$, or
\item[{\bf(t5)}]$V(\M)\subset V(\Gblack)\setminus\left(\shadow_{\Gcapt}(V(\Gexp),\rho k)\cup \largeintoatoms\cup R\right)$.
\end{itemize}
  \item [$\mathbf{(cB)}$] $V_1(\M)\subset \XA\colouringpI{0}\setminus
  (\gP\cup \gP_2\cup\gP_3\cup\exceptVertSplit\cup\shadowsplit)$ and
  $V_2(\M)\subset \XB\colouringpI{0}\setminus
  (\gP\cup\exceptVertSplit\cup\shadowsplit)$, and we have
\begin{itemize}
 \item[{\bf(t1)}]$V_1(\M)\subset \shadow_{\Gcapt}(V(\Gexp),\rho k)$,
 \item[{\bf(t2)}]$V_1(\M)\subset \largeintoatoms$, or
 \item[{\bf(t3--5)}]$V_1(\M)\cap \left(\shadow_{\Gcapt}(V(\Gexp),\rho k)\cup \largeintoatoms\right)=\emptyset$.
\end{itemize}
\end{itemize}
Then at least one of the
following configurations occurs:
\begin{itemize}
  \item $\mathbf{(\diamond6)}\big(\frac{\eta^3\rho^4}{10^{12}(\Omega^*)^4},4\epsilonD,
  \frac{\gamma^3\rho}{32\Omega^*},\frac{\eta^2\nu}{2\cdot10^4
  },\frac{3\eta^3}{2000},\proporce{2}(1+\frac\eta{20})k\big)$,
  \item $\mathbf{(\diamond7)}\big(\frac
{\eta^3\gamma^3\rho}{10^{12}(\Omega^*)^4},\frac
{\eta\gamma}{400},4\epsilonD,\frac{\gamma^3\rho}{32\Omega^*},\frac{\eta^2\nu}{2\cdot10^4
}, \frac{3\eta^3}{2000}, \proporce{2}(1+\frac\eta{20})k\big)$,
  \item
$\mathbf{(\diamond8)}\big(\frac{\eta^4\gamma^4\rho}{10^{15}
(\Omega^*)^5},\frac{\eta\gamma}{400},\frac{400\epsilon}{\eta},4\epsilonD,\frac
d2,\frac{\gamma^3\rho}{32\Omega^*},\frac{\eta\pi\clustersize}{200k},\frac{\eta^2\nu}{2\cdot10^4
}, \proporce{1}(1+\frac\eta{20})k,\proporce{2}(1+\frac\eta{20})k\big)$,
 \item 
 $\mathbf{(\diamond9)}\big(\frac{\rho
\eta^8}{10^{27}(\Omega^*)^3},\frac
{2\eta^3}{10^3}, \proporce{1}(1+\frac{\eta}{40})k,
\proporce{2}\left(1+\frac{\eta}{20}\right)k, \frac{400\varepsilon}{\eta}, \frac{d}2,
\frac{\eta\pi\clustersize}{200k},4\epsilonD,\frac{\gamma^3\rho}{32\Omega^*},
\frac{\eta^2\nu}{2\cdot10^4}\big)$,
  \item $\mathbf{(\diamond10)}\big( \epsilon, \frac{\gamma^2
d}2,\pi\sqrt{\epsilon'}\nu k,\frac
{(\Omega^*)^2k}{\gamma^2},\frac\eta{40} \big)$.
\end{itemize}
\end{lem}

\subsection{Proof of Lemma~\ref{outerlemma}}\label{ssec:proofofouterlemma}
Throughout this section (and including subordinate lemmas) we assume that we have the setting of Lemma~\ref{outerlemma}. In particular, we shall assume Settings~\ref{commonsetting} and~\ref{settingsplitting}.\Referee{(19)}

We distinguish different types of edges captured in cases
{\bf(K1)} and {\bf(K2)}. If in case {\bf(K1)} many of the captured edges from
$\XA$ to $\XA\cup \XB$ are incident with $\HugeVertices$, we will get one of the
configurations $\mathbf{(\diamond1)}$--$\mathbf{(\diamond5)}$ by employing
Lemma~\ref{lem:ConfWhenCXAXB}. Otherwise, there must be many edges from $\XA$ to
$\XA\cup \XB$ in the graph $\Gexp$, or in $\GD$.
Lemma~\ref{lem:ConfWhenNOTCXAXB} shows that the former case leads to
configuration $\mathbf{(\diamond6)}$. We will reduce the latter case to the
situation in Lemma~\ref{lem:ConfWhenMatching} which gives one of the configurations
$\mathbf{(\diamond6)}$--$\mathbf{(\diamond10)}$.

We use Lemma~\ref{lem:ConfWhenMatching} to give one of the
configurations
$\mathbf{(\diamond6)}$--$\mathbf{(\diamond10)}$ also in case
{\bf(K2)}. \footnote{Actually, our proof of
Lemma~\ref{lem:ConfWhenMatching} implies that one does not get configuration $\mathbf{(\diamond9)}$ in case {\bf(K2)}; but this fact is never needed.} 

\bigskip

Let us now turn to the details of the proof. If
$e_{\Gcapt}(\HugeVertices,\XA\cup\XB)\ge \frac{\eta^{13}kn}{10^{28}(\Omega^*)^3}$
\Referee{(17)} then
we use Lemma~\ref{lem:ConfWhenCXAXB} to obtain one of the configurations
$\mathbf{(\diamond1)}$--$\mathbf{(\diamond5)}$, with the parameters as in 
the statement of Lemma~\ref{outerlemma}.

Recall that every edge of $G$ incident to $\HugeVertices$ is captured. Thus, in the remainder of the proof we assume that\Referee{(17)}
\begin{equation}\label{eq:notmanyfromhuge}
e_G(\HugeVertices,\XA\cup\XB)=e_{\Gcapt}(\HugeVertices,\XA\cup\XB)< \frac{\eta^{13}kn}{10^{28}(\Omega^*)^3}\;.
\end{equation}

We now bound the size of the set $\gP$. By
Setting~\ref{commonsetting}\eqref{commonsetting:numbercaptured} we have that 
\begin{equation*}
 |E(G)\sm E(\Gcapt)|\leq 2\rho kn .
\end{equation*}
We shall therefore use Lemma~\ref{lem:YAYB} with $\beta_\PARAMETERPASSING{L}{lem:YAYB}=2\rho$.\Referee{(D)} This choice of $\beta_\PARAMETERPASSING{L}{lem:YAYB}$ is consistent with~\eqref{eq:verifybeta}; indeed, by~\eqref{eq:KONST} we have that $\eta\gg \rho\gg \gamma$, and thus $\rho\gg \eta^2\sqrt{\gamma}$.\footnote{Recall that the choice of constants in~\eqref{eq:KONST} proceeds from left to right.}
From Lemma~\ref{lem:YAYB} we get
$|L_\sharp|\le \frac{40\rho n}{\eta}$, $|\XA\setminus\YA|\le \frac{1200\rho
n}{\eta^2}$, and $|(\XA\cup \XB)\setminus\YB|\le \frac{1200\rho
n}{\eta^2}$. Further, using~\eqref{eq:notmanyfromhuge},
Lemma~\ref{lem:YAYB} also gives that $|\WantiC|\le
\frac{\eta^{12}n}{10^{26}(\Omega^*)^3}$. It follows from Setting~\ref{commonsetting}\eqref{commonsettingNicDoNAtom} that $|\gPatoms|\le \gamma n$. Lastly, by Setting~\ref{commonsetting}\eqref{commonsetting4} we have $|\gP_1|\le 2\gamma n$.
Thus,
\begin{align}\nonumber
|\gP|&\le |\XA\setminus\YA|+|(\XA\cup
\XB)\setminus\YB|+|\WantiC| +|L_\sharp|+|\gP_1|\\
\nonumber
&~~~+\left|\shadow_{\GD\cup\Gcapt}(\WantiC\cup
L_\sharp\cup \gPatoms\cup \gP_1,\frac{\eta^2 k}{10^5})\right|\\
\label{eq:sizeofP}
&\overset{\eqref{eq:KONST}}\le \frac{2\eta^{10}n}{10^{21}(\Omega^*)^2}\;,
\end{align}
where we used Fact~\ref{fact:shadowbound} to bound the size of the shadows (to this end recall that by Property~\ref{def:classgap} of Definition~\ref{sparseclassdef}, the graph $\GD\cup\Gcapt$ indeed has maximum degree at most $\Omega^*k$).\Referee{(18)}

\bigskip

Let us first turn our attention to case {\bf(K1)}.
By Definition~\ref{def:proportionalsplitting} we have $\HugeVertices\cap
\colouringp{0}=\emptyset$. Therefore,
\begin{align}
\nonumber
e_{\Gcapt}\big(\XA\colouringpI{0}\setminus \gP, &(\XA\cup
\XB)\colouringpI{0}\setminus \gP\big)=e_{\Gcapt}\big((\XA\setminus
(\HugeVertices\cup \gP))\colouringpI{0},(\XA\setminus
(\HugeVertices\cup \gP))\colouringpI{0}\cup (\XB\setminus
\gP)\colouringpI{0}\big)\\
\nonumber
\JUSTIFY{by Def~\ref{def:proportionalsplitting}~\eqref{It:H6}}&\ge
\proporce{0}^2\cdot e_{\Gcapt}\big(\XA\setminus (\HugeVertices\cup \gP), (\XA\cup
\XB)\setminus (\HugeVertices\cup \gP)\big)-k^{0.6}n^{0.6}\\
\nonumber
\JUSTIFY{by~\eqref{eq:proporcevelke}}&\ge \frac
{\eta^2}{10^4}\big(e_{\Gcapt}(\XA,
\XA\cup \XB)-2e_{\Gcapt}(\HugeVertices,
\XA\cup \XB)-2|\gP|\Omega^*k\big)-k^{0.6}n^{0.6}\\
\nonumber
\JUSTIFY{by~{\bf(K1)},~\eqref{eq:notmanyfromhuge},~\eqref{eq:sizeofP}}&\ge
\frac {\eta^2}{10^4}\Big(\frac {\eta kn}{4}-\frac{2\eta^{13}
kn}{10^{28}(\Omega^*)^3}-\frac{4\eta^{10}
kn}{10^{21}\Omega^*}\Big)-k^{0.6}n^{0.6}\\ 
&>\frac {\eta^3 kn}{10^{5}}\;.
\label{eq:Diana9}
\end{align}

We consider the
following two complementary cases:
\begin{itemize}
  \item [$\mathbf{(wA)}$] {$e_{\Gcapt}((\XA\setminus \gP)\colouringpI{0})\ge
 40\rho kn$.}
  \item [$\mathbf{(wB)}$] {$e_{\Gcapt}((\XA\setminus
  \gP)\colouringpI{0})<40\rho kn$.}
\end{itemize}

Note that  $\XA\setminus
\gP\subseteq \YA$, and $(\XA\cup \XB)\setminus
\gP\subseteq \YB$.
  We shall now define in each of the cases~$\mathbf{(wA)}$ and~$\mathbf{(wB)}$ certain sets $\YA_1, \YA_2$. The way these sets shall be defined will guarantee a lower bound on the number of edges between them. Although the definition of these sets is different for the cases~$\mathbf{(wA)}$ and~$\mathbf{(wB)}$, for ease of notation they receive the same names.  
  
In case~$\mathbf{(wA)}$ a standard argument (take a maximal cut) gives disjoint sets $\YA_1, \YA_2\subseteq 
(\XA\setminus (\gP\cup \exceptVertSplit\cup
\shadowsplit))\colouringpI{0}\subseteq \YA$ with
\begin{align}
\nonumber
e_{\Gcapt}(\YA_1,\YA_2)\ge &\frac 12(e_{\Gcapt}(\XA\setminus
\gP)\colouringpI{0}-|\exceptVertSplit\cup \shadowsplit|\cdot \Omega^*k)\\
\label{eq:PocitaniCaseA}
\JUSTIFY{by
 Def~\ref{def:proportionalsplitting}\eqref{It:H1} and by~\eqref{eq:boundShadowsplit}}\ge & \frac
 12(40\rho kn-2\epsilon \Omega^*kn)\notag \\  > &19\rho kn\;.
\end{align}

Let us now define $\YA_1, \YA_2$ for case~$\mathbf{(wB)}$.
Setting~\ref{commonsetting}\eqref{commonsettingXAS0} implies that
\begin{equation}
\label{eq:Dsmallvecer}
|\gP_2|\le \sqrt\gamma n\;\mbox{.}
\end{equation}

Also, by Definition~\ref{def:proportionalsplitting}\eqref{It:H6} we have
\begin{align*}
e_{\Gcapt}(\XA)&\le \frac{1}{\proporce{0}^2}\big(e_{\Gcapt}((\XA\setminus
\gP)\colouringpI{0})+k^{0.6}n^{0.6}\big)+e_{\Gcapt}(\HugeVertices,\XA)+|\gP|\Omega^*k\\
\JUSTIFY{by~\eqref{eq:proporcevelke},~$\mathbf{(wB)}$,~\eqref{eq:notmanyfromhuge},
and~\eqref{eq:sizeofP}}&\le\frac{10^4}{\eta^2}\cdot \big(40\rho kn+
k^{0.6}n^{0.6}\big)+
\frac{\eta^{13}}{10^{28}(\Omega^*)^3}kn+\frac{\eta^{10}}{10^{20}\Omega^*}kn\\
\JUSTIFY{by~\eqref{eq:KONST}}&<\frac {\eta^{8}}{10^{15}\Omega^*}kn\;.
\end{align*}

Consequently, 
$$|\gP_3|\cdot\frac{\eta^3k}{10^3}\le e_{\Gcapt}(\gP_3,\XA)\leq 2\cdot \frac{\eta^8}{10^{15}\Omega^*}kn,$$
and thus,
\begin{align}
\label{eq:D2mala}
|\gP_3|&\le 2\cdot \frac{\eta^5}{10^{12}\Omega^*}n\;.
\end{align}
Set $\YA_1:=(\XA\setminus (\gP\cup \gP_2\cup \gP_3\cup
\exceptVertSplit\cup \shadowsplit))\colouringpI{0}\subseteq \YA$ and
$\YA_2:=(\XB\setminus (\gP\cup \exceptVertSplit\cup
\shadowsplit))\colouringpI{0}\subseteq \YB$. Then the sets $\YA_1$ and $\YA_2$ are
disjoint and we have
\begin{align}
\nonumber
e_{\Gcapt}(\YA_1,\YA_2)&\ge
e_{\Gcapt}\left((\XA\setminus
\gP)\colouringpI{0},((\XA\cup\XB)\setminus
\gP)\colouringpI{0}\right)-2e_{\Gcapt}((\XA\setminus
\gP)\colouringpI{0})\\
\nonumber &~~~-(|\gP_2|+|\gP_3|+2|\exceptVertSplit|+2|\shadowsplit|)\cdot
\Omega^*k
\\
\nonumber
\JUSTIFY{by~\eqref{eq:Diana9}, $\mathbf{(wB)}$, \eqref{eq:Dsmallvecer},
\eqref{eq:D2mala}, D\ref{def:proportionalsplitting}\eqref{It:H1}, \eqref{eq:boundShadowsplit}} &\ge\frac{\eta^3kn}{
10^{5}}-80\rho kn-\sqrt{\gamma}\Omega^*kn-\frac{2\eta^5}{10^{12}}kn-4\epsilon
\Omega^* kn\\
&\geByRef{eq:KONST}
19\rho kn\;.\label{eq:PocitaniCaseB}
\end{align}
We have thus defined $\YA_1, \YA_2$ for both cases~$\mathbf{(wA)}$ and~$\mathbf{(wB)}$.

Observe first that if
$e_{\Gexp}(\YA_1,\YA_2)\ge 2\rho kn$
then we may apply
Lemma~\ref{lem:ConfWhenNOTCXAXB} to obtain
Configuration~$\mathbf{(\diamond6)}(\frac{\eta^3\rho^4}{10^{14}(\Omega^*)^3},0,1,1,
\frac {3\eta^3}{2\cdot 10^3},\proporce{2}\left(1+\frac{\eta}{20}\right)k)$.
Hence, from now on, let us assume that $e_{\Gexp}(\YA_1,\YA_2)> 2\rho kn$.
Then
by~\eqref{eq:PocitaniCaseA} and~\eqref{eq:PocitaniCaseB} we have that $$e_{\GD}(\YA_1,\YA_2)\ge 17\rho kn.$$

 We fix a family $\DenseSpots_\class$ as in Lemma~\ref{lem:clean-spots}. In particular, we have 
 \begin{align}\label{flordelino}
  e_{\DenseSpots_\class}(\YA_1,\YA_2)&\ge 16\rho kn\;,\mbox{and}\\
  \label{eq:DenseSpotClassBoundedMaxDeg}
  \maxdeg(\DenseSpots_\class)&\le
  \maxdeg(\DenseSpots)\leBy{D\ref{sparseclassdef}~\ref{def:classgap}.}\Omega^*k\;.
 \end{align}

Let $R:=\shadow_{\Gcapt}\left((\largeintoatoms\cap\largevertices{\eta}{k}{G})\setminus V(\M_A\cup\M_B),\frac{2\eta^2 k}{10^5}\right)$. For $i=1,2$ define 
\begin{align}
\begin{split}\label{eq:defY1Y5}
\mathbb{Y}^{(1)}_i& :=\shadow_{G}(V(\Gexp),\rho k)\cap \YA_i\;,\\
\mathbb{Y}^{(2)}_i& :=(\largeintoatoms\cap \YA_i)\setminus\mathbb{Y}^{(1)}_i\;,\\
\mathbb{Y}^{(3)}_i& :=(R\cap \YA_i)\setminus(\mathbb{Y}^{(1)}_i\cup \mathbb{Y}^{(2)}_i)\;,\\
\mathbb{Y}^{(4)}_i& :=(\smallatoms\cap \YA_i)\setminus(\mathbb{Y}^{(1)}_i\cup \mathbb{Y}^{(2)}_i\cup \mathbb{Y}^{(3)}_i)\;,\\
\mathbb{Y}^{(5)}_i& := \YA_i\setminus(\mathbb{Y}^{(1)}_i\cup\ldots \cup 
\mathbb{Y}^{(4)}_i)\;.
\end{split}
\end{align}
Clearly, the sets $\mathbb{Y}^{(j)}_i$ partition $\YA_i$ for $i=1,2$.

We now present two lemmas (one for
case {\bf(wA)} and one for case {\bf(wB)}) which help to distinguish several subcases based on  the majority type of edges we find between $\YA_1$ and $\YA_2$. The first of the two lemmas follows by a
simple counting argument from~\eqref{flordelino}.

\begin{lemma}\label{lem:MajorityTypeCB}
In case {\bf(wB)}, we have one of the following.
\begin{itemize}
 \item[{\bf(t1)}] $e_{\DenseSpots_{\class}}\left(\mathbb{Y}^{(1)}_1,\YA_2\right)\ge 2\rho kn$,
 \item[{\bf(t2)}] $e_{\DenseSpots_{\class}}\left(\mathbb{Y}^{(2)}_1,\YA_2\right)\ge 2\rho kn$,
\item[{\bf(t3)}] $e_{\DenseSpots_{\class}}\left(\mathbb{Y}^{(3)}_1,\YA_2\right)\ge 2\rho kn$,
\item[{\bf(t4)}] $e_{\DenseSpots_{\class}}\left(\mathbb{Y}^{(4)}_1,\YA_2\right)\ge 2\rho kn$, or
\item[{\bf(t5)}] $e_{\DenseSpots_{\class}}\left(\mathbb{Y}^{(5)}_1,\YA_2\right)\ge 2\rho kn$.
\end{itemize}
\end{lemma}

Our second lemma is a bit more involved.

\begin{lemma}\label{lem:MajorityTypeCA}
In case {\bf(wA)}, we have one of the following.
\begin{itemize}
 \item[{\bf(t1)}] $e_{\DenseSpots_{\class}}(\mathbb{Y}^{(1)}_1,\YA_2)+e_{\DenseSpots_{\class}}(\YA_1,\mathbb{Y}^{(1)}_2)\ge 4\rho kn$,
 \item[{\bf(t2)}] $e_{\DenseSpots_{\class}}\left(\mathbb{Y}^{(2)}_1,\YA_2\setminus \mathbb{Y}^{(1)}_2\right)+e_{\DenseSpots_{\class}}\left(\YA_1\setminus \mathbb{Y}^{(1)}_1,\mathbb{Y}^{(2)}_2\right)\ge 4\rho kn$,
\item[{\bf(t3)}] $e_{\DenseSpots_{\class}}\left(\mathbb{Y}^{(3)}_1,\YA_2\setminus (\mathbb{Y}^{(1)}_2\cup \mathbb{Y}^{(2)}_2)\right)
+
e_{\DenseSpots_{\class}}\left(\YA_1\setminus (\mathbb{Y}^{(1)}_1\cup \mathbb{Y}^{(2)}_1),\mathbb{Y}^{(3)}_2\right)\ge 4\rho kn$, or
\item[{\bf(t5)}] $e_{\DenseSpots_{\class}}\left(\mathbb{Y}^{(5)}_1,\mathbb{Y}^{(5)}_2\right)\ge 2\rho kn$.
\end{itemize}
\end{lemma}
\begin{proof}
By~\eqref{flordelino}, we only need to establish that 
$$e_{\DenseSpots_{\class}}\left(\mathbb{Y}^{(4)}_1,\YA_2\setminus (\mathbb{Y}^{(1)}_2\cup \mathbb{Y}^{(2)}_2\cup \mathbb{Y}^{(3)}_2)\right)
+
e_{\DenseSpots_{\class}}\left(\YA_1\setminus (\mathbb{Y}^{(1)}_1\cup \mathbb{Y}^{(2)}_1\cup \mathbb{Y}^{(3)}_1),\mathbb{Y}^{(4)}_2\right)< \rho kn\;.$$
For this, note that $\mathbb{Y}^{(4)}_1\subset\smallatoms$ and that $\YA_2\setminus (\mathbb{Y}^{(1)}_2\cup \mathbb{Y}^{(2)}_2\cup \mathbb{Y}^{(3)}_2)$ is disjoint from $\largeintoatoms$. Thus we have $e_{\DenseSpots_{\class}}\left(\mathbb{Y}^{(4)}_1,\YA_2\setminus (\mathbb{Y}^{(1)}_2\cup \mathbb{Y}^{(2)}_2\cup \mathbb{Y}^{(3)}_2)\right)<\frac{\rho k n}{100\Omega^*}$. We can bound the other summand using a symmetric argument.
\end{proof}

We can now provide a crucial step for finishing case~{\bf(K1)}.
\begin{lemma}\label{lem:Isabelle}
Let $G^*$ be the spanning subgraph of $\GD$ formed by the edges of
$\DenseSpots_\class$. If there are two disjoint sets $Z_1$ and $Z_2$ with
$e_{G^*}(Z_1,Z_2)\ge 2\rho kn$ then there exists an
$(\epsilonD,\frac{\gamma^3\rho}{32\Omega^*},\frac{\alphaD\rho
k}{\Omega^*})$-\semiregular matching $\mathcal N$ in $G^*$ with $V_i(\mathcal
N)\subset Z_i$ ($i=1,2$), and $|V(\mathcal N)|\ge\frac{\rho n}{\Omega^*}$.
\end{lemma}
\begin{proof}
By~\eqref{eq:DenseSpotClassBoundedMaxDeg}, the maximum degree of $G^*$ is bounded by $\Omega^* k$.\Referee{(20)}
Therefore, we have $|Z_1|\ge \frac{2\rho n}{\Omega^*}\ge \frac{2\rho
k}{\Omega^*}$. Thus, $$(G^*,\DenseSpots_\class,G^*[Z_1,Z_2],\{Z_1\})\in\mathcal
G\left(v(\GD), k, \Omega^*, \frac{\gamma^3}4, \frac{\rho}{\Omega^*},2\rho\right)\;,$$
where the class of the right-hand side was defined in Definition~\ref{tupelclass}.
Lemma~\ref{lem:edgesEmanatingFromDensePairsIII} (which applies with these parameters by the choice of $\alphaD$ and $k_0$ by~\eqref{eq:KONST})
immediately gives the desired output.
\end{proof}
We use Lemma~\ref{lem:Isabelle} with $Z_1,Z_2$ being the pair of sets containing
many edges as in the cases {\bf(t1)}--{\bf(t3)} and {\bf(t5)} of
Lemma~\ref{lem:MajorityTypeCA}\footnote{The quantities in
Lemma~\ref{lem:MajorityTypeCA} have two summands. We take the sets $Z_1$,$Z_2$
as those appearing in the majority summand.} and {\bf(t1)}--{\bf(t5)} of
Lemma~\ref{lem:MajorityTypeCB}. The lemma outputs a \semiregular matching
$\M_\PARAMETERPASSING{L}{lem:ConfWhenMatching}:=\mathcal N_\PARAMETERPASSING{L}{lem:Isabelle}$. This matching is a basis of the input for Lemma~\ref{lem:ConfWhenMatching}{\bf(M2)} (subcase {\bf(t1)}--{\bf(t3)}, {\bf(t5)}, or {\bf(t3--5)}). Thus, we get one of the configurations $\mathbf{(\diamond6)}$--$\mathbf{(\diamond10)}$ as in the statement of the lemma. This finishes the proof for case {\bf(K1)}.

\bigskip

Let us now turn our attention to case {\bf(K2)}.
For every pair $(X,Y)\in \Mgood$, let $X'\subseteq X\colouringpI{0}\setminus
(\gP\cup \exceptVertSplit\cup \shadowsplit)$ and $Y'\subseteq Y\colouringpI{0}\setminus
(\gP\cup \exceptVertSplit\cup \shadowsplit)$ be maximal with $|X'|=|Y'|$. Define
$\mathcal N:=\{(X',Y')\::\: (X,Y)\in \Mgood\;,\; |X'|\ge \frac
{\eta^2\clustersize}{2\cdot 10^3}\}$. By Lemma~\ref{lem:RestrictionSemiregularMatching}, and using~\eqref{eq:KONST} and~\eqref{eq:proporcevelke},
we know that 
$$|V(\Mgood\colouringpI{0})|\ge  \frac {\eta^2n}{400}.$$ Therefore, we have
\begin{align}\label{campoafuera}
|V(\mathcal N)|&\ge |V(\Mgood\colouringpI{0})|-2|\gP\cup \exceptVertSplit\cup
\shadowsplit|-2\frac {\eta^2n}{2\cdot 10^3}\notag \\
\JUSTIFY{by {\bf(K2)}, \eqref{eq:sizeofP}, Def\ref{def:proportionalsplitting}\eqref{It:H1}, \eqref{eq:boundShadowsplit}}
&\ge \frac {\eta^2n}{400}-\frac {4\cdot
\eta^{10} n}{10^{21}(\Omega^*)^2}-4\varepsilon n-\frac {\eta^2n}{10^3}\notag \\&>
\frac {\eta^2n}{1000}\;.
\end{align}
By Fact~\ref{fact:BigSubpairsInRegularPairs}, $\mathcal N$ is a $(\frac {4\cdot
10^3\epsilon'}{\eta^2}, \frac{\gamma^2}2,\frac
{\eta^2\clustersize}{2\cdot 10^3})$-\semiregular matching.

We use the definitions of the sets $\mathbb{Y}^{(1)}_i,\ldots,\mathbb{Y}^{(5)}_i$ as given in~\eqref{eq:defY1Y5} with $\YA_i:=V_i(\mathcal N)$ ($i=1,2$). As $V(\mathcal N)\subset V(\Gblack)$, we have that $\mathbb{Y}^{(4)}_i=\emptyset$ ($i=1,2$).
A set $X\in \V_i(\mathcal N)$ is said to be of \emph{Type~1} if $\left|X\cap
\mathbb{Y}^{(1)}_i\right|\ge\frac14|X|$. 
Analogously, we define elements of $\V(\mathcal N)$ of \emph{Type~2}, \emph{Type~3}, and \emph{Type~5}.  

By~\eqref{campoafuera} and as $V(\Mgood)\subset \XA$, we are in subcase $\mathbf{(wA)}$. 
For each  $(X_1,X_2)\in\mathcal N$ with at least one $X_i\in \{X_1,X_2\}$
being of Type~1, set $X_i':=X_i \cap \mathbb{Y}^{(1)}_i$ 
and take an arbitrary set $X_{3-i}'\subset X_{3-i}$ of size $|X_i'|$. Note that
by Fact~\ref{fact:BigSubpairsInRegularPairs} $(X'_i,X_{3-i}')$ forms a
$\frac{10^5\epsilon'}{\eta^2}$-regular pair of density at least $\gamma^2/4$. We let $\mathcal N_1$ be the \semiregular matching consisting of all pairs
$(X'_i,X_{3-i}')$ obtained in this way.\footnote{Note that we are thus changing
the orientation of some subpairs.}

Likewise, we construct $\mathcal N_2,\mathcal N_3$ and $\mathcal N_5$ using the features of
Type~2, 3, and 5. Observe that the matchings $\mathcal N_i$ may intersect.

Because of~\eqref{campoafuera} and since we included at least one quarter of each $\mathcal N$-edge into one of
$\mathcal N_1,\mathcal N_2,\mathcal N_3$ and $\mathcal N_5$, one of the
\semiregular matchings $\mathcal N_i$ satisfies $|V(\mathcal N_i)|\geq \frac{\eta^2n}{16\cdot 1000} \geq \frac{\rho}{\Omega^*}n$. So, $\mathcal N_i$ serves as a matching $\M_\PARAMETERPASSING{L}{lem:ConfWhenMatching}$ for Lemma~\ref{lem:ConfWhenMatching}{\bf(M1)}. Thus, we get one of the configurations $\mathbf{(\diamond6)}$--$\mathbf{(\diamond10)}$ as in the statement of the lemma. This finishes case {\bf(K2)}.

\subsection{Proof of Lemma~\ref{lem:ConfWhenCXAXB}}\label{sssec:ProofConfWhenCXAXB}
Set $\tilde\eta:=\frac{\eta^{13}}{10^{28}(\Omega^*)^3}$. 
Define $\NUP:=\{v\in V(G)\::\: \deg_{\Gcapt}(v,\HugeVertices)\ge
k\}$, and $\NDOWN:=\neighbour_{\Gcapt}(\HugeVertices)\setminus \NUP$. 
Recall that by the definition of the class $\LKSsmallgraphs{n}{k}{\eta}$, the
set $\HugeVertices$ is independent, and thus the sets $\NUP$ and $\NDOWN$ are
disjoint from $\HugeVertices$. Also, using the same definition, we have
\begin{align} 
\label{eq:NCpodL} \neighbour_{\Gcapt}(\HugeVertices)&\subset
\largevertices{\eta}{k}{G}\setminus \HugeVertices\quad\mbox{, and thus }\\
\label{eq:MuzuProniknoutsL}
e_{\Gcapt}(\HugeVertices,
B)&=e_{\Gcapt}(\HugeVertices, B\cap\largevertices{\eta}{k}{G})\;\mbox{for any $B\subset V(G)$.}
\end{align}

We shall distinguish two cases.

\noindent\underline{\bf Case A:} $e_{\Gcapt}(\HugeVertices,\NUP)\ge
e_{\Gcapt}(\HugeVertices,\XA\cup\XB)/8$.\\ Let us focus on the bipartite
subgraph $H'$ of $\Gcapt$ induced by the sets $\HugeVertices$ and $\NUP$.
Obviously, the average degree of the vertices of $\NUP$ in $H'$ is at least $k$.

First, suppose that $|\HugeVertices|\le |\NUP|$. Then,
the average degree of $\HugeVertices$ in $H'$ is at least $k$, and hence, the
average degree of $H'$ is at least $k$. Thus, there exists a bipartite subgraph $H\subset H'$ with $\mindeg(H)\ge k/2$. Furthermore, $\mindeg_{\Gcapt}(V(H))\ge k$. We conclude that we are in Configuration~$\mathbf{(\diamond1)}$.

Now, suppose $|\HugeVertices|>|\NUP|$. Using the bounds given by Case~A, and using
\eqref{libelle7.38b}, we get $$|\NUP|\ge
\frac{e_{\Gcapt}(\HugeVertices,\NUP)}{\Omega^*k}\ge\frac{\tilde{\eta}kn}{8\Omega^*k}=\frac{\tilde{\eta}n}{8\Omega^*}\;.$$
Therefore, we have
 $$e(G)\ge\sum_{v\in\HugeVertices}\deg_{\Gcapt}(v)\ge
 |\HugeVertices|\Omega^{**}k>
 |\NUP|\Omega^{**}k\ge\frac{\tilde{\eta}n}{8\Omega^*}\Omega^{**}k\geByRef{eq:KONST}
 kn\;,$$ a contradiction to Property~\ref{def:LKSsmallC} of Definition~\ref{def:LKSsmall}.

\bigskip
\noindent\underline{\bf Case B:} $e_{\Gcapt}(\HugeVertices,\NUP)<
e_{\Gcapt}(\HugeVertices,\XA\cup\XB)/8$.\\ Consequently, we get
\begin{equation}\label{eq:eGCNUB}
e_{\Gcapt}(\HugeVertices,(\XA\cup\XB)\setminus \NUP)\ge
\frac78e_{\Gcapt}(\HugeVertices,\XA\cup\XB)\geByRef{libelle7.38b} \frac
78\tilde\eta kn\;.
\end{equation}

\def\Lenve{\PARAMETERPASSING{L}{lem:envelope}}

\Referee{(E)}We now apply Lemma~\ref{lem:envelope} to $\Gcapt$ with input sets $P_\Lenve:=\HugeVertices$, $Q_\Lenve:=\largevertices{\eta}{k}{G}\setminus \HugeVertices$,
$Y_\Lenve:=\largevertices{\eta}{k}{G}\setminus\largevertices{\frac9{10}\eta}{k}{\Gcapt}$,
and parameters $\psi_\Lenve:=\tilde \eta/100$, $\Gamma_\Lenve:=\Omega^*$,
$\Omega_\Lenve:=\Omega^{**}$, and 
 $\Omega'_\Lenve:=\tilde \eta^3\Omega^{**}/(4\cdot 10^6(\Omega^*)^2)$.
Assumption~\eqref{eq:envelopeAss1} of
the lemma follows from~\eqref{eq:NCpodL}\BUG{, and Assumption~\eqref{eq:jaksestavidum} holds by the choice of $\Omega'_\Lenve$}. The lemma yields three sets
$L'':=Q''_\Lenve$, $L':=Q'_\Lenve$, and $\HugeVertices':=P'_\Lenve$, and it is easy to check that they witness
Preconfiguration~$\mathbf{(\clubsuit)}( \frac{\tilde \eta^3\Omega^{**}}{4\cdot 10^6(\Omega^*)^2})$.

Recall that $e(G)\le kn$. Since by the definition of $Y_\Lenve$, we have
$|Y_\Lenve|\le\frac{40\rho}{\eta}n$,
we obtain from Lemma 5.1~\eqref{en:52e}
 that
\begin{align}
\BUG{e_{\Gcapt}(\HugeVertices,\largevertices{\eta}{k}{G})-e_{\Gcapt}(\HugeVertices',L'')}&\BUG{\le
\frac{\tilde \eta}{100}
e_{\Gcapt}(\HugeVertices,\largevertices{\eta}{k}{G})+|Y_\Lenve|\Omega^*
k}\nonumber\\
\nonumber &\BUG{\le \frac{\tilde
\eta}{100}kn+\frac{40\rho
n}{\eta}\cdot\Omega^*k}\\
&\leByRef{eq:KONST}
\frac{\tilde\eta}2 kn.\label{libelle7.36b}
\end{align}
 So, 
\begin{align}\nonumber
e_{\Gcapt}\big(\HugeVertices',(L''\cap(\XA\cup\XB))\setminus \NUP\big)
&\ge
e_{\Gcapt}\big(\HugeVertices,(\largevertices{\eta}{k}{G}\cap(\XA\cup\XB))\setminus
\NUP\big)
\nonumber\\
&~~-\big(e_{\Gcapt}(\HugeVertices,\largevertices{\eta}{k}{G}\big)-e_{\Gcapt}(\HugeVertices',L'')\big)
\nonumber\\
\nonumber 
&=e_{\Gcapt}(\HugeVertices,(\XA\cup\XB)\setminus
\NUP)
\\ \nonumber
&~~-\big(e_{\Gcapt}(\HugeVertices,\largevertices{\eta}{k}{G}\big)-e_{\Gcapt}(\HugeVertices',L'')\big)\\
&\nonumber\geByRef{libelle7.36b}e_{\Gcapt}(\HugeVertices,(\XA\cup\XB)\setminus \NUP)
-\frac{\tilde\eta}2 kn\\ &\geByRef{eq:eGCNUB} \frac38\tilde\eta kn\;.\label{eq:eGCNUP}
\end{align}

We define $$\HugeVertices^*:=\left\{v\in\HugeVertices'\::\: \deg_{\Gcapt}(v,L''\cap (\XA\cup\XB)\cap
\NDOWN)\ge\sqrt{\Omega^{**}}k\right\}\;.$$ 

Using that $e(G)\le kn$, we shall prove the following.
\begin{lemma}\label{lem:C*XAXBNDOWN}
We have 
$e_{\Gcapt}(\HugeVertices^*,L''\cap (\XA\cup\XB)\cap \NDOWN)\ge \frac18\tilde\eta kn$.
\end{lemma}
\begin{proof}
Suppose otherwise. Then by~\eqref{eq:eGCNUP}, we obtain that
 $$e_{\Gcapt}(\HugeVertices'\setminus \HugeVertices^*,L''\cap (\XA\cup\XB)\cap\NDOWN)
 \ge \frac14\tilde\eta kn\;.$$ On the other hand, by the definition of $\HugeVertices^*$,
 $$|\HugeVertices'\setminus\HugeVertices^*|\sqrt{\Omega^{**}}k\ge e_{\Gcapt}(\HugeVertices'\setminus\HugeVertices^*,L''\cap (\XA\cup\XB)\cap\NDOWN)\;.$$ 
Consequently, we have $$|\HugeVertices'\setminus\HugeVertices^*|\ge\frac{\tilde\eta kn}{4\sqrt{\Omega^{**}}k}=\frac{\tilde\eta n}{4\sqrt{\Omega^{**}}}\;.$$ 
Thus, as $\HugeVertices$ is independent, $$e(G)\ge \sum_{v\in\HugeVertices}\deg_{\Gcapt}(v)\ge |\HugeVertices|\Omega^{**}k\ge |\HugeVertices'\setminus\HugeVertices^*|\Omega^{**}k\ge \frac{\tilde\eta}4 \sqrt{\Omega^{**}}kn\gByRef{eq:KONST} kn\;,$$ a contradiction.
\end{proof}

Let us define $O:=\shadow_{\Gcapt}(\smallatoms,\gamma k)$. 
Next, we define \begin{align*}
N_1&:=V(\Gexp)\cap L''\cap (\XA\cup\XB)\cap \NDOWN\;,\\
N_2&:=\smallatoms\cap L''\cap 
(\XA\cup\XB)\cap \NDOWN \;,\\
N_3&:=O\cap L''\cap  (\XA\cup\XB)\cap \NDOWN\;\mbox{, and}\\
N_4&:=(L''\cap (\XA\cup\XB)\cap
\NDOWN)\setminus(N_1\cup N_2\cup N_3)\;.
                \end{align*}
Observe that
\begin{equation}\label{eq:ON4}
O\cap N_4=\emptyset\;.
\end{equation}
Further, for $i=1,\ldots,4$ define
$$C_i:=\left\{v\in\HugeVertices^*\::\: \deg_{\Gcapt}(v, N_i)\ge \deg_{\Gcapt}(v, L''\cap (\XA\cup\XB)\cap
\NDOWN)/4\right\}\;.$$ 
An easy calculation gives that there exists an index $i\in[4]$ such
that
\begin{equation}\label{eq:OneInFour}
e_{\Gcapt}(C_i,N_i)\ge \frac1{16}e_{\Gcapt}(\HugeVertices^*,L''\cap (\XA\cup\XB)\cap
\NDOWN)\geBy{L\ref{lem:C*XAXBNDOWN}}\frac1{128}\tilde\eta kn\;.
\end{equation}

\def\LSP{\PARAMETERPASSING{L}{lem:clean-C+black}}
\def\L74{\PARAMETERPASSING{L}{lem:clean-C+yellow}}
Set $Y:=(\XA\cup\XB)\setminus (\YB\cup\HugeVertices)=(\XA\cup\XB)\setminus \YB$, and
$\eta_\L74=\eta_\LSP:=\frac{1}{128}\tilde\eta$. By Lemma~\ref{lem:YAYB} we have
\begin{equation}\label{eq:Ysmall}
|Y|< \frac{\eta_\L74 n}{4\Omega^*}\;.
\end{equation}

We split the rest of the proof into four subcases according to the value of $i$.

\noindent\underline{\bf Subcase B, $i=1$.}\\
We shall apply Lemma~\ref{lem:clean-C+yellow} with
 $r_\L74:=2$,
$\Omega^*_\L74:=\Omega^*,\Omega^{**}_\L74:=\sqrt{\Omega^{**}}/4$,
$\delta_\L74:=\frac{\eta_\L74\rho^2}{100(\Omega^*)^2}$, $\gamma_\L74:=\rho$,
$\eta_\L74$, $X_0:=C_1$, $X_1:=N_1$, and $X_2:=V(\Gexp)$, and
$Y$, and the graph $G_\L74$, which is formed by the vertices of $G$, with all edges from $E(\Gcapt)$ that are
in $E(\Gexp)$ or that  are incident with $\HugeVertices$. We briefly
verify the assumptions of Lemma~\ref{lem:clean-C+yellow}. First of all the choice of  $\delta_\L74$ guarantees
that $\left(\frac{3\Omega_\L74^*}{\gamma_\L74}\right)^2\delta_\L74<\frac{\eta_\L74}{10}$. Assumption~\ref{hyp:Ysmall} is
given by~\eqref{eq:Ysmall}. Assumption~\ref{hyp:C+y-edges} holds since we assume 
that~\eqref{eq:OneInFour} is satisfied for $i=1$ and by definition of
$\eta_\L74$. Assumption~\ref{hyp:C+y-large} follows from the definitions of $C_1$ and of $\HugeVertices^*$. Assumption~\ref{hyp:C+y-deg} follows from the fact that $X_1\subset V(\Gexp)=X_2$, and since $\mindeg(\Gexp)>\rho k$ which is guaranteed by the definition of
a $(k,\Omega^{**},\Omega^*,\Lambda,\gamma,\epsilon',\nu,\rho)$-sparse decomposition. This definition also guarantees 
Assumption~\ref{hyp:C+y-bounded}, as $Y\cup X_1\cup X_2\subset V(G)\setminus \HugeVertices$.

Lemma~\ref{lem:clean-C+yellow} outputs sets $\HugeVertices'':=X_0'$, $V_1:=X_1'$,
$V_2:=X_2'$ with $\mindeg_{\Gcapt}(\HugeVertices'', V_1)\ge \sqrt[4]{\Omega^{**}}k/2$ (by~\eqref{conc:C+y-large}),
$\maxdeg_{\Gexp}(V_1,X_2\setminus V_2)< \rho k/2$ (by~\eqref{conc:C+y-avoid}),
$\mindeg_{\Gcapt}(V_1,\HugeVertices'')\ge \delta_\L74 k$ (by~\eqref{conc:C+y-deg}), and
$\mindeg_{\Gexp}(V_2,V_1)\ge \delta_\L74 k$ (by~\eqref{conc:C+y-deg}).
By~\eqref{conc:X1Ydisj},  we have that $V_1\subset \YB\cap
L''$. As
$\mindeg_{\Gexp}(V_1,X_2)\ge \mindeg(\Gexp)\ge \rho k$, we have $\mindeg_{\Gexp}(V_1,V_2)\ge \mindeg_{\Gexp}(V_1,X_2)-\maxdeg_{\Gexp}(V_1,X_2\setminus
V_2)\ge \delta_\L74 k$. 

Since $L'$, $L''$,  $\HugeVertices'$
 witness
Preconfiguration~$\mathbf{(\clubsuit)}(\frac{\tilde \eta^{3}\Omega^{**}}{4\cdot 10^6(\Omega^*)^2})$,
 this verifies that we have Configuration $\mathbf{(\diamond2)}\left(\frac{\tilde \eta^{3}\Omega^{**}}{4\cdot
10^{6}(\Omega^*)^{2}},\sqrt[4]{\Omega^{**}}/2,\frac{\tilde\eta\rho^2}{12800(\Omega^*)^2}\right)$.

\noindent\underline{\bf Subcase B, $i=2$.}\\
We apply Lemma~\ref{lem:clean-C+yellow} with numerical parameters $r_\L74:=2$,
$\Omega^*_\L74:=\Omega^*$, $\Omega^{**}_\L74:=\sqrt{\Omega^{**}}/4$,
$\delta_\L74:=\frac{\eta_\L74
\gamma^2}{100(\Omega^*)^2}$, $\gamma_\L74:=\gamma$, and $\eta_\L74$. Further input
to the lemma are sets $X_0:=C_2$, $X_1:=N_2$, and $X_2:=V(G)\setminus
\HugeVertices$, and the set $Y$. The underlying graph $G_\L74$ is the graph
$\GD$ with all edges incident with $\HugeVertices$ added. Verifying assumptions
of Lemma~\ref{lem:clean-C+yellow} is analogous to Subcase~B, $i=1$ with the exception of Assumption 4. To verify this, it suffices to observe that each vertex in $X_1$ is contained
in at least one $(\gamma k,\gamma)$-dense spot from $\DenseSpots$
(cf.~Definition~\ref{def:avoiding}), and thus has degree at least $\gamma k$ in
$X_2$.
 
Lemma~\ref{lem:clean-C+yellow} outputs sets $X_0',X_1'$, and $X_2'$ which witness
Configuration~$\mathbf{(\diamond3)}(\frac{\tilde \eta^{3}\Omega^{**}}{4\cdot
10^{6}(\Omega^*)^{2}}$, $\sqrt[4]{\Omega^{**}}/2$, $\gamma/2$, $\frac{\tilde\eta\gamma^2}{12800(\Omega^*)^2})$. In fact, the only thing not analogous to the preceding subcase is that we have to check~\eqref{eq:WHtc}. In other words, we have to verify that $$\maxdeg_{\GD}\big(X_1', V(G)\setminus (X_2'\cup\HugeVertices)\big)\leq \frac{\gamma k}2\;.$$ As $V(G)\setminus (X_2'\cup\HugeVertices)=X_2\setminus X'_2$, this follows from~\eqref{conc:C+y-avoid} of Lemma~\ref{lem:clean-C+yellow}.

\noindent\underline{\bf Subcase B, $i=3$.}\\
We apply Lemma~\ref{lem:clean-C+yellow} with numerical parameters $r_\L74:=3$,
$\Omega^*_\L74:=\Omega^*$, $\Omega^{**}_\L74:=\sqrt{\Omega^{**}}/4$, 
$\delta_\L74:=\frac{\eta_\L74\gamma^3}{300(\Omega^*)^3}$, $\gamma_\L74:=\gamma$,
and $\eta_\L74$. Further inputs are the sets $X_0:=C_3$, $X_1:=N_3$, $X_2:=\smallatoms$, and $X_3:=V(G)\setminus \HugeVertices$, and the set $Y$. The underlying graph is
$G_\L74:=\Gcapt\cup\GD$.
Verifying assumptions Lemma~\ref{lem:clean-C+yellow} is analogous to Subcase~B,
$i=1$, only for Assumption 4 we observe that $\mindeg_{\Gcapt\cup\GD}
(X_1,X_2)\ge\mindeg_{\Gcapt}
(X_1,X_2)\geq\gamma k$ by definition of $X_1=N_3\subset O$, and
$\mindeg_{\Gcapt\cup\GD}(X_2,X_3)\ge \mindeg_{\GD}(X_2,X_3)\ge\gamma
k$ for the same reason as in Subcase~B, $i=2$.
  
Lemma~\ref{lem:clean-C+yellow} outputs
Configuration~$\mathbf{(\diamond4)}\left( \frac{\tilde \eta^{3}\Omega^{**}}{4\cdot
10^{6}(\Omega^*)^{2}},\sqrt[4]{\Omega^{**}}/2,\gamma/2,\frac{\tilde\eta\gamma^3}{38400(\Omega^*)^3}\right)$, with $\HugeVertices'':=X_0'$, $V_1:=X_1'$, $\smallatoms':=X'_2$ and $V_2:=X'_3$. Indeed, all calculations are similar to the ones in the preceding two subcases, we only need to note additionally that $\mindeg_{\Gcapt\cup\GD}(V_1,\smallatoms')\geq \frac{\gamma k}{2} \frac{\tilde\eta\gamma^3 k}{38400(\Omega^*)^3}$, which follows from the definition of $N_3$ and of $O$.

\noindent\underline{\bf Subcase B, $i=4$.}\\
We have $\clusters\neq\emptyset$ and $\clustersize$ is the size of an arbitrary cluster in $\clusters$. We are going to apply
Lemma~\ref{lem:clean-C+black} with $\delta_\LSP:=\eta_\LSP/100$, $\eta_\LSP$,
$h_\LSP:=\eta_\LSP \clustersize/(100\Omega^*)$, $\Omega^*_\LSP:=\Omega^*$,
$\Omega^{**}_\LSP:=\sqrt{\Omega^{**}}/4$ and sets $X_0:=C_4$, $X_1:=N_4$, and
$Y$. The underlying graph is $G_\LSP:=\Gcapt$, and $\mathcal C_\LSP$ is the set
of clusters $\clusters$.

The fact $e(G)\leq kn$ together with~\eqref{eq:OneInFour} and the choice of $\eta_\LSP$ gives
Assumption~\ref{hyp:C+b-edges} of  Lemma~\ref{lem:clean-C+black}. The choice of
$C_4$ and $\HugeVertices^*$ ensures Assumption~\ref{hyp:C+b-large}. The fact that
$X_1\cap \HugeVertices=\emptyset$ yields Assumption~\ref{hyp:C+b-bounded}. With the
help of~\eqref{eq:KONST} it is easy to check Assumption~\ref{hypfburg}.
Inequality~\eqref{eq:Ysmall} implies Assumption~\ref{hypfburg5}.
To
verify Assumption~\ref{hypfburg6}, it
is enough to use that $|\mathcal C_\LSP|\le \frac{n}\clustersize$. We have thus verified all the
assumptions of Lemma~\ref{lem:clean-C+black}.

We claim that Lemma~\ref{lem:clean-C+black} outputs Configuration
$\mathbf{(\diamond5)}\Big(\frac{\tilde \eta^{3}\Omega^{**}}{4\cdot
10^{6}(\Omega^*)^{2}}$,
$\sqrt[4]{\Omega^{**}}/2$, $\frac{\tilde\eta}{12800}$, $\frac\eta2,\frac{\tilde\eta}{12800\Omega^*}\Big)$, with $\HugeVertices'':=X_0'$ and $V_1:=X_1'$.
In fact,  all conditions of the configuration, except 
condition~\eqref{confi5last}, which we check below, are easy to verify. (Note that $V_1\subseteq \YB$ since $V_1\subseteq X_1=N_4\subseteq\XA\cup\XB$. Also,  $V_1\subseteq L''$, and thus $V_1$ is disjoint from $\HugeVertices$. Moreover, by the conditions of Lemma~\ref{lem:clean-C+black}, $V_1$ is disjoint from $Y$. So, $V_1\subset\YB$.) 
For~\eqref{confi5last}, observe that~\eqref{eq:ON4}
implies that $\maxdeg_{\Gcapt}(N_4,\smallatoms)\leq\gamma k$. Further, we have
$X_1'\subset N_4\setminus Y$. So for all $x\in X_1'\subseteq
\NDOWN\setminus Y$, we have that $\deg_{\Gcapt}(x,V(G)\setminus \HugeVertices)\ge \frac{9\eta k}{10}$. As $N_4\subseteq \bigcup\clusters\setminus
V(\Gexp)$, we obtain $\deg_{\Gblack}(x)\ge \frac{9\eta
k}{10}-\gamma k\ge\frac{\eta k}2$, satisfying~\eqref{confi5last}.

\subsection{Proof of Lemma~\ref{lem:ConfWhenNOTCXAXB}}\label{sssec:ProofConfWhenNOTCXAXB}
Set $\YA_1':=\{v\in\YA_1\::\:\deg_{\Gexp}(v,\YA_2)\ge\rho k\}$. By~\eqref{eq:manyXAXAXBobt} we have
\begin{align}
\label{eq:manyXAXAXBobtLP}
e_{\Gexp}(\YA_1',\YA_2)&\ge \rho kn\;.
\end{align}

Set $r_\PARAMETERPASSING{L}{lem:clean-yellow}:=3$,
$\Omega_\PARAMETERPASSING{L}{lem:clean-yellow}:=\Omega^*$,
$\gamma_\PARAMETERPASSING{L}{lem:clean-yellow}:=\frac{\rho\eta}{10^3}$, 
$\delta_\PARAMETERPASSING{L}{lem:clean-yellow}:=\frac
{\eta^3\rho^4}{10^{14}(\Omega^*)^3}$,
$\eta_\PARAMETERPASSING{L}{lem:clean-yellow}:=\rho$.
Observe that~\eqref{eq:condCY} is satisfied for these parameters.
Set $Y_\PARAMETERPASSING{L}{lem:clean-yellow}:=\exceptVertSplit$, $X_0:=\YA_2$,
$X_1:=\YA_1'$, $X_2=X_3:=V(\Gexp)\colouringpI{1}$, and $V:=V(G)$.
Let $E_2:=E(\Gcapt)$, and $E_1=E_3:=E(\Gexp)$. We now briefly verify
conditions~\ref{hyp:yel-Ysmall}--\ref{hyp:yel-bounded} of Lemma~\ref{lem:clean-yellow}.
Condition~\ref{hyp:yel-Ysmall} follows from 
Definition~\ref{def:proportionalsplitting}\eqref{It:H1} and~\eqref{eq:KONST}.
Condition~\ref{hyp:yel-edges} follows from~\eqref{eq:manyXAXAXBobtLP}. 
Using Definition~\ref{def:proportionalsplitting}\eqref{It:H5},~\eqref{eq:proporcevelke} and~\eqref{eq:KONST}, we see that
Condition~\ref{hyp:yel-deg} for $i=1$ follows from the definition of $\YA_1'$, and for $i=2$ from the fact that $\mindeg(\Gexp)\ge\rho k$.
Lastly, Condition~\ref{hyp:yel-bounded} follows from the fact that
$\bigcup_{i=0}^3 X_i$ is disjoint from $\HugeVertices$. 

Lemma~\ref{lem:clean-yellow} yields four non-empty
sets $X_0',\ldots,X_3'$. By 
assertions~\eqref{conc:yel-deg},~\eqref{conc:yel-avoid},~\eqref{conc:yel-X0X1}, and
hypothesis~\ref{hyp:yel-deg} of Lemma~\ref{lem:clean-yellow},
for all $i\in\{0,1,2,3\}$, $j\in\{i-1, i+1\}\sm \{-1,4\}$ we have
\begin{equation}
\label{eq:towardsD62D63Exp}
\mindeg_{H_{i,j}}(X_i',X_j')\ge
\delta_\PARAMETERPASSING{L}{lem:clean-yellow} k,
\end{equation}
where $H_{i,j}=\Gexp$, except for $\{i,j\}=\{1,2\}$, where $H_{i,j}=G_\class$.

Thus, the sets $X'_0$ and $X'_1$ witness
Preconfiguration~$\mathbf{(exp)}(\delta_\PARAMETERPASSING{L}{lem:clean-yellow})$.
By Lemma~\ref{lem:propertyYA12}, and by~\eqref{2ndcondiObt2} and~\eqref{3rdcondiObt2}, the pair
$X_0',X_1'$ together with the cover $\mathcal F$ from~\eqref{def:Fcover} witnesses either 
Preconfiguration~$\mathbf{(\heartsuit1)}(\frac
{3\eta^3}{2\cdot 10^3},\proporce{2}\left(1+\frac{\eta}{20}\right)k)$ (with respect to $\mathcal F$)
or Preconfiguration~$\mathbf{(\heartsuit2)}(\proporce{2}\left(1+\frac{\eta}{20}\right)k)$.

Notice that~\eqref{eq:towardsD62D63Exp}
establishes the properties~\eqref{COND:D6:1}--\eqref{COND:D6:4}. Thus the sets $X_0',\ldots,X_3'$ witness
Configuration~$\mathbf{(\diamond6)}(\delta_\PARAMETERPASSING{L}{lem:clean-yellow},0,1,1,
\frac{3\eta^3}{2\cdot 10^3},\proporce{2}\left(1+\frac{\eta}{20}\right)k)$.

\HIDDENTEXT{The old version (something might be wanted to be copied from it for OBT in HIDDENTEXT.txt, ``PALOALTO''}

\subsection{Proof of Lemma~\ref{lem:ConfWhenMatching}}\label{sssec:ProofConfWhenMatching}
In Lemmas~\ref{whatwegetfrom(t1)}, \ref{whatwegetfrom(t2)}, \ref{whatwegetfrom(t3)}, \ref{ObtConf9}, \ref{whatwegetfrom(t5)} below, we show that cases $\mathbf{(t1)}$, $\mathbf{(t2)}$, $\mathbf{(t3)}$, $\textbf{(t3--t5)}$, and $\mathbf{(t5)}$ of Lemma~\ref{lem:ConfWhenMatching}
lead to configuration $\mathbf{(\diamond6)}$, $\mathbf{(\diamond7)}$, $\mathbf{(\diamond8)}$, $\mathbf{(\diamond9)}$, and $\mathbf{(\diamond10)}$, respectively. While the first three of these cases are handled by a fairly straightforward application of the Cleaning Lemma (Lemma~\ref{lem:clean-Match}), the latter two cases require some further non-trivial computations.

\def\MgoodR{\Mgood\colouringpI{0}} 

\begin{lemma}\label{whatwegetfrom(t1)} 
In case $\mathbf{(t1)}$ (of either subcase $\mathbf{(cA)}$ or subcase $\mathbf{(cB)}$) we obtain Configuration
$\mathbf{(\diamond6)}\big(\frac{\eta^3\rho^4}{10^{12}(\Omega^*)^4},4\epsilonD, \frac{\gamma^3\rho}{32\Omega^*}, 
 \frac{\eta^2\nu}{2\cdot10^4}, \frac{3\eta^3}{2000},
\proporce{2}(1+\frac\eta{20})k\big)$.
\end{lemma}
\begin{proof}

 We use Lemma~\ref{lem:clean-Match}
with the following input parameters:
$r_\PARAMETERPASSING{L}{lem:clean-Match}:=3$,
$\Omega_\PARAMETERPASSING{L}{lem:clean-Match}:=\Omega^*$,
$\gamma_\PARAMETERPASSING{L}{lem:clean-Match}:=\eta\rho/200$,
$\eta_\PARAMETERPASSING{L}{lem:clean-Match}:=\rho/(2\Omega^*)$,
$\delta_\PARAMETERPASSING{L}{lem:clean-Match}:=\eta^3\rho^4/(10^{12}(\Omega^*)^4)$,
$\epsilon_\PARAMETERPASSING{L}{lem:clean-Match}:=\bar{\epsilon}$,
$\mu_\PARAMETERPASSING{L}{lem:clean-Match}:=\beta$ and 
$d_\PARAMETERPASSING{L}{lem:clean-Match}:=\bar{d}$. Note these parameters
satisfy the numerical conditions of Lemma~\ref{lem:clean-Match}. We use the vertex sets
$Y_\PARAMETERPASSING{L}{lem:clean-Match}:=\exceptVertSplit\cup \shadowsplit$, $X_0:=V_2(\M)$,
$X_1:=V_1(\M)$, $X_2=X_3:=V(\Gexp)\colouringpI{1}$, and $V:=V(G)$. 
The partitions of $X_0$
and $X_1$ in Lemma~\ref{lem:clean-Match} are the ones induced by $\V(\M)$, and
the set $E_1$ consists of all edges from $E(\DenseSpots_\class)$ between pairs from $\M$. 
Further, set $E_2:=E(\Gcapt)$ and $E_3:=E(\Gexp)$. 
 
Let us verify the conditions of Lemma~\ref{lem:clean-Match}.
Condition~\ref{hyp:Match-Ysmall} follows from 
Definition~\ref{def:proportionalsplitting}\eqref{It:H1} and~\eqref{eq:boundShadowsplit}.
Condition~\ref{hyp:Match-edges} holds by the assumption on $\M$.
Condition~\ref{hyp:Match-deg} follows from Definition~\ref{def:proportionalsplitting}\eqref{It:H5}
by~\eqref{eq:proporcevelke}, and for $i=1$ also from the definition of $\M$.
Condition~\ref{hyp:Match-reg} holds by the definition of $\M$. Finally,
Condition~\ref{hyp:Match-bounded} follows from the properties of the sparse
decomposition~$\class$.

Lemma~\ref{lem:clean-Match} outputs four sets $X'_0,\ldots,X'_3$. By
Lemma~\ref{lem:propertyYA12}, the sets $X_0'$ and $X_1'$ witness
Preconfiguration~$\mathbf{(\heartsuit
1)}(3\eta^3/(2\cdot 10^3),\proporce{2}\left(1+\frac{\eta}{20}\right)k)$, or
$\mathbf{(\heartsuit 2)}(\proporce{2}\left(1+\frac{\eta}{20}\right)k)$.
Further, Lemma~\ref{lem:clean-Match}\eqref{conc:Match-superreg} gives that $(X_0',X_1')$ witnesses Preconfiguration
$\mathbf{(reg)}(4\bar{\epsilon},\bar{d}/4,\beta/2)$.
 It is now easy to verify that we have Configuration
$\mathbf{(\diamond6)}\big(\frac{\eta^3\rho^4}{10^{12}(\Omega^*)^4},4\bar{\epsilon},\frac
{\bar{d}}{4},\frac{\beta}{2},\frac{3\eta^3}{2\cdot 10^3},
\proporce{2}(1+\frac\eta{20})k\big)$.

This leads to
Configuration~$\mathbf{(\diamond6)}$ with parameters as claimed. Indeed, no matter whether we have {\bf(M1)} or
{\bf(M2)}, we have $4\epsilonD\ge 4\cdot \frac{10^5\epsilon'}{\eta^2}$, and
$\gamma^3\rho/(32\Omega^*)\le\gamma^2/4$, and $\eta^2\nu/(2\cdot10^4)\le\eta^2\clustersize/(8\cdot10^3
 k)\le\eta^2\epsilon'/(8\cdot10^3)\le \alphaD\rho/\Omega^*$ (for the latter recall that $\clustersize\leq \eps ' k$ by Definition~\ref{bclassdef}~\eqref{Csize}).
\end{proof}

\begin{lemma}\label{whatwegetfrom(t2)}
In case $\mathbf{(t2)}$ (of either subcase $\mathbf{(cA)}$ or subcase $\mathbf{(cB)}$) we obtain Configuration
$\mathbf{(\diamond7)}\big(\frac{\eta^3\gamma^3\rho}{10^{12}(\Omega^*)^4},\frac{\eta\gamma}{400},4\epsilonD, \frac{\gamma^3\rho}{32\Omega^*},
 \frac{\eta^2\nu}{2\cdot10^4},
\frac{3\eta^3}{2\cdot 10^3}, \proporce{2}(1+\frac\eta{20})k\big)$.
\end{lemma}
\begin{proof}
 We use
Lemma~\ref{lem:clean-Match} with the following input parameters:
$r_\PARAMETERPASSING{L}{lem:clean-Match}:=3$,
$\Omega_\PARAMETERPASSING{L}{lem:clean-Match}:=\Omega^*$,
$\gamma_\PARAMETERPASSING{L}{lem:clean-Match}:=\eta\gamma/200$,
$\eta_\PARAMETERPASSING{L}{lem:clean-Match}:=\rho/\Omega^*$,
$\delta_\PARAMETERPASSING{L}{lem:clean-Match}:=\eta^3\gamma^3\rho/(10^{12}(\Omega^*)^4)$,
$\epsilon_\PARAMETERPASSING{L}{lem:clean-Match}:=\bar{\epsilon}$,
$\mu_\PARAMETERPASSING{L}{lem:clean-Match}:=\beta$ and 
$d_\PARAMETERPASSING{L}{lem:clean-Match}:=\bar d$. We use the
 vertex sets
$Y_\PARAMETERPASSING{L}{lem:clean-Match}:=\exceptVertSplit\cup \shadowsplit$, $X_0:=V_2(\M)$,
$X_1:=V_1(\M)$, $X_2:=\smallatoms\colouringpI{1}$,
$X_3:=\colouringp{1}$, and $V:=V(G)$. The partitions of $X_0$
and $X_1$ in Lemma~\ref{lem:clean-Match} are the ones induced by $\V(\M)$, and
the set $E_1$ consists of all edges from $E(\DenseSpots_{\class})$ between pairs from $\M$. Further, set $E_2:=E(\Gcapt)$ and $E_3:=E(\GD)$. 

The conditions of Lemma~\ref{lem:clean-Match} are verified as before, let us just note that
Condition~\ref{hyp:Match-deg} follows from Definition~\ref{def:proportionalsplitting}\eqref{It:H5}
and by~\eqref{eq:proporcevelke}, and for $i=1$ from the definition of $\M$,  
while for $i=2$ it holds since $\smallatoms$ is covered by the set $\DenseSpots$
of $(\gamma k, \gamma)$-dense spots (cf.~Definition~\ref{def:avoiding}).

It is now easy to check that the output of Lemma~\ref{lem:clean-Match} are sets that witness Configuration 
$\mathbf{(\diamond7)}\big(\frac{\eta^3\gamma^3\rho}{10^{12}(\Omega^*)^4},\frac{\eta\gamma}{400},4\bar{\epsilon},\frac{\bar{d}}{4},\frac{\beta}{2},
\frac{3\eta^3}{2\cdot 10^3}, \proporce{2}(1+\frac\eta{20})k\big)$.
\end{proof}

Before proceeding with dealing with cases $\mathbf{(t3)}$, $\mathbf{(t5)}$ and
{\bf (t3--5)} we state some properties of the matching
$\bar\M:=\big(\M_A\cup\M_B\big)\colouringpI{1}$.

\begin{lemma}\label{lem:MRes}
For $V_{\mathrm{leftover}}:=V(\M_A\cup\M_B)\colouringpI{1}\setminus V(\bar\M)$ and  $Y_{\bar\M}:=\exceptVertSplit\cup\shadowsplit\cup\shadow_{\GD}(V_{\mathrm{leftover}},\frac{\eta^2 k}{1000})$, we have
\begin{enumerate}[(a)]
\item \label{eq:M-semiregN}
$\bar\M$ is a
$(\frac{400\varepsilon}{\eta},\frac
d2,\frac{\eta\pi\clustersize}{200})$-\semiregular matching absorbed by $\M_A\cup
\M_B$ and
$V(\bar\M)\subseteq\colouringp{1}$, and
\item  \label{eq:sizeYM}
$|Y_{\bar\M}|\le\frac{3000\epsilon\Omega^* n}{\eta^2}$.
\end{enumerate}
\end{lemma}
\begin{proof}
Lemma~\ref{lem:MRes}~\eqref{eq:M-semiregN} follows from Lemma~\ref{lem:RestrictionSemiregularMatching}.

Observe that from properties \eqref{It:H1} and \eqref{It:H3} of Definition~\ref{def:proportionalsplitting} we can calculate that
\begin{equation}\label{sizeofleftoverN}
|V_\text{leftover}|\leq 3\cdot k^{0.9}\cdot |\M_A\cup \M_B|+\left|\bigcup
\exceptSemSplit\cup \exceptSemSplit^*\right|\leq 3\cdot k^{0.9}\cdot \frac{n}{2\pi
\clustersize}+2\exp(-k^{0.1})\leByRef{eq:KONST}2\epsilon n.
\end{equation}
Then
\begin{align*}
|Y_{\bar\M}|&\le |\bar V|+|\shadowsplit|+
\left|\shadow_{\GD}\left(V_\text{leftover}, \frac{\eta^2 k}{1000}\right)\right|\\ \JUSTIFY{by Fact~\ref{fact:shadowbound}}&\le |\bar V| + |\shadowsplit|+|V_\text{leftover}| \frac{1000\Omega^*}{\eta^2}\\
\JUSTIFY{by~\eqref{sizeofleftoverN}, \textrm{D}\ref{def:proportionalsplitting}\eqref{It:H1}, \eqref{eq:KONST} \eqref{eq:boundShadowsplit}}&<
\frac{3000\epsilon\Omega^* n}{\eta^2}\;,
\end{align*}
as desired for Lemma~\ref{lem:MRes}\eqref{eq:sizeYM}.
\end{proof}

\begin{lemma}\label{whatwegetfrom(t3)}
In Case $\mathbf{(t3)} \mathbf{(cA)}$  we obtain Configuration
$\mathbf{(\diamond8)}\big(\frac{\eta^4\gamma^4\rho }{10^{15}
(\Omega^*)^5},\frac{\eta\gamma}{400},\frac{400\epsilon}{\eta},4\bar{\epsilon},\frac
d2,\frac{\bar{d}}{4},\frac{\eta\pi\clustersize}{200k},\frac{\beta}{2},$ $
\proporce{1}(1+\frac\eta{20})k,\proporce{2}(1+\frac\eta{20})k\big)$.
\end{lemma}

\begin{proof}
We use Lemma~\ref{lem:clean-Match} with the following input parameters:
$r_\PARAMETERPASSING{L}{lem:clean-Match}:=4$,
$\Omega_\PARAMETERPASSING{L}{lem:clean-Match}:=\Omega^*$,
$\gamma_\PARAMETERPASSING{L}{lem:clean-Match}:=\eta\gamma/200$,
$\eta_\PARAMETERPASSING{L}{lem:clean-Match}:=\rho/\Omega^*$,
$\delta_\PARAMETERPASSING{L}{lem:clean-Match}:=\eta^4\gamma^4\rho/(10^{15}(\Omega^*)^5)$,
$\epsilon_\PARAMETERPASSING{L}{lem:clean-Match}:=\bar{\epsilon}$,
$\mu_\PARAMETERPASSING{L}{lem:clean-Match}:=\beta$ and
$d_\PARAMETERPASSING{L}{lem:clean-Match}:=\bar d$. 
We use the following vertex sets
$Y_\PARAMETERPASSING{L}{lem:clean-Match}:=Y_{\bar\M}$, $X_0:=V_2(\M)$,
$X_1:=V_1(\M)$, $$X_2:=(\largevertices{\eta}{k}{G}\cap\largeintoatoms)\colouringpI{0}\setminus\big(V(\Gexp)\cup\smallatoms\cup V(\M_A\cup\M_B)\cup \WantiC\cup L_\sharp\cup \gPatoms\cup\gP_1\big)\;,$$
$X_3:=\smallatoms\colouringpI{1}$,  $X_4:=\colouringp{1}$, and $V:=V(G)$. The partitions $P^{(j)}_i$ of $X_0$ and $X_1$ in Lemma~\ref{lem:clean-Match} are the ones induced by $\V(\M)$, and
the set $E_1$ consists of all edges from $E(\DenseSpots_\class)$ between pairs from $\M$.  Further, set $E_2=E_3:=E(\Gcapt)$ and $E_4:=E(\GD)$.

Most of the conditions of Lemma~\ref{lem:clean-Match} are verified as before, let us only note the few differences. Condition~\ref{hyp:Match-Ysmall} follows from Lemma~\ref{lem:MRes}\eqref{eq:sizeYM}. Using Definition~\ref{def:proportionalsplitting}\eqref{It:H5}
and~\eqref{eq:proporcevelke}, we find that
Condition~\ref{hyp:Match-deg} for $i=2$ follows from the definition of $\largeintoatoms$, and Condition~\ref{hyp:Match-deg}
for $i=3$ holds as it is the same as Condition~\ref{hyp:Match-deg}
for $i=2$ in Lemma~\ref{whatwegetfrom(t2)}. To verify Condition~\ref{hyp:Match-deg} for $i=1$ we first observe that since we are in case  $\mathbf{(t3)} $, we have 
\begin{equation}\label{eq:usememe}
V_1(\M)\subset 
\shadow_{\Gcapt}\left((\largeintoatoms\cap\largevertices{\eta}{k}{G})\setminus
V(\M_A\cup\M_B),\frac{2\eta^2 k}{10^5}\right)\setminus(
\shadow_{\Gcapt}(V(\Gexp),\rho k)\cup \largeintoatoms)\;.
\end{equation}
Also, since we are in case $\mathbf{(cA)}$, we have
\begin{equation}\label{eq:usememezwei}
V_1(\M)\cap\gP=\emptyset\;.
\end{equation}

Thus, for each $v\in V_1(\M)$ we have, using Definition~\ref{def:proportionalsplitting}\eqref{It:H5},
\begin{align*}
\deg_{\Gcapt}(v,X_2)&\ge\proporce{0} \Big(\deg_{\Gcapt}(v,(\largevertices{\eta}{k}{G}\cap\largeintoatoms)\setminus V(\M_A\cup\M_B))\\
&~~~~~~~~~-\deg_{\Gcapt}(v,V(\Gexp)\cup\smallatoms\cup\WantiC\cup L_\sharp\cup \gPatoms\cup\gP_1) \Big)-k^{0.9}\\
\JUSTIFY{by~\eqref{eq:usememe} \& \eqref{eq:usememezwei} \& \eqref{eq:proporcevelke}}&\ge \frac{\eta}{100}\left(\frac{2\eta^2 k}{10^5}-\rho k- \frac{\rho k}{100\Omega^*}-\frac{\eta^2 k}{10^5}\right)-k^{0.9}\\
\JUSTIFY{by~\eqref{eq:KONST}}&\ge \frac{\eta \gamma k}{200}\;,
\end{align*}
which indeed verifies Condition~\ref{hyp:Match-deg} for $i=1$.

Define $\mathcal N:=\bar\M\setminus\{(X,Y)\in\bar\M\::\:X\cup Y\subset
V(\NAtom)\}$. By Lemma~\ref{lem:MRes}~\eqref{eq:M-semiregN} we have that
$\mathcal N\subseteq \bar{\mathcal M}$ is a
$(\frac{400\epsilon}{\eta},\frac{d}{2},\frac{\eta\pi
\clustersize}{200})$-\semiregular matching absorbed by $\M_A\cup\M_B$, and that
$V(\mathcal N)\subset\colouringp{1}$.

To see that the output of
Lemma~\ref{lem:clean-Match} together with the matching $\mathcal N$ leads to Configuration 
$\mathbf{(\diamond8)}\big(\frac{\eta^4\gamma^4\rho }{10^{15}
(\Omega^*)^5},\frac{\eta\gamma}{400},\frac{400\epsilon}{\eta},4\bar{\epsilon},\frac
d2,\frac{\bar{d}}{4},\frac{\eta\pi\clustersize}{200k},\frac{\beta}{2},
\proporce{1}(1+\frac\eta{20})k,\proporce{2}(1+\frac\eta{20})k\big)$ let us show that~\eqref{COND:D8:7} is satisfied (the other conditions are more easily seen to hold). 

For this, let $v\in X_2'$. We have to show that

\begin{equation}\label{todososolhos}
\deg_{\GD}(v,X'_3)+\deg_{\Gblack}(v,V(\mathcal N))
\ge  \proporce{1}\left(1+\frac{\eta}{20}\right)k.
\end{equation}

 Note that $v\not\in V(\Gexp)$, and thus $\deg_{\Gexp}(v)=0$. This allows us to
calculate as follows:

\begin{align}
\begin{split}\label{eq:OPL}
\deg_{\GD}(v,X'_3)+\deg_{\Gblack}(v,V(\mathcal N))&\ge
\deg_{\Gcapt}(v,\colouringp{1})-\deg_{\GD}(v,X_3\setminus X'_3)\\
&~~~-\deg_{\Gblack}(v,V(\NAtom))-\deg_{\Gblack}(v,V_\mathrm{leftover})\\
&~~~-\deg_{\Gblack}(v,V(G)\setminus V(\M_A\cup\M_B))\;.
\end{split}
\end{align}

We now bound the terms of the right-hand side of~\eqref{eq:OPL}.
From Definition~\ref{def:proportionalsplitting}\eqref{It:H5} we obtain that
$\deg_{\Gcapt}(v,\colouringp{1})\ge
\proporce{1}\left(\deg_{\Gcapt}(v)-\deg_G(v,\HugeVertices)\right)-k^{0.9}$.
Lemma~\ref{lem:clean-Match}\eqref{conc:Match-avoid} gives that
$\deg_{\GD}(v,X_3\setminus X'_3)\le \frac{\eta\gamma k}{400}$. As $v\not\in
\gPatoms\cup V(\M_A\cup \M_B)$, we have $\deg_{\Gblack}(v,V(\NAtom))<\gamma k$. As $v\not\in Y_{\bar\M}$ and thus $v\not\in
\shadow_{\GD}\left(V_\text{leftover}, \frac{\eta^2 k}{1000}\right)$  we have
$\deg_{\GD}(v,V_\text{leftover})\le \frac{\eta^2 k}{1000}$. Lastly, recall that
$v\not\in \gP_1\cup V(\M_A\cup \M_B)$, and consequently
$\deg_{\Gblack}(v,V(G)\setminus V(\M_A\cup\M_B))<\gamma k$. Putting these bounds together, we find that
\begin{align*}
\deg_{\GD}(v,X'_3)+\deg_{\Gblack}(v,V(\mathcal N))&\ge
\proporce{1}\left(\deg_{\Gcapt}(v)-\deg_G(v,\HugeVertices)\right)-\frac{2\eta^2 k}{1000}\\
\JUSTIFY{as $v\in \largevertices{\eta}{k}{G}\setminus
(L_\sharp\cup\WantiC)$}&\ge \proporce{1}\left(\left(1+\frac{9\eta}{10}\right)k-\frac{\eta
k}{100}\right)-\frac{ \eta^2 k}{500}\\
\JUSTIFY{by~\eqref{eq:proporcevelke} \& \eqref{eq:KONST}}&\ge  \proporce{1}\left(1+\frac{\eta}{20}\right)k\;.
\end{align*}
This proves~\eqref{todososolhos}.
\end{proof}

\begin{lemma}\label{ObtConf9}
In case {\bf(t3--5)}$\mathbf{(cB)}$ we get Configuration
$\mathbf{(\diamond9)}\big(\frac{\rho
\eta^8}{10^{27}(\Omega^*)^3},\frac
{2\eta^3}{10^3}, \proporce{1}(1+\frac{\eta}{40})k,
\proporce{2}\left(1+\frac{\eta}{20}\right)k,$ $ \frac{400\varepsilon}{\eta},
\frac{d}2, \frac{\eta\pi\clustersize}{200k},4\epsilonD,\frac{\gamma^3\rho}{32\Omega^*}, \frac{\eta^2 \nu}{2\cdot 10^4}\big)$.
\end{lemma}
\begin{proof}
Recall that by Lemma~\ref{lem:propertyYA12}
we know that $\mathcal F$, as defined in~\eqref{def:Fcover}, is an $(\M_A\cup \M_B)$-cover. We introduce another $(\M_A\cup \M_B)$-cover, 
$$\mathcal F':=\mathcal F\cup\{X\in \V(\M_B)\::\: X\subset \smallatoms\}\;.$$
By~\eqref{eq:propertyYA12cB3} and as we are in case $\mathbf{(cB)}$, we have $\maxdeg_{\Gcapt}\left(V_1(\M),\bigcup
\mathcal F\right)\le \frac{2\eta^3}{3\cdot 10^3} k$.
Furthermore, as we are in case {\bf(t3--5)}, we have
$V_1(\M)\cap \largeintoatoms=\emptyset$. Thus,

\begin{equation}\label{cl:laptopN}
\maxdeg_{\Gcapt}\left(V_1(\M),\bigcup
\mathcal F'\right)\le \frac{2\eta^3}{10^3} k.
\end{equation}

 We use
Lemma~\ref{lem:clean-Match} with the following input parameters:
$r_\PARAMETERPASSING{L}{lem:clean-Match}:=2$,
$\Omega_\PARAMETERPASSING{L}{lem:clean-Match}:=\Omega^*$,
$\gamma_\PARAMETERPASSING{L}{lem:clean-Match}:=\eta^4/10^{11}$,
$\eta_\PARAMETERPASSING{L}{lem:clean-Match}:=\rho/2\Omega^*$,
$\delta_\PARAMETERPASSING{L}{lem:clean-Match}:=\rho
\eta^8/(10^{27}(\Omega^*)^3)$,
$\epsilon_\PARAMETERPASSING{L}{lem:clean-Match}:=\bar{\epsilon}$, $\mu_\PARAMETERPASSING{L}{lem:clean-Match}:=\beta$ and
$d_\PARAMETERPASSING{L}{lem:clean-Match}:=\bar d$. 
We use the following vertex sets
$Y_\PARAMETERPASSING{L}{lem:clean-Match}:=Y_{\bar\M}$, $X_0:=V_2(\M)$,
$X_1:=V_1(\M)$, and $X_2:=V(\bar\M)\setminus \bigcup\mathcal F'\subset
\bigcup\clusters\colouringpI{1}$.
The partitions of $X_0$ and $X_1$ in Lemma~\ref{lem:clean-Match} are the ones
induced by $\V(\M)$, and the set $E_1$ consists of all edges from $E(\DenseSpots_{\class})$
between pairs from $\M$.  Further, set $E_2:=E(\GD)$.

Condition~\ref{hyp:Match-Ysmall} of Lemma~\ref{lem:clean-Match} follows from Lemma~\ref{lem:MRes}\eqref{eq:sizeYM}.
Condition~\ref{hyp:Match-edges} follows by the assumption of Lemma~\ref{ObtConf9} on the size of $V(\M)$. 
Condition~\ref{hyp:Match-reg} follows from the definition of $\mathcal M$. 
Condition~\ref{hyp:Match-bounded} holds since $V(\M)$ does not meet $\HugeVertices$.

It remains to see Condition~\ref{hyp:yel-deg}, for $i=1$.
For this, first note that from Lemma~\ref{lem:propertyYA12} we get
that
\begin{equation}\label{batman}
\mindeg_{\Gcapt}\left(V_1(\M),\Vgood\colouringpI{1}\right)\overset{\mathbf{(cB)}}\ge
\mindeg_{\Gcapt}\left(\XA\setminus(\gP\cup
\exceptVertSplit),\Vgood\colouringpI{1}\right)\ge\proporce{1}\left(1+\frac{\eta}{20}\right)k\;.
\end{equation}

From this, we calculate that
\begin{align}\label{tikitikitiki}
\notag
\mindeg_{\GD}\big(V_1(\M),V(\M_A\cup \M_B)\colouringpI{1}\big) & 
\ge
\mindeg_{\Gcapt}\big(V_1(\M),V(\M_A\cup
 \M_B)\colouringpI{1}\big)\\ 
 \nonumber & ~~~-\maxdeg_{\Gexp}\big(V_1(\M),
 V(\M_A\cup\M_B)\big)\\
 \JUSTIFY{by~\eqref{eq:defVgood} \& ~\eqref{eq:defV+}}&\ge \notag
 \mindeg_{\Gcapt}\big(V_1(\M), \Vgood\colouringpI{1}\big)\\ 
\notag  &~~~-\maxdeg_{\Gcapt}\big(V_1(\M),
 \smallatoms\big)\\ 
  &~~~\notag-\maxdeg_{\Gcapt}\big(V_1(\M),
  \largevertices{\eta}{k}{G}\setminus V(\M_A\cup\M_B)\big) \\  
  \notag &~~~-\maxdeg_{\Gcapt}\big(V_1(\M),
 V(\Gexp)\setminus V(\M_A\cup\M_B)\big)\\
 & \notag ~~~-\maxdeg_{\Gcapt}\big(V_1(\M), V(\Gexp)\cap
 V(\M_A\cup\M_B)\big)\\
  \JUSTIFY{by~\eqref{batman}, as $V_1(\M)\cap
\largeintoatoms=\emptyset$ \& $\mathbf{(cB)}$} &\ge \proporce{1}\left(1+\frac{\eta}{20}k\right)-\frac{\rho k}{100\Omega^*}\notag \\ 
 \nonumber &~~~-\maxdeg_{\Gcapt}\big(\XA\setminus \gP_3,\XA\big)\notag \\ 
 \nonumber &~~~-\maxdeg_{\Gcapt}(V_1(\M),
V(\Gexp))\\
  \JUSTIFY{by def of $\gP_3$ \& as $V_1(\M)\cap \shadow_{G}(V(\Gexp),\rho
 k)=\emptyset$ by \bf{(t3--5)}} &\ge  \proporce{1}\left(1+\frac{\eta}{20}\right)k
 -\frac{\rho k}{100\Omega^*}-\frac{\eta^3 k}{10^3}-\rho k.
\end{align}

We obtain
\begin{align}
\nonumber
\mindeg_{\GD}(V_1(\M)\setminus
Y_\PARAMETERPASSING{L}{lem:clean-Match},X_2)&
\ge  \mindeg_{\GD}\big(V_1(\M)\setminus
Y_{\bar M},V(\bar \M)\big)-\maxdeg_{\GD}(V_1(\M),
\bigcup \mathcal F')\\ 
\JUSTIFY{by def of $\bar \M$, ~\eqref{cl:laptopN}} & \ge  \nonumber
\mindeg_{\GD}\big(V_1(\M),V(\M_A\cup \M_B)\colouringpI{1}\big)\\
\nonumber &~~~-\maxdeg_{\GD}(V_1(\M)\setminus 
 Y_{\bar M}, V_{\mathrm{leftover}})
 -\frac {2\eta^3k}{10^3}\\
 \nonumber
  \JUSTIFY{by~\eqref{tikitikitiki} and by def of $Y_\PARAMETERPASSING{L}{lem:clean-Match}$} &\ge  \proporce{1}\left(1+\frac{\eta}{20}\right)k -\frac{\rho k}{100\Omega^*}-\frac{\eta^3 k}{10^3}-\rho k-\frac{\eta^2 k}{1000}-\frac {2\eta^3k}{10^3}\\
  \label{eq:densecase-degToV_1N}
&\ge
 \proporce{1}(1+\frac{\eta}{30})k\;.
\end{align}

Since the last term is greater than
$\gamma_\PARAMETERPASSING{L}{lem:clean-Match} k = \frac {\eta^4}{10^{11}}k$ by~\eqref{eq:proporcevelke}, we
see that Condition~\ref{hyp:yel-deg} of Lemma~\ref{lem:clean-Match} is satisfied.

Lemma~\ref{lem:clean-Match} outputs three non-empty sets
$X_0',X_1',X_2'$ disjoint from $Y_\PARAMETERPASSING{L}{lem:clean-Match}$,
together with $(4\bar\epsilon,\frac{\bar d}4)$-super-regular pairs
$\{Q_0^{(j)},Q_1^{(j)}\}_{j\in\mathcal Y}$ which cover $(X'_0,X'_1)$ with the
following properties.
\begin{align} 
\label{eq:MaSi}
\JUSTIFY{by
Lemma~\ref{lem:clean-Match}~\eqref{conc:Match-superreg}}&\min\left\{|Q_0^{(j)}|,|Q_1^{(j)}|\right\}\ge
\frac{\beta k}{2}\;\mbox{for each $j\in\mathcal Y$}\;,\\
\label{eq:towardsD92N}\JUSTIFY{by
Lemma~\ref{lem:clean-Match}~\eqref{conc:Match-deg}}&\mindeg_{\GD}(X_2',X_1')\ge
\delta_\PARAMETERPASSING{L}{lem:clean-Match} k\;,\\ 
\begin{split}
\label{eq:towardsD93N}\JUSTIFY{by
Lemma~\ref{lem:clean-Match}~\eqref{conc:Match-avoid} and \eqref{eq:densecase-degToV_1N}}&\mindeg_{\GD}(X_1',X_2')
\ge\proporce{1}(1+\frac{\eta}{30})k-\frac
{\eta^4k}{2\cdot 10^{11}}\\& ~~~~~~~~~~~~~~~~~~~~~~~~~\ge \proporce{1}(1+\frac{\eta}{40})k\;.
\end{split}
\end{align}

We now verify that the sets $X_0',X_1',X_2'$, 
the \semiregular matching $\mathcal
N_{\PARAMETERPASSING{D}{def:CONF9}}:=\bar\M$ together with the
$(\M_A\cup\M_B)$-cover $\mathcal F'$, and the family $\{(Q_0^{(j)},Q_0^{(j)})\}_{j\in\mathcal Y}$ satisfy all the conditions of
Configuration~$\mathbf{(\diamond9)}(\delta_{\PARAMETERPASSING{L}{lem:clean-Match}},\frac
{2\eta^3}{10^3}, \proporce{1}(1+\frac{\eta}{40})k,
\proporce{2}\left(1+\frac{\eta}{20}\right)k, \frac{400\varepsilon}{\eta},
\frac{d}2, \frac{\eta\pi\clustersize}{200k},4\epsilonD,\gamma^3\rho/32\Omega^*, \eta^2 \nu/2\cdot 10^4)$.

By Lemma~\ref{lem:propertyYA12}, since we are in case $\mathbf{(cB)}$ and by~\eqref{cl:laptopN}, the pair
$X_0',X_1'$ together with the $(\M_A\cup \M_B)$-cover $\mathcal F'$ witnesses
Preconfiguration~$\mathbf{(\heartsuit1)}(\frac
{2\eta^3}{10^3},\proporce{2}\left(1+\frac{\eta}{20}\right)k)$. By Lemma~\ref{lem:MRes}~\eqref{eq:M-semiregN},
$\bar\M$ is as required for Configuration~$\mathbf{(\diamond9)}$.

To see that $G$ is in Preconfiguration~$\mathbf{(reg)}(4\epsilonD,\gamma^3\rho/32\Omega^*, \eta^2 \nu/2\cdot 10^4)$, note that $4\bar\eps\le 4\epsilonD$ and $\bar d/4\ge  \gamma^3\rho/32\Omega^*$ (in both cases {\bf (M1)} and  {\bf (M2)}).
Further,
 Property~\eqref{COND:reg:0} follows from~\eqref{eq:MaSi} since $\beta/2\ge \eta^2 \nu/2\cdot 10^4$.
 
  Finally, by definition of
$X_2$, the set $X_2'$ is as required, with
 Property~\eqref{conf:D9-XtoV} following from~\eqref{eq:towardsD93N}, and Property~\eqref{conf:D9-VtoX} following from~\eqref{eq:towardsD92N}.
\end{proof}

We are now reaching the last lemma of this section, dealing with the last remaining case.

\begin{lemma}\label{whatwegetfrom(t5)}
In Case $\mathbf{(t5)}\mathbf{(cA)}$ we get Configuration
$\mathbf{(\diamond10)}\big(\epsilon,\frac{\gamma^2
d}{2},\pi\sqrt{\epsilon'}\nu k, \frac
{(\Omega^*)^2k}{\gamma^2},\frac{\eta}{40}\big)$.
\end{lemma}
\begin{proof}
Since we are in case  $\mathbf{(t5)}$, we have  $V(\M)\subset V(\Gblack)$. Therefore,
\begin{align}
\mindeg_{\Gblack}(V(\M),\Vgood)&\ge \notag
\mindeg_{\Gcapt}(V(\M),V_+\setminus L_\sharp)
-\maxdeg_{\Gcapt}(V(\M),\HugeVertices)\\
\notag
&~~~-\maxdeg_{\Gcapt}(V(\M),\smallatoms)
-\maxdeg_{\Gcapt}(V(\M),V(\Gexp))\\
\label{eq:3veci2N}
&\ge (1+\frac\eta{20})k,
\end{align}
where the last line follows as $V(\M)\subseteq \XA\sm\gP\subseteq\YA\setminus\WantiC$ by $\mathbf{(cA)}$ and furthermore, $V(\M)\cap (\shadow_{G}(V(\Gexp),\rho k)\cup \largeintoatoms)=\emptyset$ by  $\mathbf{(t5)}$.

Define
\begin{align}
\nonumber
\gC&:=\left\{C\setminus \big(L_\#\cup V(\M_A\cup \M_B)\cup
\WantiC\cup\gP_1\big)\::\: C\in\clusters\right\}\;,
\index{mathsymbols}{*C@$\gC$}
\\ \nonumber
\gC^-&:=\left\{C\in\gC\::\:
|C|<\sqrt{\epsilon'}\clustersize\right\}\;,
\index{mathsymbols}{*C@$\gC^-$}
\end{align}

 We have
 \begin{equation}
 \left|\bigcup \gC^-\right|\leq \sum_{C\in\gC}\sqrt{\epsilon'}|C|\le
\sqrt{\epsilon'} n\;.\label{eq:CminusMala}
 \end{equation}

Set $\V^\circ:= \V(\M_A\cup \M_B)\cup (\gC\setminus \gC^-)$ and let $G^\circ$ be
the subgraph of $G$ with vertex set $\bigcup \V^\circ$ and all edges from $E(\Gblack)$
induced by $\bigcup \V^\circ$ plus all edges of $E(\Gcapt)\setminus E(\Gexp)$
between $X$ and $Y$ for all $(X,Y)\in\M_A\cup \M_B$.
Apply Fact~\ref{fact:BigSubpairsInRegularPairs} (and recall Definition~\ref{bclassdef}~\eqref{defBC:RL}) to see that each pair of sets $X,Y\in\V^\circ$ forms an $\eps$-regular pair of density either $0$ or at least $\gamma^2d/2$ (whose edges either lie in $\Gblack$ or touch $\smallatoms$).

Next, observe that from Setting~\ref{commonsetting}\eqref{commonsetting2},
Fact~\ref{fact:sizedensespot} and Fact~\ref{fact:boundedlymanyspots}, and using Definition~\ref{bclassdef}\eqref{defBC:prepartition}, we find that for all $X\in \V^\circ$ which lie in some cluster of $\clusters$, we have
$|\bigcup \neighbour_{G^\circ}(X)|\le |\bigcup\neighbour_{\GD}(X)|\le 
\frac {\Omega^*}{\gamma}\cdot \frac {\Omega^* k}{\gamma}$.
Also, observe that for all $X\in \V^\circ$ which do not lie in some cluster of
$\clusters$, we know from Setting~\ref{commonsetting}\eqref{commonsetting3}
that $X$ does not see any edges from $E(\Gblack)$. This means that  $\bigcup
\neighbour_{G^\circ}(X)$ is contained in the partner of $X$ in $\M_A\cup M_B$
(which has size at most $\clustersize\le \epsilon'k$ by
Setting~\ref{commonsetting}\eqref{commonsetting3} and
Definition~\ref{bclassdef}\eqref{Csize}).

Thus we obtain that
\begin{equation}\label{lem:regularizedobt3}
\text{$(G^\circ,{\V}^\circ)$ is
an $(\epsilon, \frac{\gamma^2
d}2,\pi\sqrt{\epsilon'}\clustersize, \frac
{(\Omega^*)^2k}{\gamma^2})$-regularized graph.}
\end{equation}

 Define $$\mathcal
L^\circ:=\left\{X\in \V^\circ\setminus  \V(\M_A\cup \M_B)\::\: \mindeg_{G^\circ}(X)\ge
(1+\frac{\eta}2)k\right\}.$$

We claim that the following holds.
\begin{claim}\label{lem:thedesiredAandBdoexist}
There are  distinct $X_A,X_B\in \V^\circ$, with $E(G^\circ[X_A,X_B])\neq \emptyset$, such that we have $\deg_{\Gblack}(v,V(\M_A\cup\M_B)\cup\bigcup\mathcal L^\circ)\ge
(1+\frac{\eta}{40})k$ for all but at most $2\epsilon'\clustersize$ vertices $v\in X_A$, and all but at most $2\epsilon'\clustersize$ vertices $v\in X_B$. 
\end{claim}

Then, setting $\tilde G_\PARAMETERPASSING{D}{def:CONF10}:=G^\circ$, $\V_\PARAMETERPASSING{D}{def:CONF10}:=\V^\circ$, $\M_\PARAMETERPASSING{D}{def:CONF10}:=\M_A\cup \M_B$, $\mathcal L^*_\PARAMETERPASSING{D}{def:CONF10}:=\mathcal L^\circ$, $A_\PARAMETERPASSING{D}{def:CONF10}:= X_A$, and  $B_\PARAMETERPASSING{D}{def:CONF10}:= X_B$,
we have obtained
Configuration~$\mathbf{(\diamond10)}\big( \epsilon, \frac{\gamma^2
d}2,\pi\sqrt{\epsilon'}\nu k,\frac {(\Omega^*)^2k}{\gamma^2},\eta/40 \big)$.
Indeed,
 using~\eqref{lem:regularizedobt3}, and the definition of $\mathcal L^\circ$ we see that 
$( \tilde
G_\PARAMETERPASSING{D}{def:CONF10},\V_\PARAMETERPASSING{D}{def:CONF10})$,
$\M_\PARAMETERPASSING{D}{def:CONF10}$ and $\mathcal
L^*_\PARAMETERPASSING{D}{def:CONF10}$ are as desired and
fulfil~\eqref{diamond10cond3}.
  Claim~\ref{lem:thedesiredAandBdoexist} together with the fact that
  $\deg_{G^\circ}(v, V(\M_A\cup \M_B)\cup \bigcup \mathcal L^\circ)\ge
  \deg_{\Gblack}(v, V(\M_A\cup \M_B)\cup \bigcup \mathcal L^\circ)$ for all
  $v\in V(G^\circ)$ ensure that also \eqref{diamond10cond1}
  and~\eqref{diamond10cond2} hold.

\medskip

It only remains to prove Claim~\ref{lem:thedesiredAandBdoexist}.

\begin{proof}[Proof of Claim~\ref{lem:thedesiredAandBdoexist}]
 In order  to find $X_A$ and $X_B$ as in the statement of the claim, we shall exploit the matching $\M$; the relation between $\M$ and
$(G^\circ,{\V}^\circ)$, $\M_A\cup\M_B$, and $\mathcal L^\circ$ is not direct. We
proceed as follows. In  Subclaim~\ref{lem:ABobt3} we find a suitable $\M$-edge.
In case $\mathbf{(M1)}$ this $\M$-edge gives readily a suitable pair
$(A_\PARAMETERPASSING{D}{def:CONF10},B_\PARAMETERPASSING{D}{def:CONF10})$. In
case $\mathbf{(M2)}$ we have to work on the $\M$-edge to get a suitable
$\BGblack$-edge, this will be done  in Subclaim~\ref{cl:CACB}. Only then do we
find $(A_\PARAMETERPASSING{D}{def:CONF10},B_\PARAMETERPASSING{D}{def:CONF10})$.

\begin{subclaim}\label{lem:ABobt3}
There is an $\M$-edge $(A,B)$ such that
$\deg_{\Gblack}(v,V(\M_A\cup\M_B)\cup\bigcup\mathcal L^\circ)\ge
(1+\frac{\eta}{40})k+\frac{\eta k}{200}$ for at least $|A|/2$ vertices $v\in A$, and at least $|B|/2$
vertices $v\in B$.
\end{subclaim}
\begin{proof}[Proof of Subclaim~\ref{lem:ABobt3}]
Set
$S:=\shadow_{\Gblack}(\bigcup \gC^-,\frac{\eta
 k}{200})$, and note that by Fact~\ref{fact:shadowbound} we have $|S|\le |\bigcup\mathcal C^-|\cdot \frac {200\Omega^*}{\eta}$. So, setting
 $\M_S:=\{(X,Y)\in\M\::\:|(X\cup Y)\cap S|\ge |X\cup Y|/4\}$ we find that $$|V(\M_S)|\ \le \  4|S| \ \overset{\eqref{eq:CminusMala}}\le \ \frac
{800\sqrt{\epsilon'}\Omega^*n}{\eta} \ < \ \frac{\rho n}{ \Omega^*} \ \leq \
|V(\M)|,$$ where the last inequality holds by the assumption of
Lemma~\ref{whatwegetfrom(t5)}. Consequently, $\M\neq\M_S$.

Let $(A,B)\in\M\setminus \M_S$. We will show that $(A,B)$ satisfies the
requirements of the subclaim. 
 We start by proving that
 \begin{equation}\label{fact:VplusCapG0}
V_{+}\cap V(G^\circ)\sm (V(\M_A\cup \M_B)\cup \bigcup \mathcal L^\circ)\subseteq  V(\Gexp)\cup
(\largeintoatoms\cap \largevertices{\eta}{k}{G}).
 \end{equation}
Indeed, observe that by~\eqref{defV+eq},
 \begin{align*}
 V_{+}\cap V(G^\circ) &\subseteq V(\M_A\cup \M_B)\cup V(\Gexp) \cup \big(
 \largevertices{\eta}{k}{G}\sm (L_\sharp \cup\WantiC\cup \gP_1)\big)\\
 &\subseteq V(\M_A\cup \M_B)\cup V(\Gexp) \cup 
 \big(\largevertices{\frac
 {9\eta}{10}}{k}{\Gcapt}
 \setminus (\WantiC\cup \gP_1)\big)\;.
 \end{align*}
 
 So, in order to show~\eqref{fact:VplusCapG0}, it suffices to see that 
  for each $X\in \V^\circ\setminus \V(\M_A\cup \M_B)$ with $X\subseteq
 \largevertices{\frac
 {9\eta}{10}}{k}{\Gcapt}
 \setminus (\WantiC\cup \gP_1\cup V(\Gexp)\cup
 \largeintoatoms)$ we have $X\in \mathcal L^\circ$. So assume $X$ is as above. Let $v\in X$. We calculate
 \begin{align*}
 \deg_{\Gblack}(v, V(G^\circ))&\ge \deg_{\Gblack}(v, V(\M_A\cup \M_B))\\
 \JUSTIFY{$v\notin V(\Gexp)$}&\ge \left(1+\frac{9\eta}{10}\right)k-\deg_G(v,\HugeVertices)-\deg_{\GD}(v,
 \smallatoms)\\ 
 & \ -\deg_{\Gblack}(v, \bigcup \clusters\setminus V(\M_A\cup \M_B))\\
 \JUSTIFY{$v\notin \WantiC\cup\largeintoatoms\cup \gP_1\cup V(\M_A\cup
 \M_B)$}&\ge \left(1+\frac{9\eta}{10}\right)k -\frac {\eta k}{100}-\frac{\rho
 k}{100\Omega^*} -\gamma k\\
&\ge \left(1+\frac{\eta}{20}\right)k\;.
 \end{align*}
 We deduce that $X\in\mathcal L^\circ$, completing the proof of~\eqref{fact:VplusCapG0}.
 
 Next, observe that by the definition of $\mathcal C$, we have 
 \begin{align}
  V_+\cap V(G^\circ) & \notag \supseteq \Vgood\cap V(G^\circ)  \\ \notag & \supseteq \Vgood\sm \big( \Vgood\sm V(G^\circ)\big)\\  & \supseteq \Vgood\sm (\WantiC\cup \gP_1\cup \bigcup
\gC^- \cup \smallatoms\cup V(\Gexp)).\label{vorabrechnung}
 \end{align}

We are now ready to prove Subclaim~\ref{lem:ABobt3}.
For each vertex $v\in A\setminus S$, we have
\begin{align*}
\deg_{\Gblack}\left(v,V(\M_A\cup \M_B)\cup \bigcup \mathcal L^\circ\right)
&\ge  \deg_{\Gblack}(v, V_+\cap V(G^\circ))\\
 &~~~-\deg_{\Gblack}\left(v,(V_+\cap V(G^\circ))\setminus (V(\M_A\cup \M_B)\cup
 \mathcal L^\circ)\right)\\
\JUSTIFY{by~\eqref{vorabrechnung}, ~\eqref{fact:VplusCapG0}} 
&\ge
\deg_{\Gblack}(v, \Vgood)
-\deg_{\Gblack}(v, \WantiC\cup \gP_1\cup \bigcup
\gC^-)\\
&~~~-\deg_{\Gblack}(v,\smallatoms)-2\deg_{\Gblack}(v,V(\Gexp))\\
&~~~-\deg_{\Gblack}\Big(v,(\largeintoatoms\cap \largevertices{\eta}{k}{G})\setminus V(\M_A\cup\M_B)\Big)\\
 \JUSTIFY{by ~\eqref{eq:3veci2N}, as $v\not\in  S\cup\gP$, by $\mathbf{(t5)}$} 
 &\ge  \left(1+\frac{\eta}{20}\right)k-
\frac{\eta^2 k}{10^5}-\frac{\eta k}{200}-\frac{\rho k}{100\Omega^*}-2\rho k
-\frac{2\eta^2 k}{10^5}
\\
&>
  (1+\frac{\eta}{40})k+\frac {\eta k}{200}\;,
\end{align*}
where for the second to last inequality we used the abreviation  `by
$\mathbf{(t5)}$' to indicate that this case implies that  $v\notin
\shadow_{\Gcapt}(V(\Gexp), \rho k)\cup \shadow_{\Gcapt}((\largeintoatoms\cap
\largevertices{\eta}{k}{G})\setminus V(\M_A\cup M_B), \frac{2\eta^2k}{10^5})$.
As $|A\setminus S|\ge |A|/2$, we note that the set $A$ satisfies the requirements of the claim.

The same calculations hold for $B$. This finishes the proof of
Subclaim~\ref{lem:ABobt3}.
\end{proof}

The next auxiliary subclaim is needed in our proof of
Claim~\ref{lem:thedesiredAandBdoexist} in case {\bf(M2)}.
\begin{subclaim}\label{cl:CACB}
Suppose that case {\bf(M2)} occurs. Then there exists an edge $C_AC_B\in E(\BGblack)$ such that $\deg_{\Gblack}(v,V(\M_A\cup\M_B)\cup\bigcup\mathcal L^\circ)\ge
(1+\frac{\eta}{40})k+\frac{\eta k}{400}$ for all but at most
$2\epsilon'\clustersize$ vertices $v\in C_A$, and all but at most
$2\epsilon'\clustersize$ vertices $v\in C_B$. Moreover, there exist $A,B\in
\V(\M)$ such that $|C_A\cap A|>\sqrt{\epsilon'}\clustersize$ and $|C_B\cap
B|>\sqrt{\epsilon'}\clustersize$.
\end{subclaim}
\begin{proof}[Proof of Subclaim~\ref{cl:CACB}]
Let $(A,B)\in\M$ be given as in Subclaim~\ref{lem:ABobt3}.  Let $P_A\subset A$,
and $P_B\subset B$ be the vertices which fail the assertion of
Subclaim~\ref{lem:ABobt3}. Note that with this notation,
Subclaim~\ref{lem:ABobt3} states that
\begin{equation}\label{apalache}
|A\sm P_A|\geq |A|/2.
\end{equation}

Call a cluster $C\in \clusters$ \emph{$A$-negligible} if $|C\cap (A\setminus
P_A)|\le \frac {\gamma^3\clustersize}{16\Omega^*k}|A|$. Let $R_A$ be the union
of all $A$-negligible clusters. 

Recall that $(A,B)$ is entirely contained in one dense spot from
$(U,W;F)\in \DenseSpots_\class$ (cf.\ {\bf(M2)}). So by
Fact~\ref{fact:sizedensespot}, and since the spots in $\DenseSpots_\class$ are $(\frac{\gamma^3
k}4,\frac{\gamma^3 k}4)$-dense, we know that $\max\{|U|,|W|\}\leq
\frac{4\Omega^* k}{\gamma^3}$. In particular, there are at most $\frac{4\Omega^* k}{\gamma^3\clustersize}$ $A$-negligible clusters which intersect $A\cap R_A$.

As these clusters are all disjoint, we find that
 $$|(A\cap R_A)\sm P_A|\le \frac{4\Omega^* k}{\gamma^3\clustersize}\cdot  |C\cap (A\setminus
P_A)|\le \frac{|A|}{4}.$$ 
This gives 
$$|A\setminus (P_A\cup R_A)|\geBy{\eqref{apalache}} \frac{|A|}2-|(A\cap R_A)\sm P_A|\ge \frac{|A|}4\;.$$ 

Similarly, we
can introduce the notion of $B$-negligible clusters, and the set $R_B$, and get
$|(B\cap R_B)\sm P_B|\le \frac{|B|}{4}$ and $|B\setminus (P_B\cup R_B)|\ge \frac{|B|}{4}$. 

By the regularity of the pair $(A,B)$ there exists at least one edge $ab\in E\big(G^*[A\setminus (P_A\cup R_A),B\setminus (P_B\cup R_B)]\big)$, where $a\in A$, $b\in B$, and $G^*$ is the graph formed by the edges of $\DenseSpots_\class$. As $V(\M)\subset V(\Gblack)$ by the assumption of case {\bf(t5)}, we have that $ab\in E(\Gblack)$. Let $C_A,C_B\in\clusters$ be the clusters containing $a$ and $b$, respectively. Note that $C_AC_B\in E(\BGblack)$.

Now as $a\notin R_A$, also $C_A$ is disjoint from $R_A$, and thus $$|C_A\cap (A\setminus
 P_A)|>\frac {\gamma^3\clustersize}{16\Omega^*k}\cdot \frac {\alphaD\rho
 k}{\Omega^*} > \sqrt{\epsilon'}\clustersize\;.$$
 This proves the ``moreover'' part of the claim for $C_A$.
 So there are at least $2\epsilon'\clustersize$ vertices $v$ in $C_A$ with $\deg_{\Gblack}(v,V(\M_A\cup\M_B)\cup\bigcup\mathcal L^\circ)\ge (1+\frac\eta{40})k+\frac{\eta k}{200}$ (by the definition of $P_A$). By Lemma~\ref{lem:degreeIntoManyPairs}, and using Facts~\ref{fact:sizedensespot} and~\ref{fact:boundedlymanyspots}, 
 we thus have that $\deg_{\Gblack}(v,V(\M_A\cup\M_B)\cup\bigcup\mathcal L^\circ)\ge (1+\frac\eta{40})k+\frac{\eta k}{400}$ for all but at most $2\epsilon'\clustersize$ vertices $v$ of $C_A$.
The same calculations hold for $C_B$. 
\end{proof}

In the remainder of the proof of Claim~\ref{lem:thedesiredAandBdoexist} we have
to distiguish between cases {\bf(M1)} and {\bf(M2)}.

Let us first consider the case {\bf(M2)}.
Let $C_A,C_B\in\clusters$ and $A,B\in\V(\M)$ be given by Subclaim~\ref{cl:CACB}.   
We have $|C_A\setminus (\WantiC
\cup
L_\sharp \cup \gP_1)|> \sqrt{\epsilon'}|C_A|$ by Subclaim~\ref{cl:CACB} and by
the definition of $\M$ and the definition of $\gP$. 
Thus, $C_A\cap V(G^\circ)$ is non-empty. Let $X_A\in\V^\circ$ be an arbitrary
set in $C_A$. Similarly, we obtain a set $X_B\in\V^\circ$, $X_B\subset C_B$. The
claimed properties of the pair $(X_A,X_B)$ follow directly from
Subclaim~\ref{cl:CACB}.

It remains to treat the case {\bf(M1)}.
Let $(A,B)$ be from Subclaim~\ref{lem:ABobt3}. Let $(X_A,X_B)\in\Mgood$ be such
that $X_A\supset A$ and $X_B\supset B$. Claim~\ref{lem:ABobt3} asserts that at least $$\frac{|A|}2\geBy{\bf{(M1)}}\frac{\eta^2\clustersize}{2\cdot 10^4} > 2\epsilon'\clustersize$$ vertices of $A$ have large degree (in $\Gblack$) into the set $V(\M_A\cup\M_B)\cup\bigcup\mathcal L^\circ$. Therefore, by Lemma~\ref{lem:degreeIntoManyPairs}, 
 $X_A$ and $X_B$ satisfy the assertion of the
Claim.

This proves Claim~\ref{lem:thedesiredAandBdoexist}. 
 \end{proof}
Recall that Claim~\ref{lem:thedesiredAandBdoexist} was the only missing piece in the proof of Lemma~\ref{whatwegetfrom(t5)}. The proof of Lemma~\ref{whatwegetfrom(t5)} is thus complete.
\end{proof}

The proof of  Lemma~\ref{lem:ConfWhenMatching} follows by putting together Lemmas~\ref{whatwegetfrom(t1)}, \ref{whatwegetfrom(t2)}, \ref{whatwegetfrom(t3)}, \ref{ObtConf9}, and \ref{whatwegetfrom(t5)}.

\section{Acknowledgements}\label{sec:ACKN}
The work on this project lasted from the beginning of 2008 until 2014
and we are very grateful to the following institutions and funding bodies for
their support. 

\smallskip

During the work on this paper Hladk\'y was also affiliated with Zentrum
Mathematik, TU Munich and Department of Computer Science, University of Warwick. Hladk\'y was funded by a BAYHOST fellowship, a DAAD fellowship, 
Charles University grant GAUK~202-10/258009, EPSRC award EP/D063191/1, and by an EPSRC Postdoctoral Fellowship during the work on the project.

Koml\'os and Szemer\'edi acknowledge the support of NSF grant
DMS-0902241.

Piguet was also affiliated with the Institute of Theoretical Computer Science, Charles University in Prague, Zentrum
Mathematik, TU Munich, the Department of Computer Science and DIMAP,
University of Warwick, and the school of mathematics, University of Birmingham. The work leading
to this invention was supported by the European Regional Development Fund (ERDF), project ``NTIS --- New Technologies for Information Society'', European Centre of Excellence, CZ.1.05/1.1.00/02.0090.
The research leading to these results has received funding from the European Union Seventh
Framework Programme (FP7/2007-2013) under grant agreement no. PIEF-GA-2009-253925.
Piguet acknowledges the support of the Marie Curie fellowship FIST,
DFG grant TA 309/2-1, a DAAD fellowship,
Czech Ministry of
Education project 1M0545,  EPSRC award EP/D063191/1,
and  the support of the EPSRC
Additional Sponsorship, with a grant reference of EP/J501414/1 which facilitated her to
travel with her young child and so she could continue to collaborate closely
with her coauthors on this project. This grant was also used to host Stein in
Birmingham.

Stein was affiliated with the Institute of Mathematics and Statistics, University of S\~ao Paulo, and the Centre for Mathematical Modeling, University of Chile. She was
supported by a FAPESP fellowship, and by FAPESP travel grant  PQ-EX 2008/50338-0, also
CMM-Basal,  FONDECYT grants 11090141 and  1140766. She also received funding by EPSRC Additional Sponsorship EP/J501414/1.

We enjoyed the hospitality of the School of Mathematics of University of Birmingham, Center for Mathematical Modeling, University of Chile, Alfr\'ed R\'enyi Institute of Mathematics of the Hungarian Academy of Sciences and Charles University, Prague, during our long term visits.

The yet unpublished work of Ajtai, Koml\'os, Simonovits, and Szemer\'edi on the Erd\H{o}s--S\'os Conjecture was the starting point for our project, and our solution crucially relies on the methods developed for the Erd\H{o}s-S\'os Conjecture. Hladk\'y, Piguet, and Stein are very grateful to the former group for explaining them those techniques.

\medskip
A doctoral thesis entitled \emph{Structural graph theory} submitted by Hladk\'y in September 2012 under the supervision of Daniel Kr\'al at~Charles University in~Prague is based on the series of the papers~\cite{cite:LKS-cut0,cite:LKS-cut1, cite:LKS-cut2, cite:LKS-cut3}. The texts of the two works overlap greatly. We are grateful to PhD committee members Peter Keevash and Michael Krivelevich. Their valuable comments are reflected in the series. 

\bigskip
We thank the referees for their very detailed remarks.

\bigskip
The contents of this publication reflects only the authors' views and not necessarily the views of the European Commission of the European Union.

\printindex{mathsymbols}{Symbol index}
\printindex{general}{General index}

\newpage
\addcontentsline{toc}{section}{Bibliography}
\bibliographystyle{alpha}
\bibliography{bibl}
\end{document}